\documentclass[11pt,a4paper,twoside]{article}

\usepackage{comment}  
\usepackage{bbm}
\usepackage{amssymb}
\usepackage{amsfonts}
\usepackage{geometry}
\usepackage{color}
\usepackage{mathtools}
\usepackage{amsmath,amsthm,amssymb,amsfonts,enumitem}
\usepackage{CJK}
\usepackage{latexsym}
\usepackage{graphicx}
\usepackage{pgfpages}
\usepackage{mathrsfs}
\usepackage{ytableau}
\usepackage{enumitem}  
\usepackage{cite}
\usepackage[hyphens]{url}
\usepackage[bookmarks=false,hidelinks]{hyperref}
\usepackage[affil-it]{authblk}
\allowdisplaybreaks
\usepackage{fmtcount}

\usepackage{empheq}

\newcommand{\RN}[1]{\uppercase\expandafter{\romannumeral#1}}

\newtheorem{definition}{Definition}[section]
\newtheorem{theorem}[definition]{Theorem}
\newtheorem{lemma}[definition]{Lemma}
\newtheorem{proposition}[definition]{Proposition}
\newtheorem{corollary}[definition]{Corollary}
\newtheorem{remark}[definition]{Remark}
\newtheorem{condition}{Assumption}[section]

\numberwithin{equation}{section}

\def\m1{\mathbbm{1}}
\def\m1{\mathbbm{1}}
\def\be{\mathbb{E}}
\def\br{\mathbb{R}}

\def\bp{\mathbb{P}}
\def\bn{\mathbb{N}}

\def\bz{\mathbb{Z}}

\def\cb{\mathcal{B}}
\def\cc{\mathcal{C}}

\def\cf{\mathcal{F}}

\def\cm{\mathcal{M}}

\def\co{\mathcal{O}}
\def\cp{\mathcal{P}}

\def\cw{\mathcal{W}}
\def\d{\mathrm{d}}
\def\re{\mathrm{e}}

\def\vp{\varphi}
\def\ve{\varepsilon}

\def\d{\mathrm{d}}

\def\vp{\varphi}
\def\ve{\varepsilon}
\def\e{\mathrm{e}}

\def\bone{{\mathbbm{1}}}

\def\bone{{\mathbbm{1}}}

\begin{document}
	\title{Periodic measures for time-inhomogeneous stochastic differential equations driven by multiplicative $\alpha$-stable processes: Existence, uniqueness and numerical approximation}
	\author[a]{Jiaquan Lu}
    \author[b,c]{Huaizhong Zhao}
    \affil[a]{Research Center for Mathematics and Interdisciplinary Sciences, Shandong University, Qingdao 266237, China}
    \affil[b]{Department of Mathematical Sciences, Durham University, DH1 3LE, UK}
    \affil[c]{School of Mathematics, Shandong University, Jinan 250100, China}
    \affil[ ]{jiaquan.lu@mail.sdu.edu.cn, huaizhong.zhao@durham.ac.uk}
    \date{ }
	
\maketitle
\vspace{-1cm}
\begin{abstract}
	
	In this paper, we study the periodic measures within the context of time-inhomogeneous Markov processes. By extending the Hairer--Mattingly method to the time-inhomogeneous case, we first establish the existence, uniqueness and exponential ergodicity of periodic measures for general time-periodic Markov processes. These results are subsequently applied to time-inhomogeneous stochastic differential equations under the assumption that the drift is weakly dissipative and locally Lipschitz. We further employ the truncated Euler--Maruyama scheme to study numerical approximations of the periodic measures and derive the convergence rate between the exact and discrete periodic measures. These numerical results are new even in the time-homogeneous case under the assumptions considered
in this paper.
    \medskip
    
	\noindent
    {\bf MSC2020 subject classifications:} Primary 60H10, 65C30; secondary 37A50, 60J75.
\end{abstract}

\tableofcontents

\section{Introduction}
Many systems are simultaneously subject to periodic forcing and random perturbations. Seasonal temperature variation and transitions between glacial and interglacial periods are familiar examples \cite{BB07,BPSV83}. The interplay between periodicity and randomness gives rise to random periodicity. A fundamental question is how to characterize the long-time behaviour of such systems. For time-homogeneous stochastic systems, stationary solutions and invariant measures provide pathwise and distributional descriptions of long-time behaviour, respectively \cite{KB37,D19,DZ96,MT12,K12}. For stochastic systems subject to periodic forcing, random periodic solutions provide a corresponding pathwise description. This notion was introduced within the framework of random dynamical systems in \cite{ZZ09}, with periodicity formulated through a corresponding shift of the underlying noise. At the distributional level, the asymptotic law may depend on the phase of the periodic forcing, leading naturally to the notion of a periodic measure. The theory of periodic measures was developed in \cite{FZ20}, where their connection with random periodic solutions was also established.

Over the past two decades, the study of random periodic solutions, periodic measures, and their ergodic theories has seen substantial progress in several directions, including quasi-periodic dynamics, bifurcation theory, stochastic resonance, random horseshoes, stochastic oscillations, the averaging principle, ergodic control, ARMA models and in modelling real life problems such as climate dynamics; see, for example, \cite{BLMZ26,SSZ19,LLT26,LZ26,FQZ23,FQZ26,LL25,MZ26,CLRS17,EK21,U21,CSM11}. Despite this progress, comparatively little is known about the ergodicity and numerical approximation of periodic measures for stochastic systems driven by $\alpha$-stable noise, particularly when the noise is multiplicative and the drift is only locally Lipschitz and weakly dissipative.

An important motivation for considering $\alpha$-stable noise is that many systems are subject to sudden shocks or extreme events that may not be adequately captured by Gaussian forcing. Such effects arise, for example, in finance, climate dynamics, and biology. In contrast to Brownian motion, an $\alpha$-stable process allows jumps and has power-law tails, making it a natural model for intermittent and heavy-tailed fluctuations. This modeling advantage comes with substantial analytical and numerical difficulties: the heavy tails restrict the range of admissible moment orders and prevent the direct application of many techniques developed for the Brownian case. Moreover, periodic measures rarely admit explicit representations, so their numerical approximation is essential in applications. These considerations motivate the simultaneous study of ergodicity and numerical approximation undertaken in this paper.

In this paper, we consider the following time-inhomogeneous stochastic differential equation (SDE) driven by multiplicative $\alpha$-stable noise:
\begin{equation}\label{SDEs}
    \begin{cases}
        \d X_t^{s,x}=b(t,X_t^{s,x})\d t+\sigma(t,X_{t-}^{s,x})\d L_t, & t\geq s,\\
        X_s^{s,x}=x,
    \end{cases}
\end{equation}
where $b:\br\times\br^d\to\br^d$ and $\sigma:\br\times\br^d\to\br^{d\times d}$ are continuous functions, and $L$ is a two-sided rotationally invariant $\mathbb R^d$-valued $\alpha$-stable process with $\alpha\in(1,2)$. The coefficients are periodic in time with a common period $\tau$. The drift is assumed to be locally Lipschitz, of polynomial growth, and weakly dissipative, whereas the diffusion coefficient is globally Lipschitz, bounded, and uniformly non-degenerate. These assumptions allow nonlinear drifts beyond the globally Lipschitz setting while retaining sufficient stability to investigate the long-time dynamics.

Our aim is to answer two related questions. First, does SDE \eqref{SDEs} admit a unique periodic measure, and, if so, do the associated transition probabilities converge to it at an exponential rate? Second, can the periodic measure be approximated by an explicit time-discretization scheme with a quantitative error bound? The combination of time inhomogeneity, multiplicative $\alpha$-stable noise, and the polynomial growth of the drift makes both questions nontrivial.

The first part of the paper develops an abstract ergodic criterion for time-periodic Markov processes. Classical Harris theory \cite{Ha56} and its quantitative refinement by \cite{HM11} are formulated primarily for time-homogeneous Markov semigroups. We extend the Hairer--Mattingly method to two-parameter periodic transition kernels. Under a one-period Lyapunov condition and a local Doeblin condition, both uniform in the initial time, we establish the existence and uniqueness of a periodic measure and prove exponential convergence to it in a weighted total variation distance. The result gives both forward convergence, as the terminal time tends to $+\infty$, and pullback convergence, as the initial time tends to $-\infty$. For the Lyapunov function used in the SDE application, the resulting weighted total variation estimate also yields convergence in total variation and in the $1$-Wasserstein distance.

We then apply this general criterion to the time-periodic Markov process generated by SDE \eqref{SDEs}.The main difficulty lies in verifying the local Doeblin condition in the presence of a drift with polynomial growth. In the Brownian case, the exponential form of Gaussian heat kernel estimates plays a key role in obtaining uniform positive lower bounds on compact sets \cite{FQZ23}. In contrast, the heat kernels of $\alpha$-stable processes have polynomial tails. To verify this condition, we proceed in two steps. First, we truncate the drift and apply the heat kernel estimates of \cite{CZ18} to the resulting equation. Motivated by the
heat kernel estimates for killed Brownian motions \cite{H56,CZ12}, we establish
a strictly positive lower bound on compact subsets of a bounded domain for the heat kernel of the truncated process killed upon leaving that domain. Second, we use a localization argument based on stopping times to transfer this lower bound to the original SDE. Together with the Lyapunov estimate obtained from weak dissipativity, this verifies the abstract criterion and yields the existence and uniqueness of a periodic measure, together with exponential convergence in a weighted total variation distance.

The second part of the paper concerns numerical approximation of the periodic measure. Explicit Euler--Maruyama schemes for SDEs driven by L\'evy processes have been extensively studied over finite time horizons under various regularity and integrability assumptions; see, among others, \cite{PT97,HK05,BP07,DKS16,PT17,AKP25}. However, when the drift grows superlinearly, the standard Euler--Maruyama scheme may exhibit moment divergence as the step size tends to zero, even over a fixed finite time horizon \cite{WXX25}. Moreover, since $\alpha$-stable processes have finite moments only of orders strictly below $\alpha$, many numerical arguments for SDEs driven by Brownian motion that rely on higher-order moment estimates cannot be applied directly \cite{MSH02,FLLZ25,FG20,HMS02,S16,LMY19}. For this reason, we consider the following truncated Euler--Maruyama scheme:
\begin{equation}\label{TEM}
    \begin{cases}
        \hat{X}_{s+k\Delta}=\pi_\Delta(\tilde{X}_{s+k\Delta}),\\
        \tilde{X}_{s+(k+1)\Delta}=\hat{X}_{s+k\Delta}+b(s+k\Delta,\hat{X}_{s+k\Delta})\Delta
        +\sigma(s+k\Delta,\hat{X}_{s+k\Delta})\Delta L_{s+k\Delta},
    \end{cases}
\end{equation}
where $k\in\bn_0$, $\tilde{X}_s=x\in\br^d$, $\Delta L_{s+k\Delta}=L_{s+(k+1)\Delta}-L_{s+k\Delta}$, and $\pi_\Delta$ is a suitable projection. Fix $\Delta=\tau/N$, the scheme defines a time-periodic Markov chain with period $N$. The numerical analysis requires additional assumptions beyond those used to establish the exact periodic measure of SDE \eqref{SDEs}. Specifically, the drift satisfies a global one-sided Lipschitz condition and a polynomial Lipschitz condition in the spatial variable, while both coefficients are $1/\alpha$-H\"older continuous in time, uniformly in the spatial variable.

For the truncated scheme, we establish discrete Lyapunov conditon and local Doeblin condition uniformly with respect to the step size, and derive the contraction rate of the one-step Markovian semigroup. This contraction property allows us to prove that the transition probabilities of the numerical scheme converge exponentially to a unique periodic measure in the weighted total variation distance. To compare the discrete and exact periodic measure, we establish a finite-horizon strong error estimate. The novelty of our analysis lies in the careful choice of the truncation radius and a refined treatment of the small-jump and large-jump contributions. Combining this estimate with a perturbation theory introduced in \cite{SS00}, we obtain an explicit convergence rate in the $1$-Wasserstein distance between the numerical and exact periodic measures. 

The main contribution of this paper is the numerical approximation of periodic measures and its quantitative error analysis for SDEs driven by multiplicative $\alpha$-stable noise with locally Lipschitz, polynomially growing drift. To the best of our knowledge, these numerical results are novel even in the time-homogeneous case under the assumptions considered here. As a foundation for the numerical analysis, we also establish a quantitative Harris criterion for time-periodic Markov processes, as well as the existence, uniqueness, and exponential ergodicity of periodic measure for the time-periodic SDE with a drift of polynomial growth. Related exponential ergodicity results for time-homogeneous SDEs have been obtained in \cite{LW20,ZZ23}. For time-homogeneous SDEs driven by $\alpha$-stable processes, the approximation of invariant measures has been studied in the $1$-Wasserstein distance for additive noise \cite{CDSX23} and in a Wasserstein distance induced by a bounded cost for multiplicative noise \cite{CXZZ26}. It is worth pointing out that the drift assumptions in both works imply at most linear growth and therefore do not cover the superlinear polynomial drifts allowed in the present paper.

The remainder of the paper is organized as follows. Section \ref{sec:pm for prcess} develops the abstract criterion for periodic measures of time-inhomogeneous Markov processes. Section \ref{sec:pm for sdes} establishes strictly positive lower bounds for the heat kernels of killed c\`adl\`ag processes and applies the criterion to prove exponential ergodicity for SDE \eqref{SDEs}. Section \ref{sec:dpm} studies the discrete periodic measure generated by scheme \eqref{TEM}. Section \ref{sec:finite time} establishes finite-time error estimates, and Section \ref{sec:error analysis} derives the convergence rate between the numerical and exact periodic measures.

\section{Periodic measures for time-inhomogeneous Markov processes}\label{sec:pm for prcess} 
Throughout this section, let $\{X_t,t\in\br\}$ be an $\br^d$-valued time-periodic Markov process with period $\tau$, defined on a complete probability space $(\Omega,\cf,\bp)$. The two-parameter Markovian transition kernel of the process is defined for all $s\leq t$, $x\in\br^d$ and $A\in\mathcal{B}(\br^d)$ by
\[P(s,x,t,A)=\bp(X_t^{s,x}\in A),\]
which satisfies the periodicity property
\[P(s,x,t,A)=P(s+\tau,x,t+\tau,A).\]
Denote by $\cp(\br^d)$ the set of all probability measures on the measurable space $(\br^d,\cb(\br^d))$.

\begin{condition}\label{cond:Markov process}
    \begin{description}
        \item[(\romannumeral1)] There exist a measurable function $V:\br^d\to[0,\infty)$ and constants $\theta\in(0,1)$ and $K> 0$ such that for all $s\in\br$ and $x\in \br^d$,
        \begin{equation}\label{eq:Lyap function}
            P(s,s+\tau)V(x)\leq \theta V(x)+K.
        \end{equation}
        \item[(\romannumeral2)] There exist constants $\eta\in(0,1)$, $\mathcal{R}>\frac{2K}{1-\theta}$ and a probability measure $\varphi\in\cp(\br^d)$ such that $\cc=\{x\in \br^d:V(x)\leq \mathcal{R}\}$ is 
        a nonempty compact set, and for all $s\in\br$ and $A\in\cb(\br^d)$,
        \begin{equation}\label{eq:local Doeblin}
            \inf_{x\in\cc}P(s,x,s+\tau,A)\geq \eta\varphi(A).
        \end{equation}
    \end{description}
\end{condition}

\begin{condition}\label{cond:Markov process2}
Let $V$ be the measurable function and $K$ be the constant appearing in Assumption \ref{cond:Markov process}. There exists a constant $\vartheta>0$ such that for all $s\leq t<s+\tau$ and $x\in \br^d$,
    \begin{equation}\label{eq:PV leq V+K}
         P(s,t)V(x)\leq \vartheta V(x)+K.
    \end{equation}
\end{condition}

For any $\beta>0$ and two probability measures $\mu_1$ and $\mu_2$, we define the weighted total variation distance
\begin{equation}\label{distance}
    \rho_\beta(\mu_1,\mu_2)=\int_{\br^d}(1+\beta V(x))|\mu_1-\mu_2|(\d x),
\end{equation}
where $|\mu_1-\mu_2|$ is the total variation of the signed measure $\mu_1-\mu_2$.

\begin{remark}
    For any $\beta>0$ and $\mu_1,\mu_2\in\cp(\br^d)$. Note that
    \[\Vert\mu_1-\mu_2\Vert_{\mathrm{TV}}\leq \rho_\beta(\mu_1,\mu_2),\]
    then convergence in $\rho_\beta$ implies convergence in total variation. In addition, suppose that $\beta>0$ and $V$ dominates the linear growth. Then, for any $\mu_1,\mu_2\in\cp_V(\br^d)$, where
    \[\cp_V(\br^d):=\left\{\mu\in\cp(\br^d):\int_{\br^d}V(x)\mu(\d x)<\infty\right\},\]
    the Kantorovich--Rubinstein duality formula gives
    \[\cw_1(\mu_1,\mu_2)\leq\int_{\br^d}|x|\,|\mu_1-\mu_2|(\d x)\leq C\rho_\beta(\mu_1,\mu_2).\]
    Therefore, for every fixed $\beta>0$, convergence in $\rho_\beta$ on $\cp_V(\br^d)$ implies convergence in both total variation and the $1$-Wasserstein distance.
\end{remark}

\begin{lemma}\label{lem:||P(s,s+nT)||leq eta^n}
    Suppose that Assumption \ref{cond:Markov process} holds, then for any $s\in\br$, $n\in\bn$ and $\mu_1,\mu_2\in\cp(\br^d)$, we have
    \begin{equation}\label{eq:rho<alpha n}
        \rho_\beta(P^*(s,s+n\tau)\mu_1,P^*(s,s+n\tau)\mu_2)\leq\bar{\eta}^n\rho_\beta(\mu_1,\mu_2),
    \end{equation}
    where $\beta=\frac{\eta}{2K}$ and $\bar{\eta}=\max\left\{1-\frac{\eta}{2},\frac{2+\beta(\theta \mathcal{R}+2K)}{2+\beta \mathcal{R}}\right\}\in(0,1)$. If Assumption \ref{cond:Markov process2} also holds, then for any $s\leq t<s+\tau$, we have
    \begin{equation}\label{eq:P(s,t)mu leq Cmu}
        \rho_\beta(P^*(s,t)\mu_1,P^*(s,t)\mu_2)\leq((1+\beta K)\vee\vartheta)\rho_\beta(\mu_1,\mu_2).
    \end{equation}
\end{lemma}

\begin{proof}
    It follows from \cite[Theorem 1.3]{HM11} and Assumption \ref{cond:Markov process} that
    \begin{equation*}\label{eq:P<alpha}
       \rho_\beta(P^*(s,s+\tau)\mu_1,P^*(s,s+\tau)\mu_2)\leq\bar{\eta}\rho_\beta(\mu_1,\mu_2),
    \end{equation*}
    Based on the properties of the periodic Markovian semigroup, we have
     \begin{align*}
         P^*(s,s+n\tau)&=P^*(s+(n-1)\tau,s+n\tau)P^*(s,s+(n-1)\tau)\\
         &=P^*(s,s+\tau)P^*(s,s+(n-1)\tau).
     \end{align*}
     By iteration, we derive that
     \begin{align*}
         &~\rho_\beta(P^*(s,s+n\tau)\mu_1,P^*(s,s+n\tau)\mu_2)\leq \bar{\eta}^n\rho_\beta(\mu_1,\mu_2).
     \end{align*}
     The proof of \eqref{eq:P(s,t)mu leq Cmu} is essentially identical to Theorem 1.3 in \cite{HM11}, so we omit it here.
\end{proof}

\begin{definition}
    A measure-valued function $\mu_\cdot$ is called a periodic measure with period $\tau$ for the two-parameter Markovian semigroup $P(s,t)$ if it satisfies
    \begin{align*}
        \mu_{t+\tau}=\mu_t~~~\mbox{and}~~~P^{*}(s,t)\mu_s=\int_{\br^d}P(s,x,t,\cdot)\mu_s(\d x)=\mu_t.
    \end{align*}
    Moreover, if the corresponding Markov transition kernel $P(s,x,t,\cdot)$ is the transition probability associated with an SDE, we say that the SDE admits a periodic measure $\mu_\cdot$.
\end{definition}

Define
\[\cm_V(\br^d):=\left\{\mu:\br\to\cp(\br^d)\Bigl|\sup_{t\in\br}\int_{\br^d}V(x)\mu_t(\d x)<\infty\right\}.\]

\begin{theorem}\label{thm:exist unique periodic measure}
    Suppose that Assumptions \ref{cond:Markov process} and \ref{cond:Markov process2} hold. Then there exists a unique periodic measure $\mu_\cdot$ in $\cm_V(\br^d)$. Moreover, there is a constant $C_{\beta,\theta,\vartheta,K}=((1+\beta K)\vee\vartheta)((2+\beta K(1-\theta)^{-1})\vee 2\beta)$ such that
    \begin{description}
        \item[(\romannumeral1)] For any $s\in\br$ and $x\in \br^d$, we have
        \begin{equation}\label{eq:P(s,s+nT) leq C etan}
            \rho_\beta(P(s,x,s+n\tau,\cdot),\mu_s)\leq C_{\beta,\theta,\vartheta,K}(1+V(x))\bar{\eta}^n,~~n\in\bn.
        \end{equation}
        \item[(\romannumeral2)] For any $t\geq s$ and $x\in \br^d$, we have
        \begin{equation}\label{eq: rho_beta(P(s,x,t,cdot),mu_t)}
            \rho_\beta(P(s,x,t,\cdot),\mu_t)\leq  C_{\beta,\theta,\vartheta,K}(1+V(x))\bar{\eta}^n,~~\bn\ni n\leq \lfloor(t-s)/\tau\rfloor.
        \end{equation}
        \item[(\romannumeral3)] For any $t\geq s$ and $x\in \br^d$, we have
        \begin{equation}\label{eq:rho_beta(P(s-ntau,x,t,cdot),mu_t)}
            \rho_\beta(P(s-n\tau,x,t,\cdot),\mu_t)\leq C_{\beta,\theta,\vartheta,K}(1+V(x))\bar{\eta}^n,~~n\in\bn.
        \end{equation}
    \end{description}
\end{theorem}

\begin{proof}
For any $t\geq s$, there exists an integer $m$ such that $s+m\tau\leq t<s+(m+1)\tau$. Then it follows from the properties of the periodic Markovian semigroup, \eqref{eq:Lyap function} and \eqref{eq:PV leq V+K} that
\begin{equation}\label{eq:P(s,t)V(x) leq CV(x)+C}
\begin{aligned}
    P(s,t)V(x)&=P(s,s+m\tau)P(s+m\tau,t)V(x)\\
    &\leq \vartheta P(s,s+m\tau)V(x)+K\\
    &\leq \vartheta \theta^{m}V(x)+\vartheta K\sum_{i=0}^{m-1}\theta^i+K\\
    &\leq \vartheta \theta^{m}V(x)+\frac{\vartheta K}{1-\theta}+K.
    \end{aligned}
\end{equation}
In particular, we have
\begin{equation}\label{eq:P(s,s+n tau) leq cV(x)}
    \begin{aligned}
        P(s,s+n\tau)V(x)\leq \theta^nV(x)+\frac{K}{1-\theta}.
    \end{aligned}
\end{equation}
By the definition of the weighted total variation distance, we derive that for all $n\in\bn$,
\begin{equation}\label{eq:rho_beta(delta_x,P*(s,s+ntau)delta)}
\begin{aligned}
    \rho_\beta(\delta_x,P^*(s,s+n\tau)\delta_x)&= \int_{\br^d}(1+\beta V(y))|\delta_x-P^*(s,s+n\tau)\delta_x|(\d y)\\
    &\leq 2+\beta V(x)+\beta P(s,s+n\tau)V(x),
\end{aligned}
\end{equation}
where $\delta_x$ denotes the Dirac measure concentrated at $x$. For any $k\geq n$, we have
\begin{equation}\label{eq: P*(s,s+ktau)}
\begin{aligned}
    P^*(s,s+k\tau)&=P^*(s+(k-n)\tau,s+k\tau)P^*(s,s+(k-n)\tau)\\
    &=P^*(s,s+n\tau)P^*(s,s+(k-n)\tau).
\end{aligned}
\end{equation}
It follows from Lemma \ref{lem:||P(s,s+nT)||leq eta^n}, \eqref{eq:P(s,s+n tau) leq cV(x)} and \eqref{eq:rho_beta(delta_x,P*(s,s+ntau)delta)} that
\begin{equation}\label{eq:rho_beta(P(s+ktau),P(s+ntau)}
\begin{aligned}
    \rho_\beta(P(s,x,s+k\tau,\cdot),P(s,x,s+n\tau,\cdot))&=\rho_\beta(P^*(s,s+k\tau)\delta_x,P^*(s,s+n\tau)\delta_x)\\
    &\leq\bar{\eta}^n\rho_\beta(P^*(s,s+(k-n)\tau)\delta_{x},\delta_{x})\\
    &\leq(2+\beta V(x)+\beta P(s,s+(k-n)\tau)V(x))\bar{\eta}^n\\
    &\leq(2+\beta K(1-\theta)^{-1}+2\beta V(x))\bar{\eta}^n,
\end{aligned}
\end{equation}
where the last inequality follows from $\theta\in(0,1)$. Thus $\{P(s,x,s+n\tau,\cdot)\}_{n\geq 0}$ is a Cauchy sequence with respect to $\rho_\beta$. For brevity, set $\mu_n^x:=P(s,x,s+n\tau,\cdot)$. Since $\Vert\mu_n^x-\mu_m^x\Vert_{\mathrm{TV}}\leq\rho_\beta(\mu_n^x,\mu_m^x)$, the sequence $\{\mu_n^x\}_{n\geq0}$ is also Cauchy in the total variation norm. By the completeness of the space of finite signed measures under the total variation norm, there exists a finite signed measure $\mu_s^x$ such that $\Vert\mu_n^x-\mu_s^x\Vert_{\mathrm{TV}}\to0$. For any $A\in\cb(\br^d)$, we have
\[\mu_s^x(A)=\lim_{n\to\infty}\mu_n^x(A)\geq0~~\text{and}~~\mu_s^x(\br^d)=\lim_{n\to\infty}\mu_n^x(\br^d)=1.\]
Thus $\mu_s^x$ is a probability measure. Moreover, the Lyapunov condition gives
\begin{align*}
    \sup_{n\geq0}\int_{\br^d}V(y)\,\mu_n^x(\d y)=
    \sup_{n\geq0}P(s,s+n\tau)V(x)\leq
    V(x)+K(1-\theta)^{-1}.
\end{align*}
For any $R>0$, let $V_R:=V\wedge R$. Since $V_R$ is bounded
and $\mu_n^x\to\mu_s^x$ in the total variation norm, we have
\[\int_{\br^d}V_R(y)\,\mu_s^x(\d y)=\lim_{n\to\infty}\int_{\br^d}V_R(y)\,\mu_n^x(\d y)\leq V(x)+K(1-\theta)^{-1}.\]
Letting $R\to\infty$ and applying the monotone convergence theorem, we obtain
\[\int_{\br^d}V(y)\,\mu_s^x(\d y)\leq V(x)+K(1-\theta)^{-1}<\infty.\]
Hence $\mu_s^x\in\cp_V(\br^d)$. Finally, the convergence also holds with respect to $\rho_\beta$. Indeed, for any $R>0$, set $U_R(y):=1+\beta(V(y)\wedge R)$. Fix $\ve>0$. Since $\{\mu_n\}_{n\geq0}$ is Cauchy in $\rho_\beta$, there exists $n_0\in\bn$ such that $\rho_\beta(\mu_n^x,\mu_m^x)\leq\ve$ for $n,m\geq n_0$. Then, for every $n\geq n_0$ and $R>0$, we have
\[\int_{\br^d}U_R(y)|\mu_n^x-\mu_s^x|(\d y)=\lim_{m\to\infty}\int_{\br^d}U_R(y)|\mu_n^x-\mu_m^x|(\d y)\leq \ve.\]
As $R\to\infty$, we have $U_R\uparrow 1+\beta V$. Thus, by the monotone convergence theorem,
\[\rho_\beta(\mu_n^x,\mu_s^x)=\lim_{R\to\infty}\int_{\br^d}U_R(y)|\mu_n^x-\mu_s^x|(\d y)\leq \ve,~~n\geq n_0.\]
Since $\ve>0$ is arbitrary, this proves that $\rho_\beta(\mu_n^x,\mu_s^x)\to0$ as $n\to\infty$. For any $y\in\br^d$, we deduce that
\begin{align*}
    \rho_\beta(P(s,x,s+n\tau,\cdot),P(s,y,s+n\tau,\cdot))\leq \bar{\eta}^n\rho_\beta(\delta_x,\delta_y)\leq (2+\beta V(x)+\beta V(y))\bar{\eta}^n,
\end{align*}
which implies that the limit $\mu_s^x$ is independent of the initial value $x$, and we denote it by $\mu_s$. Let $k\to+\infty$ in \eqref{eq:rho_beta(P(s+ktau),P(s+ntau)}, we have
\begin{equation}\label{eq:P(s,s+nT) leq C eta n in proof}
    \rho_\beta(P(s,x,s+n\tau,\cdot),\mu_{s})\leq (2+\beta K(1-\theta)^{-1}+2\beta V(x))\bar{\eta}^n.
\end{equation}
For any $t\geq s$, note that
\begin{align*}
    P^*(s,t+n\tau)&=P^*(t,t+n\tau)P^*(s,t),
\end{align*}
then by Lemma \ref{lem:||P(s,s+nT)||leq eta^n}, \eqref{eq:P(s,t)V(x) leq CV(x)+C} and $\theta\in(0,1)$, we obtain
\begin{align*}
\begin{split}
    \rho_\beta(P(s,x,t+n\tau,\cdot),P(t,x,t+n\tau,\cdot))&\leq \bar{\eta}^n\rho_\beta(P^*(s,t)\delta_x,\delta_x)\\
    &\leq(2+\beta P(s,t)V(x)+\beta V(x))\bar{\eta}^n\\
    &\leq(2+\beta K\vartheta(1-\theta)^{-1}+\beta K+\beta V(x)+\vartheta\beta V(x))\bar{\eta}^n,
\end{split}
\end{align*}
which implies that
\[\lim_{n\to+\infty}P(s,x,t+n\tau,\cdot)=\lim_{n\to+\infty}P(t,x,t+n\tau,\cdot)=\mu_t.\]
We claim that $\mu_\cdot$ is a periodic measure. Indeed, for any $t\geq s$, we have
\begin{align*}
    \mu_{s+\tau}=\lim_{n\to+\infty}P(s+\tau,x,s+\tau+n\tau,\cdot)=\lim_{n\to+\infty}P(s,x,s+n\tau,\cdot)=\mu_s,
\end{align*}
and
\[P^*(s,t)\mu_s=\lim_{n\to+\infty}P^*(s+n\tau,t+n\tau)P^*(s,s+n\tau)\delta_x=\lim_{n\to+\infty}P^*(s,t+n\tau)\delta_x=\mu_t.\]
Therefore, $\mu_\cdot$ is a periodic measure of the Markovian semigroup $P(\cdot,\cdot)$ in $\cm_V(\br^d)$. It remains to prove uniqueness of the periodic measure. Suppose that there exist two periodic measures $\mu_\cdot^1$ and $\mu_\cdot^2$ in $\cm_V(\br^d)$, we have
\begin{align*}
    \rho_\beta(\mu_t^1,\mu_t^2)&=\rho_\beta(\mu_{t+n\tau}^1,\mu_{t+n\tau}^2)\\
    &=\rho_\beta(P^*(t,t+n\tau)\mu_t^1,P^*(t,t+n\tau)\mu_t^2)\\
    &\leq \bar{\eta}^n\rho_\beta(\mu_{t}^1,\mu_{t}^2),
\end{align*}
which shows that $\mu_t^1=\mu_t^2$. 

It follows from \eqref{eq:P(s,s+nT) leq C eta n in proof} that \eqref{eq:P(s,s+nT) leq C etan} holds. For any $t\geq s+n\tau$, there exists an integer $k\geq n$ such that $s+k\tau\leq t<s+(k+1)\tau$ and
\begin{align*}
    P^*(s,t)=P^*(s+k\tau,t)P^*(s,s+k\tau).
\end{align*}
Then it follows from \eqref{eq:P(s,t)mu leq Cmu}, \eqref{eq: P*(s,s+ktau)} and \eqref{eq:P(s,s+nT) leq C etan} that
    \begin{align*}
        \rho_\beta(P(s,x,t,\cdot),\mu_t)&=\rho_\beta(P^*(s+k\tau,t)P^*(s,s+k\tau)\delta_x,P^*(s+k\tau,t)\mu_{s+k\tau})\\
        &\leq((1+\beta K)\vee\vartheta)\rho_\beta(P^*(s,s+k\tau)\delta_x,\mu_{s})\\
        &=((1+\beta K)\vee\vartheta)\rho_\beta(P^*(s,s+n\tau)P^*(s,s+(k-n)\tau)\delta_x,P^*(s,s+n\tau)\mu_s)\\
        &\leq ((1+\beta K)\vee\vartheta)\bar{\eta}^n\rho_\beta(P(s,x,s+(k-n)\tau,\cdot),\mu_{s})\\
        &\leq((1+\beta K)\vee\vartheta)(2+\beta K(1-\theta)^{-1}+2\beta V(x))\bar{\eta}^n,
    \end{align*}
    where the last inequality follows from $k\geq n$ and $\bar{\eta}\in(0,1)$. Furthermore, for any $t\geq s$, there exists an integer $k\in\bn$ such that $s+k\tau\leq t<s+(k+1)\tau$. Then by an argument similar to that used in the proof of \eqref{eq: rho_beta(P(s,x,t,cdot),mu_t)}, we can establish \eqref{eq:rho_beta(P(s-ntau,x,t,cdot),mu_t)}.
\end{proof}

\section{Periodic measures for time-inhomogeneous SDEs}\label{sec:pm for sdes}
Building upon the theory developed in Section \ref{sec:pm for prcess}, we apply those results to $\tau$-periodic SDEs driven by multiplicative $\alpha$-stable noises. We first introduce some notation and terminology.

Denote by $\{L_t,t\geq0\}$ an $\mathbb{R}^d$-valued rotationally invariant $\alpha$-stable process with $\alpha\in(1,2)$, whose L\'evy measure is given by $\nu(\d z)=\frac{C_{\alpha,d}}{|z|^{d+\alpha}}\d z$,
where $C_{\alpha,d}=\alpha 2^{\alpha-1}\pi^{-\frac{d}{2}}\Gamma\left(\frac{d+\alpha}{2}\right)/\Gamma\left(1-\frac{\alpha}{2}\right)$.
For any $1\leq p<\alpha$, it follows that
\[\int_{|z|>1}|z|^p\nu(\d z)=\frac{C_{\alpha,d}S_d}{\alpha-p},~~\int_{|z|\leq 1}|z|^2\nu(\d z)=\frac{C_{\alpha,d}S_d}{2-\alpha},\]
where $S_d$ denotes the surface area of the unit sphere $\mathbb{S}^{d-1}\subset\mathbb{R}^d$, given by $S_d=2\pi^{\frac{d}{2}}/\Gamma\left(\frac{d}{2}\right)$. Let $N(\d z,\d t)$ be the associated Poisson random measure, which satisfies
\[N((s,t]\times\mathcal{E})=\sum_{r\in(s,t]}\bone_{\mathcal{E}}(L_r-L_{r-}),\]
for any $\mathcal{E}\in\cb(\br^d\setminus\{0\})$. By L\'evy-It\^o decomposition, SDE \eqref{SDEs} can be rewritten as
\begin{equation}\label{DSDE}
 \d X_t^{s,x}=b(t, X_t^{s,x})\d t+\int_{|z|\leq 1}\sigma(t,X_{t-}^{s,x})z\Tilde{N}(\d t,\d z)+\int_{|z|>1}\sigma(t,X_{t-}^{s,x})zN(\d t,\d z).
\end{equation}
where $\tilde{N}(\d z,\d t):=N(\d z,\d t)-\nu(\d z)\d t$ denotes the compensated Poisson random measure.

Let $\langle\cdot,\cdot\rangle$ denote the Euclidean inner product, $|\cdot|$ the Euclidean norm, and $\Vert\cdot\Vert$ the standard Frobenius norm. $C^{1,2}(\br\times\br^d)$ denotes the space of functions that are continuously differentiable in first variable and twice continuously differentiable in second variable. 

In this paper, we impose the following assumptions on the coefficients $b$ and $\sigma$.
\begin{condition}\label{A:b}
  For each $R\geq1$, there exists a constant $\ell(R)\geq 1$ such that for all $t\in\br$ and $x$, $y\in\br^d$ with $|x|\vee|y|\leq R$,
    \begin{equation}\label{cond:loc Lip}
        |b(t,x)-b(t,y)|\leq \ell(R)|x-y|,
    \end{equation}
    and there exist constants $l\geq 1$ and $C_b>0$ such that for all $t\in\br$ and $x\in\br^d$,
	\begin{equation}\label{cond:polynomial growth}
		|b(t,x)|\leq C_b(1+|x|^{l}).
	\end{equation}
    Moreover, there exist constants $\lambda> 0$ and $L>0$ such that for all $t\in\br$ and $x\in\br^d$,
    \begin{equation}\label{cond:weak dissipativity}
        \langle x, b(t,x)\rangle\leq L-\lambda|x|^2.
    \end{equation} 
\end{condition}

\begin{condition}\label{A:sigma}
  There exists a constant $C_\sigma>0$ such that for all $t\in\br$ and $x,y\in\br^d$,
   \begin{equation}\label{cond:Lip}
       \Vert\sigma(t,x)-\sigma(t,y)\Vert\leq C_\sigma|x-y|,
   \end{equation}
   and there exists a constant $\Lambda\geq 1$ such that for all $t\in\br$ and $x,\xi\in\br^d$,
   \begin{equation}\label{cond:elliptic and nondeg}
       \Lambda^{-1}|\xi|^2\leq \langle\sigma\sigma^\top(t,x)\xi,\xi\rangle\leq \Lambda|\xi|^2.
   \end{equation}
  \end{condition}

\begin{condition}\label{A:periodic}
  The drift vector $b(t,x)$ and diffusion coefficient matrix $\sigma(t,x)$ are $\tau$-periodic with respect to $t$, i.e., there exists a positive constant $\tau$ such that for all $t\in\br$ and $x\in\br^d$,
   \begin{equation}\label{cond:periodic}
       b(t+\tau,x)=b(t,x),~~~\sigma(t+\tau,x)=\sigma(t,x).
   \end{equation}
\end{condition}

\begin{remark}
Note that \eqref{cond:elliptic and nondeg} shows that all the eigenvalues of $\sigma\sigma^{\top}(t,x)$ will be in $[\Lambda^{-1},\Lambda]$, which implies that
\begin{equation}\label{con:bounded}
    \Vert\sigma(t,x)\Vert:=\sqrt{Tr[\sigma\sigma^{\top}(t,x)]}\in [\sqrt{d \Lambda^{-1}},\sqrt{d \Lambda}].
\end{equation}
In particular, $\sigma(t,x)$ is uniformly bounded and uniformly non-degenerate.
\end{remark}

Throughout the paper, we use the notation
\[\Theta:=(p,\alpha,d,C_b,\lambda, L ,C_\sigma,\Lambda,\ell,l)\]
to denote a collection of parameters. The symbol $C_\cdot$ denotes a positive constant whose dependence is specified by the subscript and whose value may vary from line to line.

\begin{theorem}\label{thm:E[|X|]}
    Suppose that Assumptions \ref{A:b} and \ref{A:sigma} hold. Then for each initial value, there exists a unique regular solution $\{X_t^{s,x}\}_{t\geq s}$ of SDE \eqref{SDEs}. Moreover, for every $1\leq p<\alpha$, there exist constants $\varpi>0$ and $C_\Theta>0$ such that
    \begin{equation}\label{xp estimate}
        \be\left[|X_{t}^{s,x}|^p\right]\leq \e^{-\varpi(t-s)}(1+|x|^2)^{\frac{p}{2}}+C_\Theta(1-\e^{-\varpi(t-s)}).
    \end{equation}
    and
    \begin{equation}\label{eq:supxp estimate}
        \be\left[\sup_{s\leq r\leq t}|X_r^{s,x}|^p\right]\leq C_\Theta(1+|x|^2)^{\frac{p}{2}}+C_\Theta((t-s)^p+(t-s)^\frac{1}{2}).
    \end{equation}
\end{theorem}

\begin{proof}
    The infinitesimal generator operator of \eqref{SDEs} denoted as $\mathscr{L}$. For any $V\in C^{1,2}(\br\times\br^d)$,
    \[\mathscr{L} V(t,x)=\partial_t V(t,x)+\langle b(t,x),\nabla_xV(t,x)\rangle+\mathscr{L}_\sigma V(t,x),\]
    where
    \begin{equation}\label{eq: LV(t,x)}
        \mathscr{L}_\sigma V(t,x):=\int_{\br^d\backslash\{0\}}(V(t,x+\sigma(t,x)z)-V(t,x)-\bone_{\{|z|\leq 1\}}\langle\nabla_xV(t,x),\sigma(t,x)z\rangle)\nu(\d z).
    \end{equation}
    For any $1\leq p<\alpha$, we define the function $V(x)=(1+|x|^2)^{\frac{p}{2}}$. Then we have
    \[\nabla V(x)=\frac{px}{(1+|x|^2)^{1-\frac{p}{2}}},~~\nabla^2V(x)=\frac{pI_d}{(1+|x|^2)^{1-\frac{p}{2}}}-\frac{p(2-p)x\otimes x}{(1+|x|^2)^{2-\frac{p}{2}}},\]
    where $I_d$ denotes the $d\times d$ identity matrix and $\otimes$ is tensor product. It is not difficult to see that
    \begin{equation}\label{eq:nabla V and nabla2 V}
        |\nabla V(x)|\leq p|x|^{p-1},~~\Vert\nabla^2V(x)\Vert\leq p(3-p)\sqrt{d}.
    \end{equation}
    Therefore, by \eqref{cond:weak dissipativity}, \eqref{eq:nabla V and nabla2 V} and \eqref{con:bounded}, we deduce that
    \begin{align*}
        \mathscr{L} V(x)&=p(1+|x|^2)^{{\frac{p}{2}-1}}\langle x,b(t,x)\rangle+\int_{|z|> 1} [V(x+\sigma(t,x)z)-V(x)]\nu(\d z)\\
        &~~~~+\int_{|z|\leq 1}[V(x+\sigma(t,x)z)-V(x)-\langle\nabla V(x),\sigma(t,x)z\rangle]\nu(\d z)\\
        &\leq p(1+|x|^2)^{{\frac{p}{2}-1}}(L-\lambda|x|^2)+\int_{|z|>1}\int_0^1\langle\nabla V(x+r\sigma(t,x)z),\sigma(t,x)z\rangle\d r\nu(\d z)\\
        &~~~+\int_{|z|\leq 1}\int_0^1\int_0^1\langle \nabla^2V(x+r_2r_1\sigma(t,x)z)r_1\sigma(t,x)z,\sigma(t,x)z\rangle \d r_2\d r_1\nu(\d z)\\
        &\leq -2^{-1}p\lambda(1+|x|^2)^{\frac{p}{2}}+C_\Theta.
    \end{align*}
    For any fixed $\varpi>0$, define $V(t,x)=\re^{\varpi(t-s)}V(x)$. It follows that
    \begin{align*}
        \mathscr{L} V(t,x)&=\varpi \re^{\varpi(t-s)}V(x)+\re^{\varpi (t-s)}\mathscr{L} V(x)\\
        &\leq \left(\varpi-2^{-1}p\lambda\right)\re^{\varpi(t-s)}(1+|x|^2)^{\frac{p}{2}}+C_\Theta\re^{\varpi(t-s)}.
    \end{align*}
    Since $\lambda> 0$, we can choose $\ve$ small enough such that $\lambda>2\ve$. Let $\varpi:=p(2^{-1}\lambda -\ve)>0$ depend only on $(p,\lambda)$, then we have
    \[\mathscr{L} V(t,x)\leq -\ve p\re^{\varpi(t-s)}(1+|x|^2)^{\frac{p}{2}}+C_\Theta\re^{\varpi(t-s)}.\]
    For any $N\in\bn$, define the stopping time
    \begin{equation}
        \varsigma_N:=\inf\{t\geq s:|X_t^{s,x}|\geq N\}
    \end{equation}
    with $\inf\emptyset=+\infty$. Then by Dynkin's formula, we have
    \begin{align*}
        \be\left[V\left(t\wedge\varsigma_N,X_{t\wedge\varsigma_N}^{s,x}\right)-V(s,x)\right]&=\be\left[\int_s^{t\wedge\varsigma_N}\mathscr{L} V(r,X_r^{s,x})\d r\right]\leq C_\Theta\int_s^t\e^{\varpi (r-s)}\d r,
    \end{align*}
    which implies that
    \begin{equation}\label{exp estimate}
        \be\left[\e^{\varpi(t\wedge\varsigma_N-s)}(1+|X_{t\wedge\varsigma_N}^{s,x}|^2)^{\frac{p}{2}}\right]\leq(1+|x|^2)^{\frac{p}{2}}+C_\Theta(\re^{\varpi(t-s)}-1).
    \end{equation}
    Note that
    \begin{align*}
        (1+N^2)^{\frac{p}{2}}\bp\{{\varsigma_N\leq t}\}&= \be\left[(1+N^2)^{\frac{p}{2}}\bone_{\{\varsigma_N\leq t\}}\right]\\
        &\leq \be\left[(1+|X_{t\wedge\varsigma_N}^{s,x}|^2)^{\frac{p}{2}}\bone_{\{\varsigma_N\leq t\}}\right]\\
        &\leq \be\left[(1+|X_{t\wedge\varsigma_N}^{s,x}|^2)^{\frac{p}{2}}\right]\\
        &\leq\be\left[\e^{\varpi(t\wedge\varsigma_N-s)}(1+|X_{t\wedge\varsigma_N}^{s,x}|^2)^{\frac{p}{2}}\right].
    \end{align*}
    Therefore it follows from \eqref{exp estimate} that
    \begin{equation}\label{eq:P(tau<t)}
        \bp\{\varsigma_N\leq t\}\leq\frac{(1+|x|^2)^{\frac{p}{2}}+C_\Theta(\e^{\varpi (t-s)}-1)}{(1+N^2)^{\frac{p}{2}}}.
    \end{equation}
    Letting $N\to+\infty$ in \eqref{eq:P(tau<t)}, we derive that
    \begin{equation}\label{stopping time to 0}
        \bp\{\varsigma_N\leq t\}\rightarrow  0.
    \end{equation}
    By the monotone convergence theorem, we have
    \[\be\left[\lim_{N\to+\infty}\m1_{\{\varsigma_N>t\}}\right]=\lim_{N\to+\infty}\be\left[\m1_{\{\varsigma_N>t\}}\right]=\lim_{N\to+\infty}\bp\{\varsigma_N>t\}=1.\]
    Thus $\lim_{N\to+\infty}\varsigma_N=+\infty$, $\bp-a.s.$. Consequently, SDE \eqref{SDEs} admits a unique regular solution. Moreover, it also follows from \eqref{exp estimate} that
    \begin{equation}\label{expstoping estimate}
        \be\left[\e^{\varpi(t-s)}(1+|X_{t}^{s,x}|^2)^{\frac{p}{2}}\bone_{\{\varsigma_N>t\}}\right]\leq (1+|x|^2)^{\frac{p}{2}}+C_\Theta(\e^{\varpi(t-s)}-1).
    \end{equation}
    Taking the limit as $N\to+\infty$ in \eqref{expstoping estimate} and by \eqref{stopping time to 0}, we have
    \begin{equation*}
        \be\left[(1+|X_{t}^{s,x}|^2)^{\frac{p}{2}}\right]\leq \e^{-\varpi(t-s)}(1+|x|^2)^{\frac{p}{2}}+C_\Theta(1-\e^{-\varpi(t-s)}),
    \end{equation*}
    which implies that
    \begin{equation*}
        \be\left[|X_{t}^{s,x}|^p\right]\leq \e^{-\varpi(t-s)}(1+|x|^2)^{\frac{p}{2}}+C_\Theta(1-\e^{-\varpi(t-s)}).
    \end{equation*} 
In the following, we prove \eqref{eq:supxp estimate} holds. For simplicity, we focus on $1<p<\alpha$, while the case $p=1$ can be handled similarly. Applying It\^o's formula to $V(x)=(1+|x|^2)^{\frac{p}{2}}$, we have
    \begin{align*}
        V(X_t^{s,x})&=V(x)+\int_s^tp(1+|X_r^{s,x}|^2)^{{\frac{p}{2}-1}}\langle X_r^{s,x},b(r,X_r^{s,x})\rangle\d r\\
        &\quad+\int_s^t\int_{|z|\leq 1}[V(X_{r-}^{s,x}+\sigma(r,X_{r-}^{s,x})z)-V(X_{r-}^{s,x})]\Tilde{N}(\d z,\d r)\\
        &\quad+\int_s^t\int_{|z|> 1}[V(X_{r-}^{s,x}+\sigma(r,X_{r-}^{s,x})z)-V(X_{r-}^{s,x})]N(\d z,\d r)\\
        &\quad+\int_s^t\int_{|z|\leq 1}[V(X_{r-}^{s,x}+\sigma(r,X_{r-}^{s,x})z)-V(X_{r-}^{s,x})\\
        &\qquad\qquad\qquad\qquad~-\langle \nabla V(X_{r-}^{s,x}),\sigma(r,X_{r-}^{s,x})z\rangle]\nu(\d z)\d r\\
        &=: (1+|x|^2)^{\frac{p}{2}}+I_1(t)+I_2(t)+I_3(t)+I_4(t).
    \end{align*}
    For the term $I_1(t)$, the weak dissipativity \eqref{cond:weak dissipativity} implies
    \begin{align*}
        \be\left[\sup_{s\leq r\leq t}I_1(r)\right]&\leq p\be\left[\sup_{s\leq r\leq t}\int_s^r\frac{L-\lambda|X_u^{s,x}|^2}{(1+|X_u^{s,x}|^2)^{1-\frac{p}{2}}}\d u\right]\leq C_\Theta(t-s).
    \end{align*}
     For the term $I_2(t)$, it follows from Kunita's inequality and \eqref{eq:nabla V and nabla2 V} that
    \begin{align*}
        \be\left[\sup_{s\leq r\leq t}I_2(r)\right]&\leq \be\left[\int_s^t\int_{|z|\leq 1}|V(X_{u-}^{s,x}+\sigma(u,X_{u-}^{s,x})z)-V(X_{u-}^{s,x})|^2N(\d z,\d u)\right]^\frac{1}{2}\\
        &\leq C_p\be\left[\int_s^t\int_{|z|\leq 1}(|X_{u-}^{s,x}|^{2p-2}|\sigma(u,X_{u-}^{s,x})z|^{2}+|\sigma(u,X_{u-}^{s,x})z|^{2p})N(\d z,\d u)\right]^\frac{1}{2}\\
        &\leq C_p\be\left[\sup_{s\leq u\leq t}|X_u^{s,x}|^{2p-2}\int_s^t\int_{|z|\leq 1}|\sigma(u,X_{u-}^{s,x})z|^{2}N(\d z,\d u)\right]^\frac{1}{2}\\
        &~~~+C_p\be\left[\int_s^t\int_{|z|\leq 1}|\sigma(u,X_{u-}^{s,x})z|^{2p}N(\d z,\d u)\right]^\frac{1}{2}\\
        &\leq\frac{1}{4}\be\left[\sup_{s\leq u\leq t}|X_u^{s,x}|^{p}\right]+C_\Theta(t-s)^{\frac{p}{2}}+C_\Theta(t-s)^\frac{1}{2}.
    \end{align*}
    For the term $I_3(t)$, we deduce from Young's inequality and \eqref{eq:nabla V and nabla2 V} that
    \begin{align*}
        \be\left[\sup_{s\leq r\leq t}I_3(r)\right]
        &\leq \be\left[\int_s^t\int_{|z|> 1}\int_0^1p|X_u^{s,x}+r\sigma(u,X_{u-}^{s,x})z|^{p-1}|\sigma(u,X_{u-}^{s,x})z|\nu(\d z)\d u\right]\\
        &\leq C_p\be\left[\int_s^t\int_{|z|> 1}(|X_u^{s,x}|^{p-1}|\sigma(u,X_{u-}^{s,x})z|+|\sigma(u,X_{u-}^{s,x})z|^p)\nu(\d z)\d u\right]\\
        &\leq C_p\be\left[\sup_{s\leq u\leq t}|X_u^{s,x}|^{p-1}\int_s^t\int_{|z|> 1}|\sigma(u,X_{u-}^{s,x})z|\nu(\d z)\d u\right]\\
        &~~~+C_p\be\left[\int_s^t\int_{|z|> 1}|\sigma(u,X_{u-}^{s,x})z|^p\nu(\d z)\d u\right]\\
        &\leq \frac{1}{4}\be\left[\sup_{s\leq u\leq t}|X_u^{s,x}|^{p}\right]+C_\Theta(t-s)^p+C_\Theta(t-s).
    \end{align*}
    For the term $I_4(t)$, we apply \eqref{eq:nabla V and nabla2 V} again to obtain
    \begin{equation*}\label{I4}
    \begin{aligned}
        \be\left[\sup_{s\leq r\leq t}I_4(r)\right]&\leq C_\Theta\be\left[\int_s^t\int_{|z|\leq1}|\sigma(u,X_{u-}^{s,x})z|^2\nu(\d z)\d u\right]\leq C_\Theta(t-s).
    \end{aligned}
    \end{equation*}
    Putting the above estimates together, there exists a positive constant $C_\Theta$ such that
    \begin{align*}
        \be\left[\sup_{s\leq r\leq t}|X_r^{s,x}|^p\right]&\leq C_\Theta(1+|x|^2)^{\frac{p}{2}}+C_\Theta((t-s)^p+(t-s)^\frac{1}{2}).
    \end{align*}
    The proof is finished.
\end{proof}

For any fixed $N\in\bn$, we construct $b_N(t,x)$ as follows:
\begin{equation}\label{bN}
    b_N(t,x)=\begin{cases}
        b(t,x),&|x|\leq N,\\
        b\left(t,\frac{N}{|x|}x\right), &|x|>N.
    \end{cases}
\end{equation}
Since $b(t,x)$ is locally Lipschitz continuous in $x$ and uniformly in $t$, it is not difficult to see that the truncated coefficient $b_N(t,\cdot)$ is globally Lipschitz continuous and bounded, with a Lipschitz constant $\ell(N)$. We consider the following truncated SDE:
\begin{equation}\label{SDEN}
\begin{cases}
 \d X_t^{s,x,N}=b_N(t, X_t^{s,x,N})\d t+\sigma(t,X_{t-}^{s,x,N})\d L_t,  \quad t\geq s,\\
 X_s^N=x.
\end{cases}
\end{equation}
Under Assumptions \ref{A:b} and \ref{A:sigma}, together with \eqref{bN}, it is well known from \cite[Theorem 6.2.3]{A09} that, for each initial condition $(s,x)$, SDE \eqref{SDEN} admits a unique solution $X_t^{s,x,N}$. To verify the hypotheses of \cite[Theorem 1.5]{CZ18} explicitly, we make the change of variables $w=\sigma(t,x)z$ in the L\'evy kernel:
\[\nu_{t,x}(\d w)=j(t,x,w)\d w,~~j(t,x,w)=C_{\alpha,d}|\det \sigma(t,x)^{-1}||\sigma(t,x)^{-1}w|^{-d-\alpha}=\kappa(t,x,w)|w|^{-d-\alpha},\]
where
\[\kappa(t,x,w):=C_{\alpha,d}|\det\sigma(t,x)^{-1}|\left(\frac{|w|}{|\sigma(t,x)^{-1}w|}\right)^{d+\alpha},~~w\neq 0.\]
The transformed kernel is even in $w$, so symmetry makes the drift correction from changing the compensation convention vanish. The uniform ellipticity in \eqref{cond:elliptic and nondeg} and the spatial Lipschitz continuity in \eqref{cond:Lip} imply, uniformly in $t\in\mathbb R$ and $w\neq0$, that
\[c_0^{-1}\leq\kappa(t,x,w)\leq c_0,~~|\kappa(t,x,w)-\kappa(t,y,w)|\leq c_0|x-y|
\]
for some $c_0\geq1$. Thus the transformed kernel satisfies the comparability and spatial regularity condition in \cite[(1.3)]{CZ18} with $\beta=1$ and $\beta'=0$. Since $\alpha\in(1,2)$, the cancellation condition required only at $\alpha=1$ is irrelevant, and the bounded drift $b_N$ belongs to the Kato class $\mathbb K_d^\alpha$. Therefore, by \cite[Theorem 1.5]{CZ18}, the solution defines a Markov process whose transition probability admits a jointly continuous density $p^N(s,x,t,y)$ with respect to the Lebesgue measure. In addition, $p^N(s,x,t,y)$ satisfies sharp two-sided estimates. We collect these results in the following lemma for later use.

\begin{lemma}\label{lem:two-side estimate and gradient estimate}
Suppose that Assumptions \ref{A:b} and \ref{A:sigma} hold. Then $X_t^{s,x,N}$ has a jointly continuous density $p^N(s,x,t,y)$ with respect to $(s,x,t,y)$. For every $T>0$, there exists a constant $C_{N,T}\geq 1$ depending on $N$ and $T$ such that for all $0<t-s<T$ and $x,y\in\br^d$,
    \begin{equation}\label{eq:two-sides estimate}
        \frac{C_{N,T}^{-1}(t-s)}{((t-s)^{1/\alpha}+|x-y|)^{d+\alpha}}\leq p^N(s,x,t,y)\leq \frac{C_{N,T}(t-s)}{((t-s)^{1/\alpha}+|x-y|)^{d+\alpha}}.
    \end{equation}
\end{lemma}

\begin{lemma}\label{lem:stochastic continuous}
    Suppose that Assumptions \ref{A:b} and \ref{A:sigma} hold. Let $X_t^{s,x}$ and $X_t^{s,x,N}$ be the unique solution of SDEs \eqref{SDEs} and \eqref{SDEN}, respectively, with initial condition $(s,x)$. Then we have
        \begin{equation}\label{eq:XN time dependence}
            \lim_{t\to r}\be[|X_t^{s,x,N}-X_r^{s,x,N}|]=0~~\mbox{and}~~\lim_{N\to \infty}\be[|X_t^{s,x}-X_t^{s,x,N}|]=0.
        \end{equation}
        Moreover, we have
        \begin{equation}\label{eq:prob}
            \lim_{t\to r}\be[|X_t^{s,x}-X_r^{s,x}|]=0~~\mbox{and}~~\lim_{t\to r}\bp(|X_t^{s,x}-X_r^{s,x}|>\ve)=0.
        \end{equation}
        In particular, $P^N(s,t)$ and $P(s,t)$ are stochastically continuous.
\end{lemma}

\begin{proof}
    For every fixed $N\in\mathbb N$, the definition of $b_N$ and \eqref{cond:polynomial growth} imply that
    \begin{equation}\label{eq:bN bounded}
        \sup_{t\in\mathbb R,\,x\in\mathbb R^d}|b_N(t,x)|\leq C_b(1+N^l).
    \end{equation}
    Let $s\leq r\leq t\leq T$. From the truncated equation, we have
    \[X_t^{s,x,N}-X_r^{s,x,N}=\int_r^t b_N(u,X_u^{s,x,N})\,\d u+\int_r^t\sigma(u,X_{u-}^{s,x,N})\d L_u.\]
    Using \eqref{eq:bN bounded}, the boundedness of $\sigma$, and the standard moment estimate, we obtain
    \begin{align}\label{eq:XN increment estimate}
        \be[|X_t^{s,x,N}-X_r^{s,x,N}|]&\leq C_b(1+N^l)(t-r)+C_{\alpha,d,\Lambda}(t-r)^\frac{1}{\alpha}.
    \end{align}
    Therefore, for every fixed $N$ and $r\geq s$,
    \begin{equation}\label{eq:XN L1 continuous}
        \lim_{t\to r}\be[|X_t^{s,x,N}-X_r^{s,x,N}|]=0.
    \end{equation}
    The case $t<r$ follows from the same estimate with the roles of $t$ and $r$ interchanged.

    We next establish a finite-horizon moment estimate that is uniform with respect to $N$. Fix $p\in(1,\alpha)$ and define $V(x):=(1+|x|^2)^\frac{p}{2}$. Let $N_0:=\lceil1\vee\sqrt{L/\lambda}\rceil$. For $N\geq N_0$, if $|x|\leq N$, then
    \[\langle x,b_N(t,x)\rangle=\langle x,b(t,x)\rangle\leq L-\lambda|x|^2\leq L.\]
    If $|x|>N$, then
    \begin{align*}
        \langle x,b_N(t,x)\rangle=\frac{|x|}{N}\left\langle\frac{N}{|x|}x,b\left(t,\frac{N}{|x|}x\right)\right\rangle\leq\frac{|x|}{N}(L-\lambda N^2)\leq0.
    \end{align*}
    Therefore, we have
    \begin{equation}\label{eq:bN uniform radial bound}
        \langle x,b_N(t,x)\rangle\leq L,~~ t\in\mathbb R,~x\in\mathbb R^d,~N\geq N_0.
    \end{equation}
    It follows from \eqref{eq:bN uniform radial bound} that
    \begin{align*}
        \langle\nabla V(x),b_N(t,x)\rangle=p(1+|x|^2)^{\frac{p}{2}-1}\langle x,b_N(t,x)\rangle\leq pL.
    \end{align*}
    Denote the generator of the truncated equation by $\mathscr{L}^N$. The standard estimate for the nonlocal part of the generator and the boundedness of $\sigma$ give $\mathscr{L}_\sigma^N V(x)\leq C_{\alpha,d,\Lambda}$.
    Hence, we obtain
    \begin{equation}\label{eq:uniform generator estimate bN}
        \mathscr{L}^N V(x)\leq C_{\Theta},~~N\geq N_0,
    \end{equation}
    where the constant is independent of $N$, $t$ and $x$. For $m\geq1$, define
    \[\varsigma_m^N:=\inf\{t\geq s:|X_t^{s,x,N}|\geq m\}.\]
    Applying Dynkin's formula to
    $V(X_{t\wedge\varsigma_m^N}^{s,x,N})$ and using \eqref{eq:uniform generator estimate bN}, we obtain
    \[\be[V(X_{t\wedge\varsigma_m^N}^{s,x,N})]
    \leq V(x)+C_{\Theta}(t-s),~~s\leq t\leq T.\]
    Letting $m\to\infty$ and applying Fatou's lemma yield
    \begin{equation}\label{eq:uniform rho moment XN}
        \sup_{N\geq N_0}\sup_{s\leq t\leq T}\be[|X_t^{s,x,N}|^p]\leq C_{\Theta,T-s}(1+|x|^p)<\infty.
    \end{equation}
    Recall that \[\varsigma_N:=\inf\{t\geq s:|X_t^{s,x}|\geq N\}.\]
    On the event $\{\varsigma_N>T\}$,
    \begin{equation}\label{eq:solutions agree finite interval}
        X_t^{s,x,N}=X_t^{s,x},~~s\leq t\leq T.
    \end{equation}
    Since $X^{s,x}$ is non-explosive and has c\`adl\`ag paths, we have
    \[\sup_{s\leq t\leq T}|X_t^{s,x}|<\infty
    ~~\bp-\text{a.s.}\]
    Therefore, we arrive at
    \[\bp(\varsigma_N\leq T)\to0~~\text{as}~N\to\infty,\]
    and \eqref{eq:solutions agree finite interval} implies that
    \begin{equation}\label{eq:XN converges pathwise}
        \sup_{s\leq t\leq T}|X_t^{s,x,N}-X_t^{s,x}|\to0~~\bp-\text{a.s.}
\end{equation}

For every fixed $t\in[s,T]$, Fatou's lemma, \eqref{eq:XN converges pathwise} and \eqref{eq:uniform rho moment XN} give
\[\be[|X_t^{s,x}|^{p}]\leq C_{\Theta,T-s}(1+|x|^p).\]
Therefore, we have
\[\sup_{N\geq N_0}\be[|X_t^{s,x,N}-X_t^{s,x}|^p]\leq 2^{p-1}(\sup_{N\geq N_0}\be[|X_t^{s,x,N}|^{p}]+\be[|X_t^{s,x}|^p])\leq C_{\Theta,T-s}(1+|x|^p).\]
Since $p>1$, the family
\[\{|X_t^{s,x,N}-X_t^{s,x}|:N\geq N_0\}\]
is uniformly integrable. By \eqref{eq:XN converges pathwise} and Vitali's convergence theorem,
\begin{equation}\label{eq:XN converges L1}
    \lim_{N\to\infty}
    \be[|X_t^{s,x,N}-X_t^{s,x}|]=0,
    ~~t\in[s,T].
\end{equation}
This proves the second assertion in
\eqref{eq:XN time dependence}.

For the stochastic continuity argument, we need the following finite-horizon uniform version of \eqref{eq:XN converges L1}. By \eqref{eq:solutions agree finite interval} and H\"older's inequality,
\begin{align*}
    \sup_{s\leq t\leq T}\be[|X_t^{s,x,N}-X_t^{s,x}|]&\leq\sup_{s\leq t\leq T}
    \be[(|X_t^{s,x,N}|+|X_t^{s,x}|)\bone_{\{\varsigma_N\leq T\}}]\\
    &\leq C_{\Theta,T-s}(1+|x|^p)^\frac{1}{p}
    \bp(\varsigma_N\leq T)^\frac{p-1}{p}.
\end{align*}
It follows that
\begin{equation}\label{eq:uniform L1 approximation}
    \lim_{N\to\infty}\sup_{s\leq t\leq T}
    \be[|X_t^{s,x,N}-X_t^{s,x}|]=0.
\end{equation}
For $s\leq r,t\leq T$, the triangle inequality gives
\begin{align*}
    \be[|X_t^{s,x}-X_r^{s,x}|]
    &\leq
    \be[|X_t^{s,x}-X_t^{s,x,N}|]+
    \be[|X_t^{s,x,N}-X_r^{s,x,N}|]+
    \be[|X_r^{s,x,N}-X_r^{s,x}|].
\end{align*}
Fixing $N$ and first letting $t\to r$, we obtain from
\eqref{eq:XN L1 continuous} that
\[\limsup_{t\to r}\be[|X_t^{s,x}-X_r^{s,x}|]\leq2\sup_{s\leq u\leq T}
    \be[|X_u^{s,x,N}-X_u^{s,x}|].\]
Letting $N\to\infty$ and using \eqref{eq:uniform L1 approximation}, we conclude that
\begin{equation}\label{eq: Xt to Xr}
    \lim_{t\to r}\be[|X_t^{s,x}-X_r^{s,x}|]=0.
\end{equation}
Finally, for every $\varepsilon>0$, Markov's inequality gives
\[\bp(|X_t^{s,x}-X_r^{s,x}|>\varepsilon)\leq\frac{1}{\varepsilon}\be[|X_t^{s,x}-X_r^{s,x}|]\to0\]
as $t\to r$. The same conclusion for $X^{s,x,N}$ follows from \eqref{eq:XN L1 continuous}. Taking $r=s$ and using $X_s^{s,x,N}=X_s^{s,x}=x$ proves the stochastic continuity of $P^N$ and $P$.
\end{proof}

In the following, we consider an open set $D\subset\br^d$ with a smooth boundary $\partial D$. For any $x\in D$, define the first exit time from the domain $D$ by
\begin{equation*}\label{eq:first exit time}
    \varsigma_D:=\inf\{t\geq s:X_t^{s,x,N}\notin D\}
\end{equation*}
with $\inf\emptyset=+\infty$.
For all $t>s$, and $x\in D$ and $y\in\br^d$, define
\begin{equation*}\label{eq:rN(s,x,t,y)}
    r_D^N(s,x,t,y):=\be[p^N(\varsigma_D,X_{\varsigma_D}^{s,x,N},t,y)\bone_{\{s\leq\varsigma_D<t\}}]
\end{equation*}
and
\begin{equation}\label{eq:pN(s,x,t,y)}
    p_D^N(s,x,t,y):=p^N(s,x,t,y)-r_D^N(s,x,t,y).
\end{equation}

Hunt \cite{H56} established that the transition probabilities of killed Brownian motion satisfy some fundamental properties, including the semigroup property and basic regularity properties; see also Chung--Zhao \cite[Chapter 2]{CZ12} for a detailed presentation of these results. The main ideas of their proof can be extended to time-inhomogeneous c\`adl\`ag Markov processes.

\begin{theorem}\label{thm:p_D^N(s,x,t,y)}
    Suppose that Assumptions \ref{A:b} and \ref{A:sigma} hold. Then, for every $t>s$, $x\in D$ and $A\in\cb(\br^d)$,
    \begin{equation}\label{eq:P(X,tau)}
        \bp(X_t^{s,x,N}\in A,t<\varsigma_D)=\int_{A}p_D^N(s,x,t,y)\d y.
    \end{equation}
    Moreover, the function $p_D^N(s,x,t,y)$ satisfies Chapman--Kolmogorov equation, i.e. for $t>r>s$ and every $x\in D$,
    \begin{equation}\label{eq:pN CK}
        p_D^N(s,x,t,y)=\int_{D}p_D^N(s,x,r,z)p_D^N(r,z,t,y)\d z,
    \end{equation}
    for almost every $y\in \br^d$. Let $G\Subset D$ be a compact convex set with nonempty interior. Then, for every fixed $h_0>0$, there exists a constant $\Pi:=\Pi(N,D,G,h_0,d,\alpha,C_b,l,C_\sigma,\Lambda)>0$ such that
    \begin{equation}\label{eq:p_D(s,x,t,y)>Lambda}
        p_D^N(s,x,s+h_0,y)\geq\Pi
    \end{equation}
    for every $s\in\mathbb R$ and $x\in G$, and for almost every $y\in G$.
\end{theorem}

\begin{proof}
    By Lemma \ref{lem:stochastic continuous}, we know that the truncated process $X^{s,x,N}$ is stochastically continuous. Since $X^{s,x,N}$ is also a c\`adl\`ag process, it follows that
    \begin{equation*}
        \bp(X_t^{s,x,N}\neq X_{t-}^{s,x,N})=0.
    \end{equation*}
    Because the law of  $X^{s,x,N}$ is absolutely continuous with respect to Lebesgue measure, we have
    \begin{equation*}
        \bp(X_t^{s,x,N}\in\partial D)=0.
    \end{equation*}
    Observe that
    \[\{\varsigma_D=t\}\subset\{X_t^{s,x,N}\in\partial D\}\cup\{X_t^{s,x,N}\neq X_{t-}^{s,x,N}\}.\]
    Indeed, if $\varsigma_D=t$ and $X_t^{s,x,N}=X_{t-}^{s,x,N}$, then $X_{t-}^{s,x,N}\in\bar{D}$ and $X_t^{s,x,N}\in D^c$, which implies $X_t^{s,x,N}\in\partial D$.
    Then we derive that
    \begin{align*}
        \bp(\varsigma_D=t)&\leq\bp(X_t^{s,x,N}\in\partial D)+\bp(X_t^{s,x,N}\neq X_{t-}^{s,x,N})\\
        &=0.
    \end{align*}
    For all $A\in\cb(\br^d)$, it is not difficult to see that
    \begin{align*}
        \bp(X_t^{s,x,N}\in A,t<\varsigma_D)&=\bp(X_t^{s,x,N}\in A)-\bp(X_t^{s,x,N}\in A,s\leq\varsigma_D\leq t)\notag\\
        &=\int_{A}p^N(s,x,t,y)\d y-\bp(X_t^{s,x,N}\in A,s\leq\varsigma_D<t).
    \end{align*}
    For $n\in\bn$ and $t>r>s$, let
    \[t_k^n:=s+k(r-s)2^{-n},~~k=0,1,\cdots,2^n,\]
    and define
    \begin{align*} 
       T_n=\left\{\begin{array}{ll} t_k^n, & \mbox{if}~~t_{k-1}^n\leq \varsigma_D<t_k^n, \\ r,& \mbox{if}~~\varsigma_D\geq r. \end{array}\right. 
    \end{align*}
    Then we have
    \begin{align*}
        \bp(X_t^{s,x,N}\in A,s\leq\varsigma_D<r)&=\sum_{k=1}^{2^n}\bp(X_t^{s,x,N}\in A,t_{k-1}^n\leq \varsigma_D<t_k^n)\\
        &=\sum_{k=1}^{2^n}\be\left[\bp(X_t^{s,x,N}\in A|\cf_{t_k^n})\bone_{\{t_{k-1}^n\leq \varsigma_D<t_k^n\}}\right]\\
        &=\sum_{k=1}^{2^n}\be\left[\bp(X_t^{t_k^n,X_{t_k^n}^{s,x,N},N}\in A)\bone_{\{t_{k-1}^n\leq \varsigma_D<t_k^n\}}\right]\\
        &=\sum_{k=1}^{2^n}\be\left[\int_{A}p^N(t_k^n,X_{t_k^n}^{s,x,N},t,y)\d y\bone_{\{t_{k-1}^n\leq \varsigma_D<t_k^n\}}\right]\\
        &=\be\left[\int_{A}p^N(T_n,X_{T_n}^{s,x,N},t,y)\d y\bone_{\{s\leq \varsigma_D<r\}}\right].
    \end{align*}
On $\{\varsigma_D<r\}$, we have
$T_n\downarrow\varsigma_D^{s,x,N}$ as $n\to\infty$. By the right-continuity of
the sample paths,
\[X_{T_n}^{s,x,N}\longrightarrow X_{\varsigma_D}^{s,x,N}.\]
By Lemma \ref{lem:two-side estimate and gradient estimate}, we know that $p^N(s,x,t,y)$ is jointly continuous with respect to $(s,x,t,y)$. Hence,
\[p^N(T_n,X_{T_n}^{s,x,N},t,y)\longrightarrow p^N(\varsigma_D,X_{\varsigma_D}^{s,x,N},t,y)\]
for any $y\in \br^d$. For every fixed $\omega\in\{\varsigma_D<r\}$, the functions
\[y\mapsto p^N(T_n(\omega),X_{T_n(\omega)}^{s,x,N},t,y)\]
and
\[y\mapsto p^N(\varsigma_D(\omega),X_{\varsigma_D(\omega)}^{s,x,N},t,y)\]
are probability densities on $\br^d$. Therefore, Scheff\'e's lemma yields
\[\int_{\br^d}|p^N(T_n,X_{T_n}^{s,x,N},t,y)-p^N(\varsigma_D, X_{\varsigma_D}^{s,x,N},t,y)|\d y\longrightarrow0.\]
Then we have
\[\int_Ap^N(T_n,X_{T_n}^{s,x,N},t,y)\d y\longrightarrow\int_Ap^N(\varsigma_D, X_{\varsigma_D}^{s,x,N},t,y)\d y.\]
Since these integrals are bounded by one, the dominated convergence theorem and Tonelli's theorem give
\begin{align*}
    \bp(X_t^{s,x,N}\in A,s\leq \varsigma_D<r)&=\int_{A}\be[p^N(\varsigma_D,X_{\varsigma_D}^{s,x,N},t,y)\bone_{\{s\leq\varsigma_D<r\}}]\d y.
\end{align*}
Since $\{\varsigma_D<r\}\uparrow\{\varsigma_D<t\}$ as $r\uparrow t$, the continuity of probability measures from below and the monotone convergence theorem imply that
\begin{align*}
    \bp(X_t^{s,x,N}\in A,s\leq \varsigma_D<t)&=\int_A\be[p^N(\varsigma_D,X_{\varsigma_D}^{s,x,N},t,y)\bone_{\{s\leq \varsigma_D<t\}}]\d y\\
    &=\int_{A}r_D^N(s,x,t,y)\d y.
\end{align*}
Therefore, we obtain
\begin{align*}
    \bp(X_t^{s,x,N}\in A,t<\varsigma_D)&=\int_{A}p^N(s,x,t,y)\d y-\int_{A}r_D^N(s,x,t,y)\d y\\
    &=\int_{A}p_D^N(s,x,t,y)\d y,
\end{align*}
which shows that \eqref{eq:P(X,tau)} holds. For every fixed $t>s$ and $x\in D$, the function
\[y\mapsto p_D^N(s,x,t,y)=p^N(s,x,t,y)-r_D^N(s,x,t,y)\]
is the density of a sub-probability measure on $\br^d$. Therefore, $p_D^N(s,x,t,y)\geq 0$ for almost every $y\in \br^d$. By the fact that the killed process remains in $D$ before its lifetime, we have $p_D^N(s,x,t,y)=0$
for almost every $y\in D^c$. Since a transition density is determined only up to a Lebesgue-null set, define
\begin{equation}\label{eq:tilde p}
    \tilde{p}_D^N(s,x,t,y):=\bone_D(y)(p^N(s,x,t,y)-r_D^N(s,x,t,y))^+,~~y\in \br^d.
\end{equation}
Then $\tilde{p}_D^N$ is a nonnegative measurable version of the killed transition density and agrees with the function in \eqref{eq:pN(s,x,t,y)} for almost every $y\in\br^d$. In the sequel, we replace $p_D^N$ by this version and continue to denote it by $p_D^N$. For all $t>r>s$ and $A\in\cb(\br^d)$, it follows from the Markov property and Tonelli's theorem that
    \begin{align*}
        \int_Ap_D^N(s,x,t,y)\d y&=\bp(X_t^{s,x,N}\in A,t<\varsigma_D)\\
        &=\be\left[\bp(X_t^{r,X_r^{s,x,N},N}\in A,t<\varsigma_D^{r,X_r^{s,x,N},N})\bone_{\{r<\varsigma_D^{s,x,N}\}}\right]\\
        &=\int_D\left(\int_Ap_D^N(r,z,t,y)\d y\right)p_D^N(s,x,r,z)\d z\\
        &=\int_A\int_Dp_D^N(s,x,r,z)p_D^N(r,z,t,y)\d z\d y.
    \end{align*}
Therefore, \eqref{eq:pN CK} holds for almost every $y\in \br^d$.

For short, denote $v:=4{C_{N,h_0}}^{2/(d+\alpha)}$. Let $h:=t-s{\leq h_0}$ and suppose that for some $c>0$,
\[h\leq c^\alpha\wedge\alpha^\alpha c^\alpha v^\alpha d^{-\alpha},~|x-y|\leq c,~\mbox{dist}(y,D^c)\geq cv.\]
Note that $X_{\varsigma_D}^{s,x,N}\in D^c$ on the event $\{\varsigma_D^{s,x,N}<t\}$, then we have $|X_{\varsigma_D}^{s,x,N}-y|\geq cv$.
Define
\[g(u)=\frac{u}{(u^{1/\alpha}+cv)^{d+\alpha}}.\]
A direct calculation gives
\[g'(u)=\frac{cv-d/\alpha u^{1/\alpha}}{(u^{1/\alpha}+cv)^{d+\alpha+1}}.\]
Then the condition $h\leq\alpha^\alpha c^\alpha v^\alpha d^{-\alpha}$ implies that $g$ is non-decreasing on $[0,h]$. Hence we have
\begin{align*}
r_D^N(s,x,t,y)&=\be[p^N(\varsigma_D,X_{\varsigma_D}^{s,x,N},t,y)\bone_{\{\varsigma_D<t\}}]\\
    &\leq \be\left[\frac{{C_{N,h_0}}(t-\varsigma_D)\bone_{\{\varsigma_D<t\}}}{((t-\varsigma_D)^{1/\alpha}+|X_{\varsigma_D}^{s,x,N}-y|)^{d+\alpha}}\right]\\
    &\leq \frac{{C_{N,h_0}}h}{(h^{1/\alpha}+c v)^{d+\alpha}}.
\end{align*}
Since $h^{1/\alpha}\leq c,~|x-y|\leq c$ and $v=4{C_{N,h_0}^{2/(d+\alpha)}}$, we arrive at
\[h^{1/\alpha}+cv>2{C_{N,h_0}^{2/(d+\alpha)}}(h^{1/\alpha}+|x-y|).\]
Thus, we obtain
\[\frac{(h^{1/\alpha}+cv)^{d+\alpha}}{{C_{N,h_0}^2}(h^{1/\alpha}+|x-y|)^{d+\alpha}}>2^{d+\alpha}.\]
It then follows that
\begin{equation}\label{eq:p_D(s,x,t,y)>0 with |x-y|<a}
    \begin{aligned}
       p_D^N(s,x,t,y)&\geq \frac{{C_{N,h_0}^{-1}}h}{(h^{1/\alpha}+|x-y|)^{d+\alpha}}-\frac{{C_{N,h_0}}h}{(h^{1/\alpha}+c v)^{d+\alpha}}\\
       &=\frac{{C_{N,h_0}} h}{(h^{1/\alpha}+cv)^{d+\alpha}}\left(\frac{(h^{1/\alpha}+cv)^{d+\alpha}}{{C_{N,h_0}^2}(h^{1/\alpha}+|x-y|)^{d+\alpha}}-1\right)\\
       &>\frac{{C_{N,h_0}} h}{(h^{1/\alpha}+cv)^{d+\alpha}}>0.
    \end{aligned}
\end{equation}
We now use a chaining argument to show \eqref{eq:p_D(s,x,t,y)>Lambda}. Let
\[\delta=\mbox{dist}(G,D^c)>0,~~c_0=\frac{\delta}{4v},~~\rho=\frac{c_0}{8}.\]
Choose an integer $n\geq 2$ sufficiently large such that
\[\frac{h_0}{n}\leq c_0^\alpha\wedge\alpha^\alpha c_0^\alpha v^\alpha d^{-\alpha}\]
and
\[\frac{\mbox{diam}(G)}{n}\leq \frac{c_0}{2}.\]
For any $x,y\in G$, define
\[a_i=\left(1-\frac{i}{n}\right)x+\frac{i}{n}y,~~i=0,1,\cdots,n.\]
Since $G$ is convex, we have $a_i\in G$. Set $w_0=x$ and $w_n=y$, and restrict the intermediate variables to
\[w_i\in B(a_i,\rho),~~i=1,\cdots,n-1.\]
For every $i=1,\cdots,n-1$, we have
\[\mbox{dist}(w_i,D^c)\geq \delta-\rho>c_0v,\]
and
\[\mbox{dist}(w_n,D^c)=\mbox{dist}(y,D^c)\geq \delta>c_0v.\]
Moreover, for every $i=1,\cdots,n$, we have
\[|w_i-w_{i-1}|\leq 2\rho+\frac{\mbox{diam}(G)}{n}\leq c_0.\]
Then, by iterating the Chapman-Kolmogorov equation \eqref{eq:pN CK} and \eqref{eq:p_D(s,x,t,y)>0 with |x-y|<a} and applying Tonelli's theorem, we obtain for almost every $y\in G$, 
\begin{align*}
    p_D^N(s,x,s+h_0,y)&=\int_D\cdots\int_D\prod_{i=1}^n p_D^N\left(s+\frac{(i-1)h_0}{n},w_{i-1},s+\frac{ih_0}{n},w_i\right)\d w_1\cdots\d w_{n-1}\\
    &\geq\int_{B(a_1,\rho)}\cdots\int_{B(a_{n-1},\rho)}\prod_{i=1}^n p_D^N\left(s+\frac{(i-1)h_0}{n},w_{i-1},s+\frac{ih_0}{n},w_i\right)\d w_1\cdots\d w_{n-1}\\
    &>\left(\frac{C_Nh_0n^{-1}}{({h_0}^{1/\alpha}n^{-1/\alpha}+c_0 v)^{d+\alpha}}\right)^{n}(\mbox{Leb}(B(0,\rho)))^{n-1}\\
    &=:\Pi>0,
\end{align*}
where $\Pi$ is independent of $s,x$ and $y$.
\end{proof}

\begin{remark}
The transition density $p^N$ can be decomposed into the density of the process killed upon its first exit from $D$ and the contribution of trajectories that exit $D$ before time $t$. More precisely,
\[p^N(s,x,t,y)=p_D^N(s,x,t,y)+r_D^N(s,x,t,y).\]
This representation is known as the Hunt decomposition associated with killing the process upon its first exit from \(D\). The term $r_D^N(s,x,t,y)$ represents the contribution to the full transition density from sample paths that exit $D$ before time $t$ and then continue to evolve from the space point $X_{\varsigma_D}^{s,x,N}$ according to the original transition density $p^N$.

Since $p^N(s,x,t,\cdot)$ is a probability density for any $s<t$ and $x\in\br^d$, Tonelli's theorem gives
\[\int_{\mathbb R^d}r_D^N(s,x,t,y)\mathrm dy=\mathbb P(\varsigma_D<t).\]
Moreover, Theorem \ref{thm:p_D^N(s,x,t,y)} shows that $p_D^N(s,x,t,\cdot)$ is a sub-Markov transition density on $\br^d$ of the process killed upon leaving $D$. Therefore, $p_D^N(s,x,t,y)\geq 0$ for almost every $y\in\br^d$ and $p_D^N(s,x,t,y)=0$ for almost every $y\in D^c$. Furthermore, one can see that
\[\int_{\br^d}p_D^N(s,x,t,y)\d y=\bp(t<\varsigma_D).\]
With this convention \eqref{eq:tilde p}, the Hunt decomposition and the Chapman-Kolmogorov equation are understood in the almost-everywhere sense.
\end{remark}

\begin{theorem}\label{thm:periodic measure of SDEs}
    Suppose that Assumptions \ref{A:b}, \ref{A:sigma} and \ref{A:periodic} hold. Fix any $p\in[1,\alpha)$, and define $V(x)=(1+|x|^2)^\frac{p}{2}$. Then there exists a unique periodic measure $\mu_\cdot$ of SDE \eqref{SDEs} in $\cm_V(\br^d)$, which is continuous in 1-Wasserstein distance. Furthermore, the convergence results of Theorem \ref{thm:exist unique periodic measure} hold.
\end{theorem}

\begin{proof}
For any $t\geq s$, it follows from Theorem \ref{thm:E[|X|]} that there exists a constant $K>0$ such that
\begin{align*}
    P(s,t)V(x)=\be[(1+|X_t^{s,x}|^2)^\frac{p}{2}]\leq \e^{-\varpi(t-s)}V(x)+K,
\end{align*}
which verifies Assumption \ref{cond:Markov process2} with $\vartheta=1$. In particular, we take $t=s+\tau$, then Assumption \ref{cond:Markov process}(i) holds with $\theta=\re^{-\varpi\tau}$. Let $\mathcal{R}>2(1-\theta)^{-1}K$ and $\cc=\{x\in\br^d:V(x)\leq \mathcal{R}\}$. Since $\cc$ is a compact set, we may choose an integer $N>1$ sufficiently large such that $\cc\subset B(0,N-1)$. Let $U_1:=B(0,N)$ and $U_2:=B(0,N-1)$. Applying Theorem \ref{thm:p_D^N(s,x,t,y)} with $D=U_1$, $G=\bar{U}_2=\bar{B}(0,N-1)$ and $h_0=\tau$, we obtain a constant $\Pi>0$, independent of $s$, such that
\[p_{U_1}^N(s,x,s+\tau,y)\geq\Pi\]
for every $x\in\cc$ and almost every $y\in U_2$. Let
\[\varsigma_{U_1}^{s,x}:=\inf\{t\geq s:X_t^{s,x}\notin U_1\}
~~\text{and}~~\varsigma_{U_1}^{s,x,N}:=
    \inf\{t\geq s:X_t^{s,x,N}\notin U_1\}.\]
By the localization argument and pathwise uniqueness, we have 
\[\varsigma_{U_1}^{s,x}=
    \varsigma_{U_1}^{s,x,N}~~\bp-\text{a.s.}~~\text{and}~~X_{t\wedge\varsigma_{U_1}^{s,x}}^{s,x}=X_{t\wedge\varsigma_{U_1}^{s,x,N}}^{s,x,N}.\]
Therefore, for any $A\in\cb(\br^d)$ and $x\in\cc$, we have
    \begin{align*}
        P(s,x,s+\tau,A)&\geq P(s,x,s+\tau,A\cap U_2)\\
        &=\bp(X_{s+\tau}^{s,x}\in A\cap U_2)\\
        &\geq\bp(X_{s+\tau}^{s,x}\in A\cap U_2,s+\tau<\varsigma_{U_1}^{s,x})\\
        &=\bp(X_{s+\tau}^{s,x,N}\in A\cap U_2,s+\tau<\varsigma_{U_1}^{s,x,N})\\
        &=\int_{A\cap U_2}p_{U_1}^N(s,x,s+\tau,y)\d y\\
        &\geq\Pi\mbox{Leb}(A\cap U_2).
    \end{align*}
Define
\[\vp(A):=\frac{\mbox{Leb}(A\cap U_2)}{\mbox{Leb}(U_2)},~~A\in\cb(\br^d),\]
and set
\[\eta:=\frac{1}{2}\min\{1,\Pi\mbox{Leb}(U_2)\},\]
it is not difficult to see that $\eta\in(0,1)$ and $\vp(\cdot)\in\cp(\br^d)$. Then we have
\begin{align*}
    P(s,x,s+\tau,A)&\geq \Pi\mbox{Leb}(U_2)\vp(A)\geq \eta\vp(A).
\end{align*}
Taking the infimum over $x\in\cc$, we obtain
\[\inf_{x\in\cc}P(s,x,s+\tau,A)\geq \eta\vp(A)\]
for any $A\in\cb(\br^d)$ and $s\in\br$. Then all the conditions in Theorem \ref{thm:exist unique periodic measure} are satisfied. Therefore, SDE \eqref{SDEs} has a unique periodic measure $\mu_\cdot$ in $\cm_V(\br^d)$ and possesses the corresponding convergence results stated in Theorem \ref{thm:exist unique periodic measure}.

Finally, we prove the continuity of $t\mapsto\mu_t$ in $\mathcal W_1$. Fix $s_0\in\mathbb R$ and choose $r<s_0$. Restrict $t$ to a sufficiently small neighborhood of $s_0$ such that
$t>r$. For every $n\in\mathbb N$, the triangle inequality gives
\begin{equation}\label{eq:continuity periodic measure}
    \begin{aligned}
        \cw_1(\mu_{s_0},\mu_t)&\leq\cw_1(\mu_{s_0},P(r-n\tau,x,s_0,\cdot))+
    \cw_1(P(r-n\tau,x,t,\cdot),\mu_t)\\
    &\quad+\cw_1(P(r-n\tau,x,s_0,\cdot),P(r-n\tau,x,t,\cdot)).
    \end{aligned}
\end{equation}
Let $\varepsilon>0$, we may first choose $n$ sufficiently large such that
\[\cw_1(\mu_{s_0},P(r-n\tau,x,s_0,\cdot))<\frac{\varepsilon}{3}\]
and
\[\cw_1(P(r-n\tau,x,t,\cdot),\mu_t)<\frac{\varepsilon}{3}\]
uniformly for $t$ in a sufficiently small fixed neighborhood of $s_0$. With this $n$ fixed, the processes $X_{s_0}^{r-n\tau,x}$ and $X_t^{r-n\tau,x}$ provide a coupling of the two transition probabilities. Hence,
\begin{align*}
    \cw_1(P(r-n\tau,x,s_0,\cdot),P(r-n\tau,x,t,\cdot))\leq
    \be[|X_{s_0}^{r-n\tau,x}-X_t^{r-n\tau,x}|].
\end{align*}
Since $n$ is now fixed, Lemma \ref{lem:stochastic continuous} implies that the right-hand side converges to zero as $t\to s_0$. Therefore, there exists $\delta>0$ such that for any $|t-s_0|<\delta$,
\[\cw_1(P(r-n\tau,x,s_0,\cdot),P(r-n\tau,x,t,\cdot))<\frac{\varepsilon}{3}.\]
Substituting these three estimates into
\eqref{eq:continuity periodic measure} yields $\cw_1(\mu_{s_0},\mu_t)<\varepsilon$ for any $|t-s_0|<\delta$. The same argument applies separately when $t\downarrow s_0$ and
when $t\uparrow s_0$, because the initial time $r-n\tau$ is fixed and strictly smaller than both terminal times. Therefore, we have
\[\lim_{t\to s_0}\cw_1(\mu_t,\mu_{s_0})=0.\]
It follows that $t\mapsto\mu_t$ is continuous in the $1$-Wasserstein distance.
\end{proof}

The following corollary is an immediate consequence of the preceding theorem and will play a key role in the error analysis below.
\begin{corollary}\label{cor:W(P(s,x,t,cdot),mut)}
    Suppose that Assumptions \ref{A:b}, \ref{A:sigma} and \ref{A:periodic} hold. Then for every $1\leq p<\alpha$, there exist constants $\bar{C}, \gamma>0$ such that for all $x\in\br^d$ and $t\geq s$,
    \begin{equation*}
        \rho_\beta(P(s,x,t,\cdot),\mu_t)\leq \bar{C}(1+|x|^p)\re^{-\gamma(t-s)}.
    \end{equation*}
\end{corollary}

\section{Discrete periodic measures via the truncated Euler--Maruyama schemes}\label{sec:dpm}

We first introduce a truncation mapping to handle the nonglobally Lipschitz drift. By the uniform local Lipschitz continuity of $b$ in the spatial variable, we choose a strictly increasing continuous function $\vp:[0,\infty)\to(0,\infty)$ such that $\vp(R)\to\infty$ as $R\to\infty$ and, for every $R>0$,
\begin{equation}\label{eq:vp}
    \sup_{t\in\br}\sup_{x\neq y,~\substack{|x|\vee|y|\leq R}}\frac{|b(t,x)-b(t,y)|}{|x-y|}\leq \vp(R),~~x,y\in\br^d.
\end{equation}
In particular, under Assumption \ref{cond:one-side Lip and initial value} (\romannumeral1), we choose $\vp(R):=\ell(1+2R^q)$; see Remark \ref{remark} for details. Choose a finite constant
\[M\geq \vp(1)\vee 1\vee\sup_{t\in\br}|b(t,0)|.\]
For any $\Delta \in(0,1)$, set $R_\Delta:=\vp^{-1}(M\Delta^{-\frac{1}{3}})$, where $\vp^{-1}$ is the inverse function of $\vp$ and defined on $[\vp(0),\infty)$. Since $M\Delta^{-\frac{1}{3}}\geq\vp(1)$, the radius $R_\Delta$ is well defined and satisfies $R_\Delta\geq1$. Define a truncation mapping $\pi_\Delta:\br^d\to\br^d$ by
\begin{equation}\label{eq:pi_Delta}
    \pi_\Delta(x)=\begin{cases}
        x,&|x|\leq R_\Delta,\\
        R_\Delta\frac{x}{|x|}, &|x|>R_\Delta.
    \end{cases}
\end{equation}
As the projection onto the closed ball
$\bar{B}(0,R_\Delta)$, the mapping $\pi_\Delta$ satisfies
\[|\pi_\Delta(x)-\pi_\Delta(y)|\leq|x-y|~~\text{and}~~|\pi_\Delta(x)|\leq|x|\]
for any $x,y\in\br^d$. Therefore, for every $t\in\br$ and $x,y\in\br^d$, it follows that
\begin{equation}\label{eq:b(t,pix)-b(t,piy)}
    |b(t,\pi_\Delta(x))-b(t,\pi_\Delta(y))|\leq M\Delta^{-\frac{1}{3}}|\pi_\Delta(x)-\pi_\Delta(y)|\leq M\Delta^{-\frac{1}{3}}|x-y|,
\end{equation}
and
\begin{equation}\label{eq:b(t,pix)}
    |b(t,\pi_\Delta(x))|\leq M(1+\Delta^{-\frac{1}{3}}|\pi_\Delta(x)|)\leq M(1+\Delta^{-\frac{1}{3}}|x|).
\end{equation}

Fix $s\in\br$ and $x\in\br^d$. To simplify the notation, write $X_t:=X_{s+t}^{s,x}$ for $t\geq 0$ and define
\[b^s(t,x):=b(s+t,x),~~\sigma^s(t,x):=\sigma(s+t,x),~~L_t^s:=L_{s+t}-L_s.\]
Then SDE \eqref{SDEs} takes the form
\begin{equation}
    \d X_t=b^s(t,X_t)\d t+\sigma^s(t,X_{t-})\d L_t^s,~~t\geq0
\end{equation}
with $X_0=x$. Fix a positive integer $N$ such that $\Delta=\frac{\tau}{N}\in(0,1)$. For the numerical approximation, write $\hat{X}_k:=\hat{X}_{s+k\Delta}^{s,\pi_\Delta(x)}$, $\Tilde{X}_k:=\Tilde{X}_{s+k\Delta}^{s,x}$ and $\Delta^sL_k:=L_{(k+1)\Delta}^s-L_{k\Delta}^s$. For each $k\in\bn_0$, the truncated Euler--Maruyama scheme \eqref{TEM} can be rewritten as
\begin{equation}\label{eq:TEM}
    \begin{cases}
        \hat{X}_{k}=\pi_\Delta(\Tilde{X}_{k}),\\
        \Tilde{X}_{k+1}=\hat{X}_k+b^s(k\Delta,\hat{X}_k)\Delta+\sigma^s(k\Delta,\hat{X}_k)\Delta^s L_k
    \end{cases}
\end{equation}
with the initial condition $\tilde{X}_0=x\in\br^d$. In particular, $\hat{X}_0=\pi_\Delta(x)$ and $\hat{X}_k\in\bar{B}(0,R_\Delta)$ for every $k\in\bn_0$. Equivalently,
\[\hat{X}_{k+1}=\pi_\Delta(\hat{X}_k+b^s(k\Delta,\hat{X}_k)\Delta+\sigma^s(k\Delta,\hat{X}_k)\Delta^s L_k),~~\hat{X}_0=\pi_\Delta(x).\]
It follows from \eqref{eq:b(t,pix)} that 
\begin{equation}\label{eq:b(hat X) linear growth}
    |b^s(k\Delta,\hat{X}_k)|\leq M(1+\Delta^{-\frac{1}{3}}|\hat{X}_k|),~~k\in\bn_0.
\end{equation}

For each $i\in\bz$ and $y\in\bar{B}(0,R_\Delta)$, let $\{\hat{X}_{s+j\Delta}^{s+i\Delta,y}\}_{j\geq i}$ denote the time-inhomogeneous Markov chain defined by \eqref{eq:TEM} starting at time $s+i\Delta$ with $\tilde{X}_{s+i\Delta}^{s+i\Delta,y}=\hat{X}_{s+i\Delta}^{s+i\Delta,y}=y$. For an arbitrary $x\in\br^d$, the truncated numerical approximation is written as $\{\hat{X}_{s+j\Delta}^{s+i\Delta,\pi_\Delta(x)}\}_{j\geq i}$. It also satisfies
\[\hat{X}_{s+j\Delta}^{s+i\Delta,\pi_\Delta(x)}=\pi_\Delta\bigl(\tilde{X}_{s+j\Delta}^{s+i\Delta,x}\bigr),~~j\geq i,\]
where $\{\tilde{X}_{s+j\Delta}^{s+i\Delta,x}\}_{j\geq i}$ is the sequence in \eqref{eq:TEM}
with initial value $\tilde{X}_{s+i\Delta}^{s+i\Delta,x}=x$.

For $i,j\in\bz$ with $i\leq j$, $y\in\bar{B}(0,R_\Delta)$ and $A\in\cb(\br^d)$, define the transition kernel by
\begin{equation*}
    \hat{P}(s+i\Delta,y,s+j\Delta,A):=\bp\bigl(\hat{X}_{s+j\Delta}^{s+i\Delta,y}\in A\bigr).
\end{equation*}
For every $x\in\br^d$, extend this notation by projection:
\begin{equation*}
    \hat{P}(s+i\Delta,x,s+j\Delta,A):=\hat{P}(s+i\Delta,\pi_\Delta(x),s+j\Delta,A)=\bp\bigl(\hat{X}_{s+j\Delta}^{s+i\Delta,\pi_\Delta(x)}\in A\bigr).
\end{equation*}
This extension agrees with the transition kernel on $\bar{B}(0,R_\Delta)$, since $\pi_\Delta(y)=y$ for $y\in\bar{B}(0,R_\Delta)$. In particular, its one-step transition probability is
\begin{equation*}
    \hat{P}(s+i\Delta,x,s+(i+1)\Delta,A)=\bp(\pi_\Delta(\pi_\Delta(x)+b^s(i\Delta,\pi_\Delta(x))\Delta+\sigma^s(i\Delta,\pi_\Delta(x))\Delta^sL_i)\in A).
\end{equation*}
It follows that for every $x\in\br^d$ and $i\leq j$,
\begin{equation}\label{eq:P(i,x,i,A)}
    \hat{P}(s+i\Delta,x,s+i\Delta,\cdot)=\delta_{\pi_\Delta(x)},
\end{equation}
and
\begin{equation*}
    \hat{P}(s+i\Delta,x,s+j\Delta,\bar{B}(0,R_\Delta))=1.
\end{equation*}
Therefore, the restriction of $\hat{P}$ to initial states in $\bar{B}(0,R_\Delta)$ is the Markov transition function of the projected chain on $\bar{B}(0,R_\Delta)$. In particular,
\[\hat{P}(s+i\Delta,y,s+i\Delta,\cdot)=\delta_y,~~y\in\bar{B}(0,R_\Delta).\]
For $x\in\br^d$ outside this ball, the notation $\hat{P}(s+i\Delta,x,s+j\Delta,\cdot)$ denotes the distribution obtained by initializing the projected chain at $\pi_\Delta(x)$.

By the independent increments of the $\alpha$-stable process together with the numerical recursion, the Chapman--Kolmogorov equation holds:
\begin{equation}\label{eq:C-K equation}
    \hat{P}(s+i\Delta,x,s+j\Delta,A)=\int_{\bar{B}(0,R_\Delta)}\hat{P}(s+m\Delta,y,s+j\Delta,A)\hat{P}(s+i\Delta,x,s+m\Delta,\d y),
\end{equation}
for every $x\in\br^d$, $A\in\cb(\br^d)$ and $i,m,j\in\bz$ with $i\leq m\leq j$. The cases $m=i$ and $m=j$ follow from \eqref{eq:P(i,x,i,A)} and the identity $\pi_\Delta(y)=y$ on $\bar{B}(0,R_\Delta)$. Since all the transition probabilities are supported on $\bar{B}(0,R_\Delta)$, the integral in \eqref{eq:C-K equation} can equivalently be taken over $\br^d$.

Let $B_b(\br^d)$ denote the space of bounded Borel measurable functions on $\br^d$, and let $\bar{\cm}(\br^d)$ denote the space of finite signed Borel measures on $\br^d$. For $i,j\in\bz$ with $i\leq j$, define the associated Markov operator by
\begin{equation*}\label{eq:P(i,j)f}
    \hat{P}(s+i\Delta,s+j\Delta)f(x):=\int_{\br^d}f(y)\,\hat P(s+i\Delta,x,s+j\Delta,\d y)=\be\bigl[f\bigl(\hat{X}_{s+j\Delta}^{s+i\Delta,\pi_\Delta(x)}\bigr)\bigr]
\end{equation*}
for every $x\in\br^d$. Its dual operator is defined by
\begin{equation}\label{eq:P*(i,j)mu}
    \hat{P}^*(s+i\Delta,s+j\Delta)\mu(A):=\int_{\br^d}\hat{P}(s+i\Delta,x,s+j\Delta,A)\mu(\d x),~~A\in\cb(\br^d).
\end{equation}
In particular, $\hat{P}^*(s+i\Delta,s+j\Delta)$ maps probability measures to probability measures supported on $\bar{B}(0,R_\Delta)$. For every $f\in B_b(\br^d)$ and $\mu\in\bar{\cm}(\br^d)$, Fubini's theorem yields
\begin{equation*}
    \int_{\br^d}f(y)(\hat{P}^*(s+i\Delta,s+j\Delta)\mu)(\d y)=\int_{\br^d}(\hat{P}(s+i\Delta,s+j\Delta)f)(x)\mu(\d x).
\end{equation*}
The integral in \eqref{eq:P*(i,j)mu} is taken over all of $\br^d$ to allow initial measures that assign mass outside $\bar{B}(0,R_\Delta)$.

The Chapman--Kolmogorov equation \eqref{eq:C-K equation} gives
\begin{equation*}
    \hat P(s+i\Delta,s+j\Delta)=\hat P(s+i\Delta,s+m\Delta)\hat P(s+m\Delta,s+j\Delta)
\end{equation*}
and
\begin{equation*}
    \hat P^*(s+i\Delta,s+j\Delta)=\hat P^*(s+m\Delta,s+j\Delta)\hat P^*(s+i\Delta,s+m\Delta),
\end{equation*}
for $i\leq m\leq j$. Moreover, the projected extension satisfies
\begin{equation*}\label{eq:projected extension}
    \hat{P}(s+i\Delta,s+i\Delta)f=f\circ\pi_\Delta,~~\hat {P}^*(s+i\Delta,s+i\Delta)\mu=(\pi_\Delta)_\#\mu,
\end{equation*}
where $(\pi_\Delta)_\#\mu$ is the pushforward measure defined by
\[((\pi_\Delta)_\#\mu)(A):=\mu(\{x\in\br^d:\pi_\Delta(x)\in A\}),~~A\in\cb(\br^d).\]
When the operators are viewed on $B_b(\bar{B}(0,R_\Delta))$ and $\bar{\cm}(\bar{B}(0,R_\Delta))$, respectively, they satisfy the usual identity property
\[\hat{P}(s+i\Delta,s+i\Delta)=I,~~\hat{P}^*(s+i\Delta,s+i\Delta)=I.\]
Under Assumption \ref{A:periodic}, the coefficients $b$ and $\sigma$ are $\tau$-periodic. Together with the stationary increments of the $\alpha$-stable process and the identity $N\Delta=\tau$, this yields
\begin{equation}\label{eq:P(i+tau,x,j+tau,A)}
    \hat{P}(s+i\Delta+\tau,x,s+j\Delta+\tau,A)=\hat{P}(s+i\Delta,x,s+j\Delta,A),~~i\leq j.
\end{equation}
For each $i\in\bz$, define the period kernel on $\bar{B}(0,R_\Delta)$ by
\begin{equation*}
    Q_{s+i\Delta}(y,A):=\hat{P}(s+i\Delta,y,s+i\Delta+\tau,A),~~y\in\bar{B}(0,R_\Delta).
\end{equation*}
Equations \eqref{eq:C-K equation} and \eqref{eq:P(i+tau,x,j+tau,A)} imply that
\begin{equation}
    \hat{P}(s+i\Delta,x,s+i\Delta+k\tau,\cdot)=Q_{s+i\Delta}^k(\pi_\Delta(x),\cdot),~~x\in\br^d,~k\in\bn_0,
\end{equation}
where $Q_{s+i\Delta}^0(y,\cdot)=\delta_y$ for $y\in\bar{B}(0,R_\Delta)$. Thus the long-time behavior of the transition probabilities on the left-hand side is determined by the time-homogeneous Markov kernel $Q_{s+i\Delta}$ on $\bar{B}(0,R_\Delta)$.

\begin{condition}\label{cond:one-side Lip and initial value}
\begin{description}
    \item[(\romannumeral1)] There exist constants $\ell_0,\ell,\ell_b>0$ and $q>0$ such that for any $t\in\br$ and $x,y\in\br^d$,
    \begin{equation}\label{eq:one-side Lip}
        \langle x-y,b(t,x)-b(t,y)\rangle\leq \ell_0|x-y|^2,
    \end{equation}
    and for any $s,t\in\br$ and $x,y\in\br^d$,
    \begin{equation}\label{eq:another loc Lip}
        |b(s,x)-b(t,y)|\leq \ell_b|t-s|^\frac{1}{\alpha}+\ell(1+|x|^{q}+|y|^{q})|x-y|.
    \end{equation}
    \item[(\romannumeral2)] There exist constants $\lambda>0$ and $L>0$ such that for any $t\in\mathbb R$ and $x\in\mathbb R^d$,
    \begin{equation}\label{eq:numerical weak dissipativity}
        \langle x,b(t,x)\rangle \leq L-\lambda|x|^2.
    \end{equation}
    \item [(\romannumeral3)]  There exist constants $\ell_\sigma,C_\sigma>0$ such that for all $s,t\in\br$ and $x,y\in\br^d$,
    \begin{equation}\label{eq:sigma Holder in time}
        \Vert\sigma(s,x)-\sigma(t,y)\Vert\leq \ell_\sigma|t-s|^\frac{1}{\alpha}+C_\sigma|x-y|,
    \end{equation}
    and there exists a constant $\Lambda\geq1$ such that, for all
    $t\in\mathbb R$ and $x,\xi\in\mathbb R^d$,
    \begin{equation}\label{eq:numerical ellipticity}
        \Lambda^{-1}|\xi|^2\leq\langle\sigma(t,x)\sigma(t,x)^\top\xi,\xi\rangle\leq
        \Lambda|\xi|^2.
    \end{equation}
\end{description}
\end{condition}

\begin{remark}\label{remark}
   Under Condition \ref{cond:one-side Lip and initial value}(\romannumeral1), for every $R>0$, $t\in\mathbb R$, and $x,y\in\mathbb R^d$ with $|x|\vee|y|\leq R$, we have
   \[|b(t,x)-b(t,y)|\leq\ell(1+2R^q)|x-y|.\]
   Thus, the function $\vp$ in \eqref{eq:vp} may be chosen as
   \[\vp(R):=\ell(1+2R^q),~R\geq0.\]
   This function is strictly increasing and continuous from $[0,\infty)$ onto $[\ell,\infty)$, with inverse
   \[\vp^{-1}(u)=\left(\frac{u}{2\ell}-\frac{1}{2}\right)^\frac{1}{q},~ u\geq\ell.\]
   By the time continuity and $\tau$-periodicity of $b$, we have $B_0:=\sup_{t\in\mathbb R}|b(t,0)|=\max_{t\in[0,\tau]}|b(t,0)|<\infty$. Setting the two time arguments equal and taking $y=0$ in \eqref{eq:another loc Lip}, we obtain
   \begin{equation}\label{eq:numerical b ply growth}
        \begin{aligned}
            |b(t,x)|&\leq B_0+\ell(1+|x|^q)|x|\leq (M+2\ell)(1+|x|^{q+1}),~~t\in\mathbb R,~x\in\mathbb R^d.
        \end{aligned}
   \end{equation}
   Consequently, \eqref{eq:another loc Lip} together with $B_0<\infty$, implies the polynomial growth requirement \eqref{cond:polynomial growth}, with the admissible choices $l=q+1$ and $C_b=B_0+2\ell$. Moreover, \eqref{cond:loc Lip} holds with $\ell(R):=\max\{1,\ell(1+2R^q)\}$ for $R\ge1$. Together with \eqref{eq:numerical weak dissipativity}, this shows that Condition \ref{cond:one-side Lip and initial value}(i)--(ii) implies Condition \ref{A:b} under the assumed $\tau$-periodicity of $b$.
\end{remark}

Throughout the remainder of the paper, we use the notation
\[\Xi:=(p,\alpha,d,q,C_b,\lambda,L,C_\sigma,\Lambda,\ell_0,\ell,\ell_b,\ell_\sigma,M,\tau)\]
to denote the collection of relevant parameters. The notation $C_{\cdot}$ denotes a positive constant whose dependence is specified by the subscript and whose value may change from one occurrence to another.

\begin{lemma}\label{lem:Lyapunov condition for TEM}
Suppose that Assumption \ref{cond:one-side Lip and initial value} holds. Then, for every $p\in[1,\alpha)$, there exist constants $C_0>0$ and $\Delta^*\in(0,1)$, independent of $s,x,k$ and $\Delta$, such that for every $s\in\br$, $x\in\br^d$, $\Delta\in(0,\Delta^*]$ and $k\in\bn_0$,
\begin{equation}\label{eq:EV(Yi+1) leq gamma V(Yi)+C1 Delta t}
    \be[V(\hat{X}_{k+1})|\cf_k]\leq \bar{\gamma} V(\hat{X}_k)+C_0\Delta,
\end{equation}
where $V(x):=(1+|x|^2)^\frac{p}{2}$ and $\bar{\gamma}=(1-2^{-1}p\lambda\Delta)\in[1/2,1)$. Moreover, for every $n\in\bn_0$,
\begin{equation}\label{eq:E|Y(s+i Delta t)| leq C(1+e-beta Delta t)}
    \be[V(\hat{X}_n)]\leq \re^{-\hat{\varpi} n\Delta}V(x)+\frac{2C_0}{p\lambda},
\end{equation}
where $\hat{\varpi}=-\frac{\ln(1-2^{-1}p\lambda\Delta)}{\Delta}>0$. 
\end{lemma}

\begin{proof}
Fix $p\in[1,\alpha)$ and set $\Delta^*:=\min\bigl\{2^{-1},(p\lambda)^{-1},(4^{-1}M^{-2}\lambda)^3\bigr\}$. Then $\Delta^*\in(0,1)$, and for every $\Delta\in(0,\Delta^*]$, $\bar\gamma:=1-2^{-1}p\lambda\Delta\in[1/2,1)$ and $2M^2\Delta^{1/3}\leq 2^{-1}\lambda$. Recall that $\hat X_0=\pi_\Delta(x)$, and hence $|\hat{X}_k|\leq R_\Delta$ for each $k\in\bn_0$.

Fix $s\in\br$, $\Delta\in(0,\Delta^*]$ and $k\in\bn_0$. For a deterministic value $y\in\bar{B}(0,R_\Delta)$, define
\[G_k(y):=\be[V(\pi_\Delta(y+\Delta b^s(k\Delta,y)+\sigma^s(k\Delta,y)\Delta^s L_\Delta))].\]
We will prove that
\[G_k(y)\leq \bar{\gamma}V(y)+C_0\Delta~~\text{for every}~y\in\bar{B}(0,R_\Delta).\]
Since $|y|\leq R_\Delta$, the definition of $\vp$ and the choice of $M$ give
    \begin{equation*}
        |b^s(k\Delta,y)|\leq|b^s(k\Delta,0)|+\vp(R_\Delta)|y|\leq M+M\Delta^{-\frac{1}{3}}|y|,
    \end{equation*}
For every $r\in[0,\Delta]$, it follows that
    \[r|b^s(k\Delta,y)|^2\leq 2M^2\Delta+2M^2\Delta^\frac{1}{3}|y|^2\leq 2M^2\Delta+\frac{\lambda}{2}|y|^2.\]
    By the weakly dissipative assumption, we have
    \[\langle y, b^s(k\Delta,y)\rangle+r|b^s(k\Delta,y)|^2\leq -\frac{\lambda}{2}|y|^2+L+2M^2\Delta.\]
    Moreover,
    \begin{align*}
        1+|y+rb^s(k\Delta,y)|^2&=1+|y|^2+2r\langle y,b^s(k\Delta,y)\rangle+r^2|b^s(k\Delta,y)|^2\\
        &\leq 1+|y|^2-\frac{3\lambda}{2}r|y|^2+2L r+2M^2r\Delta\\
        &\leq 1+|y|^2+2L\Delta+2M^2\Delta^2\\
        &=:D_\Delta(y).
    \end{align*}
    The fundamental theorem of calculus yields
    \begin{align*}
        V(y+\Delta b^s(k\Delta,y))&=V(y)+p\int_0^{\Delta}\frac{\langle y,b^s(k\Delta,y)\rangle+r|b^s(k\Delta,y)|^2}{(1+|y+rb^s(k\Delta,y)|^2)^{1-\frac{p}{2}}}\d r\\
        &\leq V(y)-\frac{p\lambda}{2}\Delta |y|^2D_\Delta(y)^{\frac{p}{2}-1}+p(L+2M^2\Delta)\Delta,
    \end{align*}
    where we used $1-2^{-1}p>0$ and $1\leq 1+|y+rb^s(k\Delta,y)|^2\leq D_\Delta(y)$.
    Since $|y|^2=D_\Delta(y)-1-2L\Delta-2M^2\Delta^2$, it follows that
    \begin{align*}
        V(y+\Delta b^s(k\Delta,y))&\leq V(y)-\frac{p\lambda}{2}\Delta D_\Delta(y)^\frac{p}{2}+p(L+2M^2\Delta)\Delta\\
        &~~~+\frac{p\lambda}{2}\Delta(1+2L\Delta+2M^2\Delta^2)D_\Delta(y)^{\frac{p}{2}-1}\\
        &\leq \left(1-\frac{p\lambda}{2}\Delta\right)V(y)+C_{p,\lambda,L,M}\Delta,
    \end{align*}
    where
    $C_{p,\lambda,L,M}:=2^{-1}p\lambda(1+2L+2M^2)+p(L+2M^2)$. In the last step, we used $D_\Delta(y)\geq 1+|y|^2$, $D_\Delta(y)^{\frac{p}{2}-1}\leq 1$ and $\Delta<1$. Next, set $A:=\sigma^s(k\Delta,y)$ and consider the auxiliary process
\begin{equation*}
    Y_t^{k,y}:=y+b^s(k\Delta,y)\Delta+A\tilde{L}_t^s,~~t\geq 0.
\end{equation*}
Let $\mathscr{L}_A$ denote the generator with the fixed matrix $A$, namely,
\[\mathscr{L}_Af(w):=\int_{\br^d}[f(w+Az)-f(w)-\bone_{\{|z|\leq 1\}}\langle\nabla f(w),Az\rangle]\nu(\d z).\]
By \cite[Lemma A.1]{CXZZ26} and the uniform upper bound in \eqref{eq:numerical ellipticity},
\[\mathscr{L}_A V(w)\leq C_{p,\alpha,d}\Vert A\Vert^\alpha(1+|w|^2)^{\frac{p}{2}-\frac{\alpha}{2}}\leq C_{p,\alpha,d,\Lambda},~~w\in\br^d.\]
To justify the use of It\^o's formula, set
\[\varsigma_m^y:=\inf\{t\geq 0:|Y_t^{k,y}|\geq m\},~~m>|y+b^s(k\Delta,y)\Delta|.\]
For any $0\leq t\leq \Delta$, it follows that
\begin{align*}
    \be[V(Y_{t\wedge\varsigma_m^y}^{k,y})]&=V(Y_0^{k,y})+\be\left[\int_0^{t\wedge\varsigma_m^y}\mathscr{L}_AV(Y_{r-}^{k,y})\d r\right]\leq V(Y_0^{k,y})+C_{p,\alpha,d,\Lambda} t.
\end{align*}
Since $Y^{k,y}$ has c\`adl\`ag paths, they are bounded on every compact time interval, so $\varsigma_m^y\to\infty$ almost surely as $m\to\infty$. Fatou's lemma gives
\[\be[V(Y_\Delta^{k,y})]\leq V(y+b^s(k\Delta,y)\Delta)+C_{p,\alpha,d,\Lambda}\Delta.\]
Using $V(\pi_\Delta(w))\leq V(w)$, we conclude that
\begin{equation*}
    G_k(y)=\be[V(\pi_\Delta(Y_\Delta^{k,y}))]\leq \bar{\gamma}V(y)+C_0\Delta,
\end{equation*}
where $C_0=C_{p,\lambda,L,M}+C_{p,\alpha,d,\Lambda}$.

For every $k\in\mathbb N_0$, the increment $\Delta^sL_k$ is independent of $\cf_k$ and has the same distribution as $\Delta^s L_\Delta$. Since $\hat{X}_k$ is $\cf_k$-measurable and takes values in $\bar{B}(0,R_\Delta)$, we have
\[\be[V(\hat{X}_{k+1})|\cf_k]=G_k(\hat{X}_k)\leq \bar{\gamma}V(\hat{X}_k)+C_0\Delta.\]
This estimate also holds for $k=0$, because $\hat{X}_0=\pi_\Delta(x)$ for every $x\in\br^d$. Taking expectations and iterating yields, for every $n\in\mathbb N_0$,
\begin{align*}
    \be[V(\hat{X}_n)]&\leq \bar{\gamma}^nV(\pi_\Delta(x))+C_0\Delta\sum_{i=0}^{n-1}\bar{\gamma}^i=\bar{\gamma}^nV(\pi_\Delta(x))+\frac{2C_0}{p\lambda}(1-\bar{\gamma}^n)\leq \bar{\gamma}^nV(x)+\frac{2C_0}{p\lambda},
\end{align*}
where the sum is understood to be zero when $n=0$. Taking $\hat{\varpi}=-\frac{\ln(1-2^{-1}p\lambda\Delta)}{\Delta}>0$, we have
\begin{equation*}
    \be[V(\hat{X}_n)]\leq \re^{-\hat{\varpi} n\Delta}V(x)+\frac{2C_0}{p\lambda}.
\end{equation*}
The constants $C_0$ and $\Delta^*$ are independent of $s,x,k$, and $\Delta$.
\end{proof}

\begin{remark}
    The constant $\hat{\varpi}$ given in Lemma \ref{lem:Lyapunov condition for TEM} converges to  $2^{-1}p\lambda$ as $\Delta$ tends to $0$.
\end{remark}

\begin{lemma}\label{lem:local Doeblin condition for TEM}
Suppose that Assumption \ref{cond:one-side Lip and initial value} holds. Let $C_0$ be the constant in Lemma \ref{lem:Lyapunov condition for TEM}. Fix $\hat{\mathcal{R}}>4C_0p^{-1}\lambda^{-1}$, independent of $s$, $i$ and $\Delta$. Set $V(x):=(1+|x|^2)^\frac{p}{2}$ and $\cc:=\{x\in\br^d:V(x)\leq \hat{\mathcal{R}}\}$. Then there exist constants $\hat{\eta}\in(0,1)$ and  $\Delta^*\in(0,1)$, and a probability measure $\hat{\vp}$ on $\br^d$ such that for every $\Delta\in(0,\Delta^*)$, $s\in\br$, $i\in\bn_0$ and $A\in\cb(\br^d)$,
\begin{align}\label{eq:d local Doeblin}
    \inf_{x\in \hat{\cc}}\hat{P}(s+i\Delta,x,s+(i+1)\Delta,A)\geq\hat{\eta}\Delta\hat{\vp}(A),
\end{align}
The constants and the measure are independent of $s$, $i$ and $\Delta$.
\end{lemma}

\begin{proof}
Choose $\Delta^*\in(0,1)$ no larger than the step-size threshold in Lemma \ref{lem:Lyapunov condition for TEM}. Fix $\Delta\in(0,\Delta^*)$, $s\in\mathbb R$ and $i\in\mathbb N_0$. For each $x\in\br^d$, set $A_x:=\sigma^s(i\Delta,{\pi_\Delta(x)})$ and $Y_x:={\pi_\Delta(x)}+\Delta b^s(i\Delta,{\pi_\Delta(x)})+A_xL_\Delta$.
Since $\Delta^s L_i$ has the same distribution as $L_\Delta$, the definition of the numerical scheme gives
\[\hat{P}(s+i\Delta,x,s+(i+1)\Delta,A)=\bp(\pi_{\Delta}(Y_x)\in A),~~A\in\cb(\br^d).\]
Let $p_\alpha(t,w)$ denote the density of $L_t$. It follows from \cite[(A.9)]{CXZZ26} that there exists a constant $C_{\alpha,d}\geq 1$ such that
\begin{align*}
    \frac{C_{\alpha,d}^{-1}t}{(t^\frac{1}{\alpha}+|w|)^{d+\alpha}}\leq p_\alpha(t,w)\leq\frac{C_{\alpha,d}t}{(t^\frac{1}{\alpha}+|w|)^{d+\alpha}},~~t>0,~w\in\br^d.
\end{align*}
The uniform ellipticity assumption implies that $A_x$ is invertible. A change of variables then shows that $Y_x$ has density
\[q_{s,i,\Delta,x}(y):=|\det(A_x^{-1})|p_\alpha(\Delta,A_x^{-1}(y-{\pi_\Delta(x)}-\Delta b^s(i\Delta,{\pi_\Delta(x)}))),~~y\in\br^d.\]
For each $s\in\br$, $i\in\bn_0$, $x\in\br^d$ and every Borel set $D\in \cb(B(0,R_\Delta))$, we have
\begin{align*}
    \hat{P}(s+i\Delta,x,s+(i+1)\Delta,D)&=\bp(\hat{X}_{s+(i+1)\Delta}^{s+i\Delta,{\pi_\Delta(x)}}\in D)=\bp(\pi_{\Delta}(\tilde{X}_{s+(i+1)\Delta}^{s+i\Delta,x})\in D).
\end{align*}
Note that
\[\{\omega: \tilde{X}_{s+(i+1)\Delta}^{s+i\Delta,x}\in D\}\subseteq\{\omega: \pi_\Delta(\tilde{X}_{s+(i+1)\Delta}^{s+i\Delta,x})\in D\}\]
and
\[\{\omega: \tilde{X}_{s+(i+1)\Delta}^{s+i\Delta,x}\notin D\}\cap\{\omega: \pi_\Delta(\tilde{X}_{s+(i+1)\Delta}^{s+i\Delta,x})\in D\}=\emptyset.\]
It follows that
\[\bp(\pi_{\Delta}(\Tilde{X}_{s+(i+1)\Delta}^{s+i\Delta,x})\in D)=\bp(\Tilde{X}_{s+(i+1)\Delta}^{s+i\Delta,x}\in D).\]
Therefore, we derive that
\begin{align*}
    \hat{P}(s+i\Delta,x,s+(i+1)\Delta,D)&=\bp(Y_x\in D)=\int_Dq_{s,i,\Delta,x}(y)\d y.
\end{align*}
Thus, the restriction of the transition probability measure $\hat P(s+i\Delta,x,s+(i+1)\Delta,\cdot)$ to $B(0,R_\Delta)$ admits the density $q_{s,i,\Delta,x}$ with respect to Lebesgue measure. Since $M\geq \vp(1)$ and $\Delta<1$, we have $R_\Delta>1$ and $B(0,1)\subset B(0,R_\Delta)$. It follows that for every $A\in\cb(\br^d)$,
\begin{align*}
    \hat{P}(s+i\Delta,x,s+(i+1)\Delta,A)&\geq \hat{P}(s+i\Delta,x,s+(i+1)\Delta, A\cap B(0,1))\\
    &=\int_{A\cap B(0,1)}q_{s,i,\Delta,x}(y)\d y.
\end{align*}
For every $x\in\hat{\cc}$, the inequality $V(\pi_\Delta(x))\leq V(x)$ implies $\pi_\Delta(x)\in\hat{\cc}$ and $|\pi_\Delta(x)|\leq|x|\leq\hat{\mathcal{R}}$. For every $t\in\br$ and $x\in\hat{\cc}$, we have
\[|b(t,\pi_\Delta(x))|\leq M+\ell(1+|\pi_\Delta(x)|^q)|\pi_\Delta(x)|\leq M+\ell(1+|\hat{\mathcal{R}}|^q)\hat{\mathcal{R}}=:B_*.\]
Consequently, we have $\sup_{t\in\br,x\in\hat{\cc}}|b(t,\pi_\Delta(x))|\leq B_*$. The uniform ellipticity assumption gives $|\det(A_x^{-1})|
\geq\Lambda^{-\frac{d}{2}}$ and $\Vert A_x^{-1}\Vert\leq\sqrt{d\Lambda}$. Hence, for every $x\in\hat{\cc}$ and $y\in B(0,1)$, it follows that
\begin{align*}
    |A_x^{-1}(y-\pi_\Delta(x)-b^s(i\Delta,\pi_\Delta(x))\Delta)|&\leq \sqrt{d\Lambda}(|y|+|\pi_\Delta(x)|+\Delta|b^s(i\Delta,\pi_\Delta(x))|)\\
    &\leq \sqrt{d\Lambda}(1+\hat{\mathcal{R}}+B_*)=:C_*.
\end{align*}
The constants $B_*$ and $C_*$ are independent of $s,i,\Delta,x$ and $y$ in the ranges specified above. Applying the lower bound for $p_\alpha$ yields
\begin{align*}
    q_{s,i,\Delta,x}(y)&\geq \frac{C_{\alpha,d}^{-1}\Lambda^{-\frac{d}{2}}\Delta}{(\Delta^\frac{1}{\alpha}+C_*)^{d+\alpha}}\geq \frac{C_{\alpha,d}^{-1}\Lambda^{-\frac{d}{2}}}{(1+C_*)^{d+\alpha}}\Delta=:c_0\Delta,~~x\in\hat{\cc},~y\in B(0,1),
\end{align*}
where $c_0>0$ is independent of $s,i,\Delta,x$ and $y$. Define
\[\hat{\vp}(A):=\frac{\mbox{Leb}(A\cap B(0,1))}{\mbox{Leb}(B(0,1))},~~A\in\cb(\br^d),\]
and set
\[\hat{\eta}:=\frac{1}{2}\min\{1,c_0\mbox{Leb}(B(0,1))\}.\]
Then $\hat{\vp}$ is a probability measure and $\hat{\eta}\in(0,1)$, and both are independent of $s,i$ and $\Delta$. Therefore, for every $A\in\cb(\br^d),x\in\hat{\cc},s\in\br$, $i\in\bn_0$ and all $\Delta\in(0,\Delta^*)$, we deduce that 
\[\hat{P}(s+i\Delta,x,s+(i+1)\Delta,A)\geq \int_{A\cap B(0,1)}q_{s,i,\Delta,x}(y)\d y\geq c_0\Delta\mbox{Leb}(A\cap B(0,1))\geq \hat{\eta}\Delta\hat{\vp}(A).\]
Taking the infimum over $x\in\hat{\cc}$ completes the proof.
\end{proof}

\begin{definition}
    A measure-valued function $\hat{\mu}_{s+\cdot\Delta}:\bz\to\cp(\br^d)$ is called a $N$-periodic measure of the discrete semigroup $\hat{P}(s+i\Delta,s+j\Delta)$ if 
    \[\hat{P}^*(s+i\Delta,s+j\Delta)\hat{\mu}_{s+i\Delta}(A):=\int_{\br^d}\hat{P}(s+i\Delta,x,s+j\Delta,A)\hat{\mu}_{s+i\Delta}(\d x)=\hat{\mu}_{s+j\Delta}(A),\]
    and
    \[\hat{\mu}_{s+i\Delta+N\Delta}=\hat{\mu}_{s+i\Delta},\]
    for every $i,j\in\bz$ with $i\leq j$.
\end{definition}

\begin{theorem}\label{thm:derect pm}
    Suppose that Assumptions \ref{cond:one-side Lip and initial value} and \ref{A:periodic} hold. Then, for every $p\in[1,\alpha)$, there exist constants $\Delta^*\in(0,1)$, $\hat{\gamma}>0$ and $C_\Xi>0$ such that, for every $N\in\mathbb N$ satisfying $\Delta:=\frac{\tau}{N}\in(0,\Delta^*]$ and every $s\in\mathbb R$, the truncated Euler--Maruyama scheme \eqref{eq:TEM} admits a unique $N$-periodic measure
    \[\hat{\mu}_{s+\cdot\Delta}:\bz\longrightarrow\cp_V(\br^d),~~\sup_{i\in\bz}\hat{\mu}_{s+i\Delta}(V)<\infty,\]
    where $V(x)=(1+|x|^2)^\frac{p}{2}$. For each fixed $\Delta$, this periodic measure is independent of the initial value $x$. Moreover, for every $k\in\bn_0$, $i\in\bz$ and $x\in\br^d$,
    \[\rho_\beta(\hat{P}(s-k\tau+i\Delta,x,s+i\Delta,\cdot),\hat{\mu}_{s+i\Delta})\leq C_\Xi(1+|x|^p)\re^{-\hat{\gamma}k\tau}.\]
    The constants $\beta$, $\Delta^*$, $\hat\gamma$ and $C_\Xi$ are independent of $s,x,N,\Delta,i$ and $k$.
\end{theorem}

\begin{proof}
Let $C_0$ be the constant in Lemma \ref{lem:Lyapunov condition for TEM}. Fix
\[\hat{\mathcal{R}}:=1+\frac{4C_0}{p\lambda}>\frac{4C_0}{p\lambda}\text{ and } \hat{\cc}:=\{x\in\br^d:V(x)\leq \hat{\mathcal{R}}\}.\]
Apply the construction in the proof of Lemma \ref{lem:local Doeblin condition for TEM} to this level set, and denote the resulting minorization constant and probability measure by $\hat{\eta}$ and $\hat{\vp}$, respectively. Choose a sufficiently small $\Delta^*$, fix $N\in\bn$ such that $\Delta:=\frac{\tau}{N}\in(0,\Delta^*]$. For each $s\in\br$ and $i\in\bn_0$, by Lemmas \ref{lem:Lyapunov condition for TEM} and \ref{lem:local Doeblin condition for TEM}, we have
\begin{equation}\label{eq:Lyp proof}
    \hat{P}(s+i\Delta,s+(i+1)\Delta)V(x)\leq \Bigl(1-\frac{p\lambda}{2}\Delta\Bigr)V(x)+C_0\Delta,~~x\in\br^d
\end{equation}
and
\[\inf_{x\in\hat{\cc}}\hat{P}(s+i\Delta,x,s+(i+1)\Delta,A)\geq\hat{\eta}\Delta \hat{\vp}(A),~~x\in\hat{\cc},~A\in\cb(\br^d).\]
Set $\beta:=\frac{\hat{\eta}\Delta}{2C_0\Delta}=\frac{\hat{\eta}}{2C_0}>0$, which is independent of $\Delta$. For each $i\in\bn_0$, applying the quantitative contraction estimate in \cite[Theorem 1.3]{HM11} gives
\begin{equation*}
    \rho_\beta(\hat{P}^*(s+i\Delta,s+(i+1)\Delta)\mu_1,\hat{P}^*(s+i\Delta,s+(i+1)\Delta)\mu_2)\leq \hat{\bar{\eta}}\rho_\beta(\mu_1,\mu_2),
\end{equation*}
where
\[\hat{\bar{\eta}}:=\max\left\{1-\frac{\hat{\eta}}{2}\Delta,1-\frac{\beta(p\lambda\hat{\mathcal{R}}-4C_0)}{2(2+\beta\hat{\mathcal{R}})}\Delta\right\}.\]
Set
\[\hat{\gamma}:=\min\left\{\frac{\hat{\eta}}{2},\frac{\beta(p\lambda\hat{\mathcal{R}}-4C_0)}{2(2+2\beta\hat{\mathcal{R}})}\right\}>0.\]
Then $\hat{\gamma}$ is independent of $\Delta$, and $\hat{\bar{\eta}}\leq 1-\hat{\gamma}\Delta\leq \re^{-\hat{\gamma}\Delta}\in(0,1)$. By the Chapman--Kolmogorov equation and iteration, we have
\begin{equation*}
    \rho_\beta(\hat{P}^*(s+i\Delta,s+j\Delta)\mu_1,\hat{P}^*(s+i\Delta,s+j\Delta)\mu_2)\leq \re^{-\hat{\gamma}(j-i)\Delta}\rho_\beta(\mu_1,\mu_2),~~0\leq i<j
\end{equation*}
for all $\mu_1,\mu_2\in\cp_V(\br^d)$. We now also verify this estimate when $i=j$. Since 
\[\hat{P}(s+i\Delta,x,s+i\Delta,\cdot)=\delta_{\pi_\Delta(x)}(\cdot),\]
we have $\hat{P}^*(s+i\Delta,s+i\Delta)\mu=(\pi_\Delta)_\#\mu$. The total variation of the signed measure satisfies
\[|(\pi_\Delta)_\#(\mu_1-\mu_2)|\leq (\pi_\Delta)_\#|\mu_1-\mu_2|.\]
Therefore, using $V(\pi_\Delta(x))\leq V(x)$, we have
\begin{align*}
    \rho_\beta(\hat{P}^*(s+i\Delta,s+i\Delta)\mu_1,\hat{P}^*(s+i\Delta,s+i\Delta)\mu_2)&\leq\int_{\br^d}(1+\beta V(\pi_\Delta(x)))|\mu_1-\mu_2|(\d x)\\
    &\leq \rho_\beta(\mu_1,\mu_2).
\end{align*}
Thus for all $i,j\in\bn_0$ with $i\leq j$, we conclude that
\begin{equation}\label{eq:rho(i,j) i,j geq0}
    \rho_\beta(\hat{P}^*(s+i\Delta,s+j\Delta)\mu_1,\hat{P}^*(s+i\Delta,s+j\Delta)\mu_2)\leq \re^{-\hat{\gamma}(j-i)\Delta}\rho_\beta(\mu_1,\mu_2).
\end{equation}
We next establish the corresponding estimate for arbitrary integer indices by using periodicity. The periodicity of $b$ and $\sigma$, the stationary increments of the $\alpha$-stable process, and the numerical recursion imply that
\begin{equation}\label{eq:kernel periodicity proof}
    \hat{P}(s+(i+mN)\Delta,x,s+(j+mN)\Delta,A)=\hat{P}(s+i\Delta,x,s+j\Delta,A)
\end{equation}
for all $i,j,m\in\bz$ with $i\leq j$, $x\in\br^d$ and $A\in\cb(\br^d)$. Fix any $i,j\in\bz$ with $i\leq j$. Choose $m\in\bn_0$ such that $i+mN\geq0$, then $0\leq i+mN\leq j+mN$. Integrating \eqref{eq:kernel periodicity proof} against an initial measure gives
\[\hat{P}^*(s+i\Delta,s+j\Delta)=\hat{P}^*(s+(i+mN)\Delta,s+(j+mN)\Delta).\]
We apply \eqref{eq:rho(i,j) i,j geq0} with the nonnegative
indices $i+mN$ and $j+mN$ to obtain
\begin{equation}\label{eq:rho(i,j)}
\begin{aligned}
    &\rho_\beta(\hat{P}^*(s+i\Delta,s+j\Delta)\mu_1,\hat{P}^*(s+i\Delta,s+j\Delta)\mu_2)\\
    =&~\rho_\beta(\hat{P}^*(s+(i+mN)\Delta,s+(j+mN)\Delta)\mu_1,\hat{P}^*(s+(i+mN)\Delta,s+(j+mN)\Delta)\mu_2)\\
    \leq&~\re^{-\hat{\gamma}((j+mN)-(i+mN))\Delta}\rho_\beta(\mu_1,\mu_2)\\
    =&~\re^{-\hat{\gamma}(j-i)\Delta}\rho_\beta(\mu_1,\mu_2).
\end{aligned}
\end{equation}
Thus the contraction estimate holds for all integers $i\leq j$.

We next construct the periodic measure. It is not difficult to see that the metric space $(\cp_V(\br^d),\rho_\beta)$ is complete. By \eqref{eq:Lyp proof}, $\hat{P}^*(s,s+\tau)$ maps $\cp_V(\br^d)$ into itself. By \eqref{eq:rho(i,j)}, it is a strict contraction with factor at most $\re^{-\hat{\gamma}\tau}$. The Banach fixed-point theorem gives a unique $\hat{\mu}_s\in\cp_V(\br^d)$ such that
\[\hat{P}^*(s,s+N\Delta)\hat{\mu}_s=\hat{\mu}_s.\]
Since every transition probability is supported in $\bar{B}(0,R_\Delta)$, we obtain $\hat{\mu}_s(\bar{B}(0,R_\Delta))=1$. Moreover, we have $\hat{P}^*(s,s)\hat{\mu}_s=(\pi_\Delta)_\#\hat{\mu}_s=\hat{\mu}_s$. We therefore define
\[\hat{\mu}_{s+i\Delta}:=\hat{P}^*(s,s+i\Delta)\hat{\mu}_s,~~i\in\bn_0.\]
The Chapman--Kolmogorov equation gives
\[\hat{P}^*(s+i\Delta,s+j\Delta)\hat{\mu}_{s+i\Delta}=\hat{\mu}_{s+j\Delta},~~0\leq i\leq j.\]
Furthermore, for $i\in\bn_0$, it follows from \eqref{eq:kernel periodicity proof} that
\begin{align*}
    \hat{\mu}_{s+(i+N)\Delta}&=
    \hat P^*(s+N\Delta,s+(i+N)\Delta)
    \hat P^*(s,s+N\Delta)\hat\mu_s\\
    &=\hat P^*(s+N\Delta,s+(i+N)\Delta)\hat\mu_s\\
    &=\hat P^*(s,s+i\Delta)\hat\mu_s=\hat\mu_{s+i\Delta}.
\end{align*}
For $i<0$, we can write uniquely $i=mN+r$ with $m\in\bz$ and $r\in\{0,\cdots,N-1\}$, and define
\[\hat{\mu}_{s+i\Delta}:=\hat{\mu}_{s+r\Delta}.\]
Together with the periodicity already proved for nonnegative indices, this definition gives
\[\hat{\mu}_{s+(i+N)\Delta}=\hat{\mu}_{s+i\Delta},~~i\in\bz.\]
For arbitrary integers $i\leq j$, choose $m\in\bn_0$ such that $i+mN\geq 0$. It follows that
\begin{align*}
    \hat{P}^*(s+i\Delta,s+j\Delta)\hat{\mu}_{s+i\Delta}&=\hat{P}^*(s+(i+mN)\Delta,s+(j+mN)\Delta)\hat{\mu}_{s+(i+mN)\Delta}\\
    &=\hat{\mu}_{s+(j+mN)\Delta}=\hat{\mu}_{s+j\Delta}.
\end{align*}
Thus $\hat{\mu}_{s+\cdot\Delta}$ is an $N$-periodic measure, and each of its components is supported on $\bar{B}(0,R_\Delta)$.

This periodic measure is unique. Indeed, let $\hat{\nu}_{s+\cdot\Delta}$ be any other $N$-periodic measure. For every $i\in\bz$, we have
\[\hat{\nu}_{s+i\Delta}=\hat{P}^*(s+(i-1)\Delta,s+i\Delta)\hat{\nu}_{s+(i-1)\Delta}.\]
Hence, for every $i\in\bz$, $\hat{\nu}_{s+i\Delta}$ is supported on $\bar{B}(0,R_\Delta)$ and belongs $\cp_V(\br^d)$. Periodicity gives
\[\hat{P}^*(s,s+N\Delta)\hat{\nu}_s=\hat{\nu}_{s+N\Delta}=\hat{\nu}_s.\]
The uniqueness of the fixed point implies $\hat{\nu}_s=\hat{\mu}_s$, and hence $\hat{\nu}_{s+i\Delta}=\hat{\mu}_{s+i\Delta}$ for every $i\in\mathbb Z$.

We now verify that this periodic measure is independent of
the initial value $x$. Fix $i\in\bz$ and define the period
operator by
\[\mathcal{T}_{s+i\Delta}:=\hat P^*(s+i\Delta,s+i\Delta+\tau).\]
The operator $\mathcal T_{s+i\Delta}$ does not involve a choice of $x$. By \eqref{eq:rho(i,j)}, we have
\begin{equation}\label{eq:period contraction independence}
    \rho_\beta(\mathcal{T}_{s+i\Delta}\mu_1,\mathcal{T}_{s+i\Delta}\mu_2)\leq\re^{-\hat{\gamma}\tau}\rho_\beta(\mu_1,\mu_2).
\end{equation}
The evolution relation and periodicity imply
\[\mathcal{T}_{s+i\Delta}\hat{\mu}_{s+i\Delta}
    =\hat{\mu}_{s+i\Delta+\tau}=\hat{\mu}_{s+i\Delta}.\]
If $\hat{\nu}\in\mathcal P_V(\mathbb R^d)$ is another fixed point, then \eqref{eq:period contraction independence} gives
\[\rho_\beta(\hat{\nu},\hat\mu_{s+i\Delta})\leq\re^{-\hat\gamma\tau}\rho_\beta(\hat{\nu},\hat\mu_{s+i\Delta}).\]
The distance is finite and $\mathrm e^{-\hat\gamma\tau}<1$, so
$\hat{\nu}=\hat{\mu}_{s+i\Delta}$. For any $x\in\mathbb R^d$, the initial law of the projected chain is $\delta_{\pi_\Delta(x)}$. By periodicity of the transition kernels and the Chapman--Kolmogorov equation, we have
\begin{equation}\label{eq:period iterates independence}
    \hat{P}(s+i\Delta,x,s+i\Delta+n\tau,\cdot)=\mathcal{T}_{s+i\Delta}^n\delta_{\pi_\Delta(x)},~~n\in\bn_0,
\end{equation}
where $\mathcal T_{s+i\Delta}^{\,0}=I$. Therefore, we conclude that
\begin{align*}
    \rho_\beta(\hat{P}(s+i\Delta,x,s+i\Delta+n\tau,\cdot),\hat{\mu}_{s+i\Delta})&\leq\re^{-\hat{\gamma} n\tau}\rho_\beta(\delta_{\pi_\Delta(x)},\hat\mu_{s+i\Delta})\\
    &\leq\re^{-\hat{\gamma}n\tau}(2+\beta V(\pi_\Delta(x))+\beta\hat{\mu}_{s+i\Delta}(V))\to0.
\end{align*}
Thus every initial value $x$ gives the same limit $\hat\mu_{s+i\Delta}$. More explicitly, for any two initial values $x,y\in\br^d$,
\begin{align*}
    \rho_\beta(\mathcal T_{s+i\Delta}^{\,n}\delta_{\pi_\Delta(x)},\mathcal{T}_{s+i\Delta}^{\,n}\delta_{\pi_\Delta(y)})\leq\re^{-\hat{\gamma}n\tau}\rho_\beta(\delta_{\pi_\Delta(x)},\delta_{\pi_\Delta(y)})\to 0.
\end{align*}
Finally, it follows from \eqref{eq:kernel periodicity proof} that
\[\hat{P}(s+i\Delta-n\tau,x,s+i\Delta,\cdot)=\hat{P}(s+i\Delta,x,s+i\Delta+n\tau,\cdot).\]
Hence the pullback limit is the same measure $\hat\mu_{s+i\Delta}$ for every $x$. Since $i\in\mathbb Z$ is arbitrary, the periodic measure is independent of $x$.

We establish a uniform moment estimate. Iterating \eqref{eq:Lyp proof} yields for $0\leq i\leq j$,
\[\hat P(s+i\Delta,s+j\Delta)V(x)\leq \left(1-\frac{p\lambda}{2}\Delta\right)^{j-i}V(x)+\frac{2C_0}{p\lambda}\left(1-\left(1-\frac{p\lambda}{2}\Delta\right)^{j-i}\right).\]
Then it follows that
\[\hat{\mu}_s(V)\leq \left(1-\frac{p\lambda}{2}\Delta\right)^N\hat{\mu}_s(V)+\frac{2C_0}{p\lambda}\left(1-\left(1-\frac{p\lambda}{2}\Delta\right)^N\right).\]
Since $\hat{\mu}(V)<\infty$ and $(1-2^{-1}p\lambda\Delta)^N<1$, we obtain $\hat{\mu}_s(V)\leq 2p^{-1}\lambda^{-1}C_0$. Thus, for every $i\in\bn_0$,
\[\hat{\mu}_{s+i\Delta}(V)\leq \left(1-\frac{p\lambda}{2}\Delta\right)^i\hat{\mu}_s(V)+\frac{2C_0}{p\lambda}\left(1-\left(1-\frac{p\lambda}{2}\Delta\right)^i\right)\leq\frac{2C_0}{p\lambda}.\]
For each $i<0$, the same bound follows from $\hat{\mu}_{s+i\Delta}=\hat{\mu}_{s+r\Delta}$ with $r\in\{0,\cdots,N-1\}$. Therefore,
\begin{equation}\label{eq:mu(V) proof}
\sup_{i\in\bz}\hat{\mu}_{s+i\Delta}(V)\leq \frac{2C_0}{p\lambda},
\end{equation}
and the bound is uniform in $s,\Delta$ and $N$.

We now prove convergence in the weighted total variation distance $\rho_\beta$. For every $i,j\in\bz$ with $i\leq j$ and $x\in\br^d$, we have
\begin{align*}
    \rho_\beta(\hat{P}(s+i\Delta,x,s+j\Delta,\cdot),\hat\mu_{s+j\Delta})\leq\re^{-\hat{\gamma}(j-i)\Delta}\rho_\beta(\delta_{\pi_\Delta(x)},\hat\mu_{s+i\Delta}).
\end{align*}
Using \eqref{eq:mu(V) proof} and $V(\pi_\Delta(x))\leq V(x)\leq C(1+|x|^p)$, we have 
\begin{equation}\label{eq:forward rho beta}
        \rho_\beta(\hat{P}(s+i\Delta,x,s+j\Delta,\cdot),\hat{\mu}_{s+j\Delta})\leq C_\Xi(1+|x|^p)\re^{-\hat{\gamma}(j-i)\Delta}, 
\end{equation}
where $C_\Xi:=2+\beta+2p^{-1}\lambda^{-1}\beta C_0$. Finally, for every $i\in\mathbb Z$, $k\in\mathbb N_0$ and $x\in\mathbb R^d$, it follows from \eqref{eq:forward rho beta} that
\begin{align*}
    \rho_\beta(\hat P(s+i\Delta-k\tau,x,s+i\Delta,\cdot),\hat{\mu}_{s+i\Delta})&=\rho_\beta(\hat {P}(s+i\Delta,x,s+(i+kN)\Delta,\cdot),\hat{\mu}_{s+(i+kN)\Delta})\\
    &\leq C_\Xi(1+|x|^p)e^{-\hat{\gamma}k\tau}.
\end{align*}
This proves the stated pullback estimate. The constants $C_\Xi$ and $\hat{\gamma}$ are independent of $s$, $\Delta$, $N$, $i$, $k$, and $x$.
\end{proof}
    
\section{Convergence of solutions to SDEs over a finite time horizon}\label{sec:finite time}
In order to analyze the finite-time convergence rate, we
introduce a c\`adl\`ag interpolation of the numerical solution in
continuous time. For any $k\in\bn_0$ and $t\in[k\Delta,(k+1)\Delta)$, define
\begin{equation}\label{def:bar Xt}
    \bar{X}_t=\hat{X}_k+b^s(k\Delta,\hat{X}_k)(t-k\Delta)+\sigma^s(k\Delta,\hat{X}_k)(L_t^s-L_{k\Delta}^s).
\end{equation}
where $\hat{X}_0=\bar{X}_0=\pi_\Delta(x)$ and $\tilde{X}_0=X_0=x\in\br^d$. Since a L\'evy process has no jump at any fixed deterministic time almost surely, and the set of grid points is countable, we can work on a full-probability event on which $L_{s+k\Delta}=L_{s+k\Delta-}$ for any $k\geq 1$. At the grid point $t=k\Delta$, it follows from \eqref{def:bar Xt} and the definition of the numerical scheme that
\begin{equation}\label{eq:cadlag}
    \lim_{t\to k\Delta+}\bar{X}_t=\bar{X}_{k\Delta}=\hat{X}_k,~~\lim_{t\to k\Delta-}\bar{X}_t=\tilde{X}_k,~~k\in\bn
\end{equation}
Consequently, in addition to the jumps inherited from the driving $\alpha$-stable process, the interpolation may have an artificial
projection jump at each grid point $k\Delta$, given by
\[\Delta\bar{X}_{k\Delta}:=\bar{X}_{k\Delta}-\bar{X}_{k\Delta-}=\hat{X}_k-\tilde{X}_k,~~k\in\bn.\]
To incorporate these projection jumps into a global semimartingale representation, define
\begin{equation}\label{eq:J_t}
   J_t^\Delta:=\sum_{k=1}^{\lfloor t/\Delta\rfloor}(\hat{X}_k-\tilde{X}_k),~~t\geq 0,
\end{equation}
where the process $J^\Delta$ is an adapted and c\`adl\`ag, with piecewise constant paths of finite variation on every bounded time interval. Its jumps satisfy $\Delta J_{k\Delta}^\Delta=\hat{X}_k-\tilde{X}_k$. Then the interpolation admits the global representation
\begin{equation}\label{eq:bar X}
    \bar{X}_t=\pi_\Delta(x)+\int_0^tb^s(\hat{r},\hat{X}_r)\d r+\int_0^t\sigma^s(\hat{r},\hat{X}_{r-})\d L^s_r+{J_t^\Delta},
\end{equation}
where $\hat{X}_r=\hat{X}_k$ and $\hat{r}=k\Delta$ for $r\in[k\Delta,(k+1)\Delta)$. Here and below, the stochastic integrand is understood as its predictable step-process version.

\begin{lemma}\label{E[sup|bar X t|]}
    Suppose that Assumption \ref{cond:one-side Lip and initial value} holds. Let $\bar{X}$ be the interpolation defined in \eqref{def:bar Xt}. Then, for every $1\leq p<\alpha$, there exist constants $C_\Xi>0$ and $\Delta^*\in(0,1)$ such that for all $s\in\br$, $\Delta\in(0,\Delta^*]$ and $x\in\br^d$,
    \begin{equation*}
        \sup_{t\geq 0}\be[|\bar{X}_t|^p]\leq C_\Xi(1+|x|^p).
    \end{equation*}
\end{lemma}

\begin{proof}
    Fix $p\in[1,\alpha)$, and choose $\Delta^*\in(0,1)$ no larger than the step-size threshold in Lemma \ref{lem:Lyapunov condition for TEM}. Fix $s\in\br$, $x\in\br^d$ and $\Delta\in(0,\Delta^*]$. Since $\hat{X}_0=\pi_\Delta(x)$, we have $\hat{X}_k\in\bar{B}(0,R_\Delta)$ for each $k\in\bn_0$. For any $t\geq0$, there exists an integer $k\in\bn_0$ such that $t\in[k\Delta,(k+1)\Delta)$. By the definition of the interpolation, we obtain
    \[\be[|\bar{X}_t|^p]\leq C_p\be[|\hat{X}_k|^p]+C_p\Delta^p\be[|b^s(k\Delta,\hat{X}_k)|^p]+C_p\be[|\sigma^s(k\Delta,\hat{X}_k)(L_t^s-L_{k\Delta}^s)|^p].\]
    By \eqref{eq:b(hat X) linear growth} and the boundedness of $\sigma$, we have
    \[\be[|\bar{X}_t|^p]\leq C_p\be[|\hat{X}_k|^p]+C_{p,M}\Delta^p(1+\Delta^{-\frac{p}{3}}\be[|\hat{X}_k|^p])+C_{p,d,\alpha,\Lambda}\Delta^\frac{p}{\alpha}\]
    It follows from Lemma \ref{lem:Lyapunov condition for TEM} and $\Delta<1$ that
    \begin{equation*}
        \be[|\bar{X}_t|^p]\leq C_\Xi(1+|x|^p).
    \end{equation*}
   Taking the supremum over $t\geq 0$ completes the proof.
\end{proof}

\begin{lemma}\label{lem:barX-hat X}
    Suppose that Assumption \ref{cond:one-side Lip and initial value} holds. Then for every $p\in[1,\alpha)$, there exist constants $C_\Xi>0$ and $\Delta^*\in(0,1)$ such that for all $\Delta\in(0,\Delta^*]$ and $x\in\br^d$,
    \begin{equation*}
        \be[|\bar{X}_t-\hat{X}_t|^p]\leq C_\Xi(1+|x|^p)\Delta^{p\left(\frac{1}{\alpha}\wedge\frac{2}{3}\right)},~~t\geq 0.
    \end{equation*}
\end{lemma}

\begin{proof}
    Fix $p\in[1,\alpha)$, and choose $\Delta^*\in(0,1)$ no larger than the step-size threshold in Lemma \ref{lem:Lyapunov condition for TEM}. Fix $s\in\br$, $x\in\br^d$ and $\Delta\in(0,\Delta^*]$. Since $\hat{X}_0=\pi_\Delta(x)$, we have $\hat{X}_k\in\bar{B}(0,R_\Delta)$ for each $k\in\bn_0$. For any $t\geq0$, there exists an integer $k\in\bn_0$ such that $t\in[k\Delta,(k+1)\Delta)$. Note that
    \[\bar{X}_t-\hat{X}_t=b^s(k\Delta,\hat{X}_k)(t-k\Delta)+\sigma^s(k\Delta,\hat{X}_k)(L_t^s-L_{k\Delta}^s).\]
    Then it follows from Lemma \ref{lem:Lyapunov condition for TEM} and \eqref{con:bounded} that
    \begin{equation}\label{eq:E[|bar X-hat X|^p] x leq R}
        \begin{aligned}
            \be[|\bar{X}_t-\hat{X}_t|^p]&\leq C_p\be[|b^s(k\Delta,\hat{X}_k)|^p]\Delta^p+C_p\be[|\sigma^s(k\Delta,\hat{X}_k)(L^s_t-L_{k\Delta}^s)|^p]\\
            &\leq C_{p,M}\Delta^p(1+\Delta^{-\frac{p}{3}}\be[|\hat{X}_k|^p])+C_\Xi\Delta^\frac{p}{\alpha}.
        \end{aligned}
    \end{equation}
    We derive from \eqref{eq:E|Y(s+i Delta t)| leq C(1+e-beta Delta t)} that 
    \begin{align*}
        \be[|\bar{X}_t-\hat{X}_t|^p]&\leq C_{p,M}(\Delta^p+\Delta^\frac{2p}{3}\be[|\hat{X}_k|^p])+C_\Xi\Delta^\frac{p}{\alpha}\leq C_\Xi(1+|x|^p)\Delta^{p\left(\frac{1}{\alpha}\wedge\frac{2}{3}\right)}.
    \end{align*}
    The proof is complete.
\end{proof}

\begin{lemma}\label{lem:P(eta leq t)}
    Suppose that Assumption \ref{cond:one-side Lip and initial value} holds. Then, for every $p\in[1,\alpha)$, there exist constants $C_\Xi\geq 1$ and $\Delta^*\in(0,1)$ such that, for every $\Delta\in(0,\Delta^*]$, the stopping time
    \[\bar{\varsigma}_{R_\Delta}:=\inf\{t\geq 0:|\bar{X}_t|\geq R_\Delta\}\]
    with $\inf\emptyset=\infty$, satisfies
    \begin{equation}\label{eq:P(eta leq t)}
        \bp(\bar{\varsigma}_{R_\Delta}\leq T)\leq \frac{C_\Xi(1+|x|^p)(1+T)}{R_\Delta^p}
    \end{equation}
    for every $T>0$ and $x\in\br^d$, where $C_\Xi$ is independent of $T$ and $\Delta$.
\end{lemma}

\begin{proof}
    For notational simplicity, we write $\bar{\varsigma}:=\bar{\varsigma}_{R_\Delta}$. Choose $\Delta^*\in(0,1)$ no larger than the step-size threshold in Lemma \ref{lem:Lyapunov condition for TEM}. For initial value $|x|\geq R_\Delta$, we have $\bar{\varsigma}=0$. For every $T>0$, we have
    \[\bp(\bar{\varsigma}\leq T)=1\leq \frac{C_\Xi(1+|x|^p)(1+T)}{R_\Delta^p}.\]
    We therefore assume throughout the remainder of the proof that $|x|<R_\Delta$.
    Define $V(x)=(1+|x|^2)^\frac{p}{2}$. Recall that the global representation of $\bar{X}$ is
    \[\bar{X}_t=x+\int_0^tb^s(\hat{r},\hat{X}_r)\d r+\int_0^t\sigma^s(\hat{r},\hat{X}_{r-})\d \tilde{L}_r^s+J_t^\Delta,\]
    where $J_t^\Delta=\sum_{k=1}^{\lfloor t/\Delta\rfloor}(\hat{X}_k-\tilde{X}_k)$. Applying It\^o's formula for semimartingales to $V(\bar{X}_{t\wedge\bar{\varsigma}})$ gives
    \begin{align*}
        V(\bar{X}_{t\wedge\bar{\varsigma}})&=V(x)+\int_0^{t\wedge\bar{\varsigma}}\langle \nabla V(\bar{X}_r),b^s(\hat{r},\hat{X}_r)\rangle\d r+\sum_{k\geq1,k\Delta\leq t\wedge\bar{\varsigma}}(V(\hat{X}_k)-V(\tilde{X}_k))\\
        &\quad+\int_0^{t\wedge\bar{\varsigma}}\int_{|z|>1}[V(\bar{X}_{r-}+\sigma^s(\hat{r},\hat{X}_{r-})z)-V(\bar{X}_{r-})]N(\d z,\d r)+\co_{t\wedge\bar{\varsigma}}\\
        &\quad+\int_0^{t\wedge\bar{\varsigma}}\int_{|z|\leq 1}[V(\bar{X}_{r-}+\sigma^s(\hat{r},\hat{X}_{r-})z)-V(\bar{X}_{r-})\\
        &\qquad\qquad\qquad\qquad\qquad\qquad\quad-\langle\nabla V(\bar{X}_{r-}),\sigma^s(\hat{r},\hat{X}_{r-})z\rangle]\nu(\d z)\d r,
    \end{align*}
    where
    \[\co_t=\int_0^t\int_{|z|\leq 1}[V(\bar{X}_{r-}+\sigma^s(\hat{r},\hat{X}_{r-})z)-V(\bar{X}_{r-})]\Tilde{N}(\d z,\d r)\]
    is a square-integrable martingale. The optional sampling theorem gives $\be[\co_{T\wedge\bar{\varsigma}}]=0$. Since $\hat{X}_k=\pi_\Delta(\tilde{X}_k)$, we obtain $|\hat{X}_k|\leq |\tilde{X}_k|$. Because $V$ is increasing as a function of $|x|$, it follows that
    \[V(\hat{X}_k)-V(\tilde{X}_k)\leq 0.\]
    On the event $\{r<\bar{\varsigma}\}$, we have $|\bar{X}_r|<R_\Delta$ and $|\hat{X}_r|\leq R_\Delta$. It follows from \eqref{eq:b(t,pix)-b(t,piy)} and \eqref{eq:another loc Lip} that
    \begin{equation}\label{eq: b(hat{X}-bar{X})}
        \begin{aligned}
            |b^s(\hat r,\hat X_r)-b^s(r,\bar{X}_r)|&\leq |b^s(\hat{r},\hat{X}_r)-b^s(\hat{r},\bar{X}_r)|+|b^s(\hat{r},\bar{X}_r)-b^s(r,\bar{X}_r)|\\
            &\leq M\Delta^{-\frac{1}{3}}|\hat{X}_r-\bar{X}_r|+\ell_b|r-\hat{r}|^\frac{1}{\alpha}.
        \end{aligned}
    \end{equation}
    Hence, we derive from \eqref{eq:numerical weak dissipativity} that on $\{r<\bar{\varsigma}\}$,
    \begin{align*}
        \langle\nabla V(\bar X_r),b^s(\hat r,\hat X_r)\rangle&=p(1+|\bar{X}_r|^2)^{\frac{p}{2}-1}\langle\bar X_r,b^s(\hat r,\hat X_r)\rangle\\
        &=p(1+|\bar{X}_r|^2)^{\frac{p}{2}-1}(\langle\bar X_r,b^s( r,\bar{X}_r)\rangle+\langle\bar {X}_r,b^s(\hat r,\hat X_r)-b^s(r,\bar{X}_r)\rangle)\\
        &\leq p(1+|\bar X_r|^2)^{\frac {p}{2}-1}(-\lambda|\bar X_r|^2+L+|\bar{X}_r|(M\Delta^{-\frac{1}{3}}|\hat{X}_r-\bar{X}_r|+
        \ell_b\Delta^{\frac{1}{\alpha}})).
    \end{align*}
    Observe that
    \[|\bar{X}_r|(1+|\bar{X}_r|^2)^{\frac{p}{2}-1}\leq (1+|\bar{X}_r|^2)^{\frac{p}{2}-  \frac{1}{2}}=V(\bar X_r)^{\frac{p-1}{p}},\]
    and
    \[|\bar{X}_r|^2(1+|\bar{X}_r|^2)^{\frac{p}{2}-1}=V(\bar{X}_r)-(1+|\bar X_r|^2)^{\frac p2-1}\geq V(\bar{X}_r)-1.\]
    Then we have
    \begin{align*}
        \langle\nabla V(\bar{X}_r),b^s(\hat{r},\hat{X}_r)\rangle\leq -p\lambda V(\bar{X}_r)+p(\lambda+L)+pV(\bar{X}_r)^\frac{p-1}{p}(M\Delta^{-\frac{1}{3}}|\hat{X}_r-\bar{X}_r|+
        \ell_b\Delta^{\frac{1}{\alpha}}).
    \end{align*}
    Young's inequality yields
    \begin{align*}
        pM\Delta^{-\frac{1}{3}}V(\bar{X}_r)^{\frac{p-1}{p}}|\hat{X}_r-\bar{X}_r|\leq \frac{p\lambda}{4}V(\bar{X}_r)+C_{p,M,\lambda}\Delta^{-\frac{p}{3}}|\hat{X}_r-\bar{X}_r|^p,
    \end{align*}
    and
    \begin{align*}
        p\ell_b\Delta^\frac{1}{\alpha}V(\bar{X}_r)^{\frac{p-1}{p}}&\leq\frac{p\lambda}{4}V(\bar{X}_r)+C_{p,\lambda,\ell_b}\Delta^\frac{p}{\alpha}.
    \end{align*}
    Therefore, we obtain
    \begin{align}\label{eq:xb}
        \langle\nabla V(\bar{X}_r),b^s(\hat{r},\hat{X}_r)\rangle\leq-\frac{p\lambda}{2}V(\bar{X}_r)+C_\Xi\Delta^{-\frac{p}{3}}|\hat{X}_r-\bar{X}_r|^p+C_\Xi(1+\Delta^\frac{p}{\alpha}).
    \end{align}
    Since the Lebesgue integral over a single point is zero, we have
    \begin{align*}
        \be\left[\int_{[0,T\wedge\bar{\varsigma}]}\langle \nabla V(\bar{X}_r),b^s(\hat{r},\hat{X}_r)\rangle\d r\right]=\be\left[\int_{[0,T\wedge\bar{\varsigma})}\langle \nabla V(\bar{X}_r),b^s(\hat{r},\hat{X}_r)\rangle\d r\right].
    \end{align*}
    It follows from \eqref{eq:xb} and \eqref{eq:numerical ellipticity} that
    \begin{align*}
        \be[V(\bar{X}_{T\wedge\bar{\varsigma}})]&\leq V(x)+\be\left[\int_{[0,T\wedge\bar{\varsigma})}(C_\Xi\Delta^{-\frac{p}{3}}|\hat{X}_r-\bar{X}_r|^p+C_\Xi\Delta^\frac{p}{\alpha}+C_\Xi)\d r\right]\\
        &\quad+C_\Xi\int_0^T\int_{|z|>1}\be[|\bar{X}_r|^p+|\sigma^s(\hat{r},\hat{X}_{r})z|^p]\nu(\d z)\d r\\
        &\quad+C_\Xi\int_0^T\int_{|z|\leq 1}\be[|\sigma^s(\hat{r},\hat{X}_{r})z|^2]\nu(\d z)\d r\\
        &\leq V(x)+C_\Xi \Delta^{-\frac{p}{3}}\int_0^T\be[|\hat{X}_r-\bar{X}_r|^p]\d r+C_\Xi(1+|x|^p)T.
    \end{align*}
    Since $|x|<R_\Delta$, Lemma \ref{lem:barX-hat X} yields
    \begin{align*}
        \be[V(\bar{X}_{T\wedge\bar{\varsigma}})]&\leq C_\Xi(1+|x|^p)(1+T).
    \end{align*}
    Finally, on the event $\{\bar{\varsigma}\leq T\}$, the definition of $\bar{\varsigma}$ and the right-continuity of $\bar{X}$ imply that $|\bar{X}_{\bar{\varsigma}}|\geq R_\Delta$. Therefore, we have
    \begin{align*}
        R_\Delta^p\bp(\bar{\varsigma}\leq T)\leq \be[V(\bar{X}_{T\wedge\bar{\varsigma}})\bone_{\{\bar{\varsigma}\leq T\}}]\leq \be[V(\bar{X}_{T\wedge\bar{\varsigma}})],
    \end{align*}
    which implies that \eqref{eq:P(eta leq t)} holds for $1<p<\alpha$.

    For $p=1$, we have $V(\bar X_r)^\frac{p-1}{p}=1$. Hence, Lemma \ref{lem:barX-hat X} gives
    \[\Delta^{-\frac{1}{3}}\be[|\bar X_r-\hat X_r|]\leq C_\Xi(1+|x|)\Delta^{(\frac{1}{\alpha}\wedge\frac{2}{3})-\frac{1}{3}}\leq C_\Xi(1+|x|).\]
    The temporal increment term is also uniformly bounded for $\Delta\le1$. The remaining estimates therefore yield
    \[\be V(\bar X_{T\wedge\bar\varsigma})\leq C_\Xi(1+|x|)(1+T),\]
    which proves \eqref{eq:P(eta leq t)} for $p=1$.
\end{proof}

For every fixed $\delta\in(0,1]$ and $p\in[1,\alpha)$, define $V_\delta(x):=(\delta+|x|^2)^\frac{p}{2}$.
For any $x,y,w\in\br^d$, let
\[\mathcal{D}(x,y):=V_\delta(x+y)-V_\delta(x)-\langle\nabla V_\delta(x),y\rangle,\]
and
\[\mathcal{D}(x,y,w):=V_\delta(x+y+w)-V_\delta(x+y)-\langle\nabla V_\delta(x),w\rangle.\]

Recall that $e_t:=X_t-\bar{X}_t$ and $\hat{\varsigma}:=\inf\{t\geq 0:|X_t|\geq R_\Delta\}\wedge\inf\{t\geq 0:|\bar{X}_t|\geq R_\Delta\}$. Define $G:=\{\hat{\varsigma}>0\}$ and $\Sigma_t:=\sigma^s(t,X_t)-\sigma^s(\hat{t},\hat{X}_t)$, where $\Sigma_t^1:=\sigma^s(t,X_t)-\sigma^s(t,\bar{X}_t)$ and $\Sigma_t^2:=\sigma^s(t,\bar{X}_t)-\sigma^s(\hat{t},\hat{X}_t)$. Define
\[\mathcal{J}(t):=\int_0^t\int_{|z|\leq\frac{1}{2C_\sigma}}\mathcal{D}(e_{r-},\Sigma_{r-}^1z)\nu(\d z)\d r,\]
and
\[\bar{\mathcal{J}}(t):=\int_0^t\int_{|z|\leq\frac{1}{2C_\sigma}}\mathcal{D}(e_{r-},\Sigma_{r-}^1z,\Sigma_{r-}^2z)\nu(\d z)\d r.\]

\begin{lemma}\label{lem:|z|leq 1/2C_sigma}
Suppose that Assumption \ref{cond:one-side Lip and initial value} holds. Then, for every $1\leq p\leq m<\alpha$, there exist constants $C_\Xi>0,C_{\Xi,m}>0$ and $\Delta^*\in(0,1)$ such that, for every $s\in\br$, $x\in\br^d$, $\delta\in(0,1]$, $T>0$, $t\in[0,T]$ and $\Delta\in(0,\Delta^*]$,
\begin{equation*}
    \begin{aligned}
    \be\left[\mathcal{J}(t\wedge \hat{\varsigma})\bone_G\right]\leq C_\Xi\int_0^t\be[V_\delta(e_r)\bone_{\{r<\hat{\varsigma}\}}]\d r,
    \end{aligned}
\end{equation*}
and
\begin{equation*}\label{/2C_sigma}
    \begin{aligned}
        \be\left[\bar{\mathcal{J}}(t\wedge\hat{\varsigma})\bone_G\right]
        \leq C_{\Xi}\int_0^t\be[V_\delta(e_r)\bone_{\{r<\hat{\varsigma}\}}]\d r + C_{\Xi,m}(1+|x|^m)T(\Delta^{p(\frac{1}{\alpha}\wedge\frac{2}{3}})+\delta^{\frac{p}{2}-\frac{\alpha}{2}}\Delta^{m(\frac{1}{\alpha}\wedge\frac{2}{3})}).
    \end{aligned}
\end{equation*}
The constants $C_\Xi,C_{\Xi,m}$ and $\Delta^*$ are independent of $s,x,\delta,t,T$, and $\Delta$.
\end{lemma}

\begin{proof}
    We begin by proving the first inequality. Taylor's formula yields
    \begin{equation*}
        \mathcal{D}(e_r,\Sigma_r^1z)=\int_0^1(1-u)\langle\nabla^2 V_\delta(e_r+u\Sigma_r^1z)\Sigma_r^1z,\Sigma_r^1z\rangle\d u.
    \end{equation*}
    Let $U_\delta(x):=(\delta+|x|^2)^\frac{1}{2}$. Then $V_\delta(x)=U_\delta(x)^p$, $\nabla V_\delta(x)=pU_\delta(x)^{p-2}x$ and $\nabla^2V_\delta(x)=pU_\delta(x)^{p-2}I
      +p(p-2)U_\delta(x)^{p-4}xx^\top$. Since $1\leq p<\alpha<2$, we have $|\nabla V_\delta(x)|\leq pU_\delta(x)^{p-1}$ and $\|\nabla^2V_\delta(x)\|
    \leq C_pU_\delta(x)^{p-2}$. For every $u\in[0,1]$, since $U_\delta$ is 1-Lipschitz, we obtain
    \begin{equation}\label{eq:V(e_sigma) geq V(e)}
        U_\delta(e_r+u\Sigma_r^1z)\geq U_\delta(e_r)-u\Vert\Sigma_r^1z\Vert\geq U_\delta(e_r)-\frac{1}{2}|e_r|\geq\frac{1}{2}U_\delta(e_r).
    \end{equation}
    For $|z|\leq\frac{1}{2C_\sigma}$, we obtain
    \begin{equation}\label{eq:D leq CV|z|2}
        \mathcal{D}(e_r,\Sigma_r^1z)\leq C_\Xi U_\delta(e_r)^{p-2}\Vert\Sigma_r^1\Vert^2|z|^2\leq C_\Xi V_\delta(e_r)|z|^2.
    \end{equation}
    It follows from \eqref{eq:V(e_sigma) geq V(e)} and \eqref{eq:D leq CV|z|2} that
    \begin{equation}\label{eq: J}
        \mathcal{J}(t\wedge\hat{\varsigma})\leq C_\Xi\int_0^{t\wedge\hat{\varsigma}}V_\delta(e_r)\d r=C_\Xi\int_0^{t}V_\delta(e_{r})\bone_{\{\hat{\varsigma}>r\}}\d r.
    \end{equation}
    Multiplying both sides of \eqref{eq: J} by $\bone_G$ and taking expectations yields the first inequality.

    We now prove the second inequality. Define
    \[\varrho_r:=\frac{2C_\sigma U_\delta(e_r)+4\Vert\Sigma_r^2\Vert}{U_\delta(e_r)}=2C_\sigma+\frac{4\Vert\Sigma_r^2\Vert}{U_\delta(e_r)}.\]
    Then $\varrho_r\geq2C_\sigma$. For $u,\bar{u}\in[0,1]$, Taylor's formula gives
    \begin{equation}\label{eq:Taylor}
        \mathcal{D}(e_r,\Sigma_r^1z,\Sigma_r^2z)=\int_0^1\int_0^1\langle\nabla^2V_\delta(e_r+\bar{u}(\Sigma_r^1z+u\Sigma_r^2z))(\Sigma_r^1z+u\Sigma_r^2z),\Sigma_r^2z\rangle\d\bar{u}\d u.
    \end{equation}
    It follows from $|z|\leq\frac{1}{\varrho_r}$ that
    \[|\Sigma_r^1z|\leq \frac{C_\sigma}{\varrho_r}|e_r|\leq\frac{1}{2}|e_r|~~\mbox{and}~~|\Sigma_r^2z|\leq \frac{\Vert\Sigma_r^2\Vert}{\varrho_r}\leq \frac{1}{4}U_\delta(e_r).\]
    Since $U_\delta$ is 1-Lipschitz, we have
    \begin{align}\label{eq:V geq 1/4V}
        U_\delta(e_r+\bar{u}(\Sigma_r^1z+u\Sigma_r^2z))\geq U_\delta(e_r)-\frac{1}{2}|e_r|-\frac{1}{4}U_\delta(e_r)\geq\frac{1}{4}U_\delta(e_r).
    \end{align}
    Combining \eqref{eq:Taylor} and \eqref{eq:V geq 1/4V}, we obtain
    \begin{align*}
        |\mathcal{D}(e_r,\Sigma_r^1z,\Sigma_r^2z)|\leq C_\Xi U_\delta(e_r)^{p-2}(U_\delta(e_r)+\Vert\Sigma_r^2\Vert)\Vert\Sigma_r^2\Vert|z|^2.
    \end{align*}
Hence, we arrive at
\begin{align}\label{eq:|z| leq a}
    \int_{|z|\leq \frac{1}{\varrho_r}}|\mathcal{D}(e_r,\Sigma_r^1z,\Sigma_r^2z)|\nu(\d z)&\leq C_\Xi U_\delta(e_r)^{p-2}(U_\delta(e_r)+\Vert\Sigma_r^2\Vert)\Vert\Sigma_r^2\Vert \varrho_r^{\alpha-2}.
\end{align}
For the remaining region $\frac{1}{\varrho_r}\leq |z|\leq \frac{1}{2C_\sigma}$ , we first observe that
\begin{equation}
    |V_\delta(u+w)-V_\delta(u)|\leq C_p(U_\delta(u)^{p-1}|w|+|w|^p).
\end{equation}
Since $U_\delta(e_r+\Sigma_r^1z)\leq \frac{3}{2}U_\delta(e_r)$ for $|z|\leq \frac{1}{2C_\sigma}$, we obtain
\[|\mathcal{D}(e_r,\Sigma_r^1z,\Sigma_r^2z)|\leq C_\Xi(U_\delta(e_r)^{p-1}\Vert\Sigma_r^2\Vert|z|+\Vert\Sigma_r^2\Vert^p|z|^p).\]
For each $a\in\{1,p\}$, since $a<\alpha$, we have
\begin{align*}
    \int_{\frac{1}{\varrho_r}<|z|\leq\frac{1}{2C_\sigma}}|z|^a\nu(\d z)=\frac{C_{\alpha,d}S_d}{\alpha-a}(\varrho_r^{\alpha-a}-(2C_\sigma)^{\alpha-a})\leq C_{\alpha,d,a}\varrho_r^{\alpha-a}.
\end{align*}
Therefore, we conclude that
\begin{equation}\label{eq:a leq |z| leq b}
    \int_{\frac{1}{\varrho_r}<|z|\leq\frac{1}{2C_\sigma}}|\mathcal{D}(e_r,\Sigma_r^1z,\Sigma_r^2z)|\nu(\d z)\leq C_\Xi(U_\delta(e_r)^{p-1}\Vert\Sigma_r^2\Vert\varrho_r^{\alpha-1}+\Vert\Sigma_r^2\Vert^p\varrho_r^{\alpha-p}).
\end{equation}
Combining \eqref{eq:|z| leq a} and \eqref{eq:a leq |z| leq b}, we obtain
\begin{equation}\label{eq:|z|leq b}
    \int_{|z|\leq \frac{1}{2C_\sigma}}|\mathcal{D}(e_r,\Sigma_r^1z,\Sigma_r^2z)|\nu(\d z)\leq C_\Xi(V_\delta(e_r)+\Vert\Sigma_r^2\Vert^p+\delta^\frac{p-\alpha}{2}\Vert\Sigma_r^2\Vert^\alpha).
\end{equation}
Here the bound follows from Young's inequality when $p>1$; when $p=1$, the corresponding term is simply $\Vert\Sigma_r^2\Vert$.
It remains to estimate the moments of $\Sigma_r^2$.
Write $\Sigma_r^2=\sigma^s(r,\bar X_r)-\sigma^s(\hat r,\bar X_r)+\sigma^s(\hat r,\bar X_r)-\sigma^s(\hat r,\hat X_r)$. By the regularity of $\sigma$ and
Lemma \ref{lem:barX-hat X},
\begin{equation*}
\begin{aligned}
    \be[\|\Sigma_r^2\|^p]\leq C_\Xi(\Delta^\frac{p}{\alpha}+\be|\bar X_r-\hat X_r|^p)\leq C_\Xi(1+|x|^p)\Delta^{p(\frac{1}{\alpha}\wedge\frac{2}{3})}.
\end{aligned}
\end{equation*}
Since $m<\alpha$ and $\|\sigma(t,x)\|\leq \sqrt{d\Lambda}$, we have
\[\|\sigma^s(\hat r,\bar X_r)-\sigma^s(\hat r,\hat X_r)\|^\alpha\leq(2\sqrt{d\Lambda})^{\alpha-m}\|\sigma^s(\hat r,\bar X_r)-\sigma^s(\hat r,\hat X_r)\|^m\leq C_{\Xi,m}|\bar X_r-\hat X_r|^m.\]
Another application of Lemma \ref{lem:barX-hat X} gives
\begin{equation*}
\begin{aligned}
    \be[\|\Sigma_r^2\|^\alpha]\leq C_{\Xi}(\Delta+\mathbb E|\bar X_r-\hat X_r|^m)\leq
    C_{\Xi,m}(1+|x|^m)\Delta^{m(\frac{1}{\alpha}\wedge\frac{2}{3})}.
\end{aligned}
\end{equation*}
Integrating \eqref{eq:|z|leq b} up to
$t\wedge\hat\varsigma$ and taking expectations, we obtain
\[\begin{aligned}
    \be[|\bar{\mathcal J}(t\wedge\hat\varsigma)|\bone_G]&\leq C_{\Xi}\int_0^t
    \be[V_\delta(e_r)\mathbf1_{\{r<\hat{\varsigma}\}}]\d r+
    C_{\Xi,m}\int_0^t(\be\|\Sigma_r^2\|^p+\delta^\frac{p-\alpha}{2}\be\|\Sigma_r^2\|^\alpha)\d r\\
    &\leq C_{\Xi}\int_0^t\be[V_\delta(e_r)\mathbf1_{\{r<\hat\varsigma\}}]\d r+
    C_{\Xi,m}(1+|x|^m)T(\Delta^{p(\frac{1}{\alpha}\wedge\frac{2}{3})}+\delta^\frac{p-\alpha}{2}\Delta^{m(\frac{1}{\alpha}\wedge\frac{2}{3})}).
\end{aligned}\]
Here we used $1+|x|^p\leq2(1+|x|^m)$. This completes the proof.
\end{proof}

\begin{lemma}\label{lem:E[|X_t-bar X_t|]}
    Suppose that Assumption \ref{cond:one-side Lip and initial value} holds. Then, for every $p\in(1,\alpha)$, there exist constants $C_\Xi>0$ and $\Delta^*\in(0,1)$ such that for all $s\in\br$, $x\in\br^d$, $T>0$ and  $\Delta\in(0,\Delta^*]$,
    \begin{equation}
        \label{eq:E[|X_t-bar X_t|]}
        \sup_{0\leq t\leq T}\be[|X_t-\bar{X}_t|]\leq C_\Xi(1+|x|^p)\re^{C_\Xi T}\Delta^{\left(\left(\frac{1}{\alpha}\wedge\frac{2}{3}\right)-\frac{1}{3}\right)\wedge\frac{p-1}{3q}}.
    \end{equation}
\end{lemma}

\begin{proof}
    Recall that $e_t=X_t-\bar{X}_t$. Observe that
    \begin{equation}\label{eq:e_t}
        e_t=x-\pi_\Delta(x)+\int_0^t(b^s(r,X_r)-b^s(\hat{r},\hat{X}_r))\d r+\int_0^t(\sigma^s(r,X_{r-})-\sigma^s(\hat{r},\hat{X}_{r-}))\d \tilde{L}_r^s-J_t^\Delta.
    \end{equation}
    Recall $\hat{\varsigma}=\varsigma_{R_\Delta}\wedge\bar{\varsigma}_{R_\Delta}$. We claim that the correction process $J^\Delta$ has no nonzero projection jump before $\hat{\varsigma}$. Indeed, if $\hat{X}_k\neq\tilde{X}_k$, then the definition of the projection implies that $|\tilde{X}_k|>R_\Delta$. On the other hand, $\lim_{t\uparrow k\Delta}\bar{X}_t=\tilde{X}_k$. Therefore, there exists some $t<k\Delta$ sufficiently close to $k\Delta$ such that $|\bar{X}_t|\geq R_\Delta$. It follows that $\bar{\varsigma}_{R_\Delta}<k\Delta$. Consequently, we have
    \begin{equation}\label{eq:hat X-tilde X=0}
        (\hat{X}_k-\tilde{X}_k)\bone_{\{k\Delta\leq\hat{\varsigma}\}}=0,~~k\in\bn.
    \end{equation}
    Thus, the projection-jump contribution in It\^o's formula for the stopped error process vanishes.
    
    Recall the notation $\Sigma_r:=\sigma^s(r,X_r)-\sigma^s(\hat{r},\hat{X}_r)$. We decompose $\Sigma_r$ as $\Sigma_r=\Sigma_r^1+\Sigma_r^2$, where $\Sigma_r^1:=\sigma^s(r,X_r)-\sigma^s(r,\bar{X}_r)$ and $\Sigma_r^2:=\sigma^s(r,\bar{X}_r)-\sigma^s(\hat{r},\hat{X}_r)$.
    For every fixed $\delta>0$, let $V(x)=(\delta+|x|^2)^\frac{1}{2}$. Let $G:=\{\hat{\varsigma}>0\}$. On the event $G$, we have $|x|<R_\Delta$ and $e_0=x-\pi_\Delta(x)=0$. Consequently, we obtain $V(e_0)\bone_G=\sqrt{\delta}\bone_G$. If $|x|\geq R_\Delta$, then $\hat{\varsigma}=0$ and $G$ is empty. Applying It\^o's formula to $V(e_{t\wedge\hat{\varsigma}})$ and using \eqref{eq:hat X-tilde X=0}, we obtain
    \begin{equation}
        \begin{aligned}
            \be[V(e_{t\wedge\hat{\varsigma}})\bone_G]&=\be[\sqrt{\delta}\bone_G]+\be\left[\int_0^{t\wedge\hat{\varsigma}}\frac{\langle e_r, b^s(r,X_r)-b^s(\hat{r},\hat{X}_r)\rangle}{(\delta+|e_r|^2)^{\frac{1}{2}}}\bone_G\d r\right]+\be[\co_{t\wedge\hat{\varsigma}}\bone_G]\\
            &\quad+\be\left[\int_0^{t\wedge\hat{\varsigma}}\int_{|z|>\frac{1}{2C_\sigma}}[V(e_{r-}+\Sigma_{r-}z)-V(e_{r-})]\bone_G N(\d z, \d r)\right]\\
            &\quad+\be[\mathcal{J}(t\wedge\hat{\varsigma})\bone_G+\bar{\mathcal{J}}(t\wedge\hat{\varsigma})\bone_G]\\
            &=:\be[\sqrt{\delta}\bone_G]+\be[I_{t\wedge \hat{\varsigma}}^1\bone_G]+\be[\co_{t\wedge\hat{\varsigma}}\bone_G]+\be[I_{t\wedge \hat{\varsigma}}^2\bone_G]+\be[I_{t\wedge \hat{\varsigma}}^3\bone_G],
        \end{aligned}
    \end{equation}
    where
    \[\co_t=\int_0^t\int_{|z|\leq \frac{1}{2C_\sigma}}[V(e_{r-}+\Sigma_{r-}z)-V(e_{r-})]\Tilde{N}(\d z, \d r)\]
    is a square-integrable martingale. Since $G\in\cf_0$ and $\co_0=0$, it follows that
    \begin{equation}\label{eq:O_t=0}
        \be[\co_{t\wedge\hat{\varsigma}}\bone_G]= \be[\be[\co_{t\wedge\hat{\varsigma}}|\cf_0]\bone_G]=\be[\co_0\bone_G]=0.
    \end{equation}
    
    For any $r<\hat{\varsigma}$, we have $|X_r|<R_\Delta$ and $|\hat{X}_r|\leq R_\Delta$. By the one-sided Lipschitz condition and $|e_r|(\delta+|e_r|^2)^{-\frac{1}{2}}\leq 1$, we have 
    \begin{equation}\label{eq:langle e_t,b-b rangle}
        \frac{\langle e_r,b^s(r,X_r)-b^s(\hat{r},\hat{X}_r)\rangle}{(\delta+|e_r|^2)^\frac{1}{2}}\leq \ell_0|e_r|+|b^s(r,\bar{X}_r)-b^s(\hat{r},\hat{X}_r)|.
    \end{equation}
    Since the Lebesgue integral over a single point is zero, we have
    \[\be\left[\int_{[0,t\wedge\hat{\varsigma}]}(b^s(r,\bar{X}_r)-b^s(\hat{r},\hat{X}_r))\bone_{\{r<\hat{\varsigma}\}}\d r\right]=\be\left[\int_{[0,t\wedge\hat{\varsigma})}(b^s(r,\bar{X}_r)-b^s(\hat{r},\hat{X}_r))\bone_{\{r<\hat{\varsigma}\}}\d r\right].\]
    Then it follows from \eqref{eq: b(hat{X}-bar{X})} and Lemma \ref{lem:barX-hat X} that
    \begin{align*}
        \be\left[I^1_{t\wedge \hat{\varsigma}}\bone_G\right]&\leq \be\left[\int_{[0,t\wedge\hat{\varsigma})}(\ell_0|e_r|+|b^s(r,\bar{X}_r)-b^s(\hat{r},\hat{X}_r)|)\bone_G\d r\right]\\
        &\leq \ell_0\int_0^t\be[|e_{r\wedge \hat{\varsigma}}|\bone_G]\d r+M\Delta^{-\frac{1}{3}}\int_0^t\be[|\bar{X}_r-\hat{X}_r|]\d r+\ell_b T\Delta^\frac{1}{\alpha}\\
        &\leq \ell_0\int_0^t\be[|e_{r\wedge \hat{\varsigma}}|\bone_G]\d r+C_\Xi(1+|x|) T\Delta^{\left(\frac{1}{\alpha}\wedge\frac{2}{3}\right)-\frac{1}{3}},
    \end{align*}
    where the last inequality follows from $\frac{1}{\alpha}>\left(\frac{1}{\alpha}\wedge\frac{2}{3}\right)-\frac{1}{3}$. Since $V(x)$ is a 1-Lipschitz function, we have $|V(e_r+\Sigma_rz)-V(e_r)|\leq \Vert\Sigma_r\Vert|z|$.
    Moreover, we have $\Vert\Sigma_t^1\Vert\leq C_\sigma|e_t|$. Then it follows that
    \begin{align*}
        \be\left[I^2_{t\wedge \hat{\varsigma}}\bone_G\right]&\leq \be\left[\int_0^{t\wedge\hat{\varsigma}}\int_{|z|>\frac{1}{2C_\sigma}} \Vert\Sigma_r^1+\Sigma_r^2\Vert|z|\bone_G\nu(\d z)\d r\right]\\
        &\leq C_\Xi \int_0^t\be[| e_{r\wedge\hat{\varsigma}}|\bone_{\{\hat{\varsigma}>0\}}]\d r+C_\Xi\int_0^t\be[\Vert\Sigma_r^2\Vert]\d r.
    \end{align*}
    By \eqref{eq:sigma Holder in time} and Lemma \ref{lem:barX-hat X}, we have
    \begin{align*}
        \be[\Vert\Sigma_r^2\Vert]&\leq \ell_\sigma|r-\hat{r}|^\frac{1}{\alpha}+C_\sigma\be[|\bar{X}_r-\hat{X}_r|]\leq C_\Xi(1+|x|)\Delta^{\frac{1}{\alpha}\wedge\frac{2}{3}}.
    \end{align*}
    Then we conclude that
    \begin{align*}
        \be\left[I^2_{t\wedge \hat{\varsigma}}\bone_G\right]\leq C_\Xi\int_0^t\be[|e_{r\wedge\hat{\varsigma}}|\bone_G]\d r+C_\Xi(1+|x|)T\Delta^{\frac{1}{\alpha}\wedge\frac{2}{3}}.
    \end{align*}
    By Lemma \ref{lem:|z|leq 1/2C_sigma} with $p=1$ and $m=p$, we have
    \begin{align*}
       \be\left[I^3_{t\wedge   \hat{\varsigma}}\bone_G\right]&\leq C_\Xi\int_0^t\be[V(e_{r\wedge\hat{\varsigma}})\bone_G]\d r+C_\Xi(1+|x|^p)T(\Delta^{\frac{1}{\alpha}\wedge\frac{2}{3}}+\delta^{\frac{1}{2}-\frac{\alpha}{2}}\Delta^{p(\frac{1}{\alpha}\wedge\frac{2}{3})}).
    \end{align*}
    Combining the preceding estimates, we obtain
    \begin{equation*}
        \begin{aligned}
            \be[V(e_{t\wedge\hat{\varsigma}})\bone_G]&\leq \sqrt{\delta}+C_\Xi\int_0^t\be[V(e_{r\wedge\hat{\varsigma}})\bone_G]\d r\\
            &~~+C_\Xi(1+|x|^p)T(\Delta^{\left(\frac{1}{\alpha}\wedge\frac{2}{3}\right)-\frac{1}{3}}+\Delta^{{\frac{1}{\alpha}\wedge\frac{2}{3}}}+\delta^{{\frac{1}{2}-\frac{\alpha}{2}}}\Delta^{{p(\frac{1}{\alpha}\wedge\frac{2}{3}})}).
        \end{aligned}
    \end{equation*}
    Choose $\delta:=\Delta^{\frac{2p}{\alpha}(\frac{1}{\alpha}\wedge\frac{2}{3})}$, then $\sqrt{\delta}=\Delta^{\frac{p}{\alpha}(\frac{1}{\alpha}\wedge\frac{2}{3})}$ and $\delta^\frac{1-\alpha}{2}\Delta^{p(\frac{1}{\alpha}\wedge\frac{2}{3})}=\Delta^{\frac{p}{\alpha}(\frac{1}{\alpha}\wedge\frac{2}{3})}$.
    It is not difficult to see that $\frac{p}{\alpha}(\frac{1}{\alpha}\wedge\frac{2}{3})>(\frac{1}{\alpha}\wedge\frac{2}{3})-\frac{1}{3}$. Gronwall's inequality gives
    \begin{align}\label{eq:b1}
        \sup_{0\leq t\leq T}\be[|e_{t\wedge\hat{\varsigma}}|\bone_G]\leq C_\Xi(1+|x|^p)\re^{C_\Xi T}\Delta^{\left(\frac{1}{\alpha}\wedge\frac{2}{3}\right)-\frac{1}{3}}.
    \end{align}

    We now remove the stopping time. For any $t\leq T$, we decompose
    \[\be[|e_t|]=\be[|e_t|\bone_{\{\hat{\varsigma}>T\}}]+\be[|e_t|\bone_{\{\hat{\varsigma}\leq T\}}].\]
    The first term is bounded by \eqref{eq:b1}.
    The exit probability estimates for the exact solution and the interpolation imply
    \begin{equation*}
        \bp(\hat{\varsigma}\leq T)\leq \bp(\varsigma_{R_\Delta}\leq T)+\bp(\bar{\varsigma}_{R_\Delta}\leq T)\leq \frac{C_\Xi(1+|x|^p)(1+T)}{R_\Delta^p}.
    \end{equation*}
    Thus H\"older's inequality, Theorem \ref{thm:E[|X|]} and Lemma \ref{E[sup|bar X t|]} give
    \begin{equation*}
        \begin{aligned}
            \be[|e_t|\bone_{\{\hat{\varsigma}\leq T\}}]&\leq (\be[|e_{t}|^p])^\frac{1}{p}(\bp(\hat{\varsigma}\leq T))^\frac{p-1}{p}\leq C_\Xi(1+|x|^p)(1+T)^\frac{p-1}{p}R_\Delta^{1-p}.
        \end{aligned}
    \end{equation*}
    Hence, we conclude that
    \begin{equation}\label{eq:sup|e_t|}
        \sup_{0\leq t\leq T}\be[|e_t|]\leq C_\Xi(1+|x|^p)\re^{C_\Xi T}(\Delta^{\left(\frac{1}{\alpha}\wedge\frac{2}{3}\right)-\frac{1}{3}}+R_\Delta^{1-p}).
    \end{equation}
    Recall that $\vp(R)=\ell(1+2R^{q})$ and $M\geq 3\ell$, then we have $R_\Delta^q=2^{-1}\ell^{-1}M\Delta^{-\frac{1}{3}}-2^{-1}\geq\Delta^{-\frac{1}{3}}$. It follows that $R_\Delta^{-(p-1)}\leq \Delta^\frac{p-1}{3q}$.  Substituting this into \eqref{eq:sup|e_t|} proves \eqref{eq:E[|X_t-bar X_t|]} for every $x\in\br^d$.
\end{proof}

\begin{theorem}\label{thm:sup E[|X_t-hat X_t|]}
    Suppose that Assumption \ref{cond:one-side Lip and initial value} holds. Then, for every $p\in(1,\alpha)$, there exist constants $C_\Xi>0$ and $\Delta^*\in(0,1)$ such that for every $s\in\br$, $x\in\br^d$, $T>0$ and $\Delta\in(0,\Delta^*]$,
    \begin{equation*}
        \sup_{0\leq t\leq T}\be[|X_t-\hat{X}_t|]\leq C_\Xi(1+|x|^p)\re^{C_\Xi T}\Delta^{\left(\left(\frac{1}{\alpha}\wedge\frac{2}{3}\right)-\frac{1}{3}\right)\wedge\frac{p-1}{3q}}.
    \end{equation*}
    Here $\hat{X}_0=\pi_\Delta(x)$, and the constants are independent of $s,x,T$ and $\Delta$.
\end{theorem}

\begin{proof}
    By Lemmas \ref{lem:E[|X_t-bar X_t|]} and \ref{lem:barX-hat X}, we derive that
    \begin{align*}
        \sup_{0\leq t\leq T}\be[|X_t-\hat{X}_t|]&\leq \sup_{0\leq t\leq T}\be[|X_t-\bar{X}_t|]+\sup_{0\leq t\leq T}\be[|\bar{X}_t-\hat{X}_t|]\\
        &\leq C_\Xi(1+|x|^p)\re^{C_\Xi T}\Delta^{((\frac{1}{\alpha}\wedge\frac{2}{3})-\frac{1}{3})\wedge\frac{p-1}{3q}}+C_\Xi(1+|x|)\Delta^{\frac{1}{\alpha}\wedge\frac{2}{3}}\\
        &\leq C_\Xi(1+|x|^p)\re^{C_\Xi T}\Delta^{((\frac{1}{\alpha}\wedge\frac{2}{3})-\frac{1}{3})\wedge\frac{p-1}{3q}}.
    \end{align*}
    This completes the proof.
\end{proof}

Taking $T=n\Delta$ in Theorem \ref{thm:sup E[|X_t-hat X_t|]} gives the following bound for the approximation error uniformly over the grid points in a finite time interval. It will be used in the next section.

\begin{corollary}\label{cor:X_i-hat{X}_i}
Suppose that Assumption \ref{cond:one-side Lip and initial value} holds. Then, for every $p\in(1,\alpha)$, there exist constants $C_\Xi,\bar{\gamma}>0$ and $\Delta^*\in(0,1)$ such that for every $s\in\br$, $x\in\br^d$, $n\in\bn_0$ and $\Delta\in(0,\Delta^*]$,
\begin{equation*}
    \sup_{0\leq i\leq n}\be[|X_{s+i\Delta}^{s,x}-\hat{X}_{s+i\Delta}^{s,\pi_\Delta(x)}|]\leq C_\Xi(1+|x|^p)\re^{\bar{\gamma} n\Delta}\Delta^{((\frac{1}{\alpha}\wedge\frac{2}{3})-\frac{1}{3})\wedge\frac{p-1}{3q}}.
\end{equation*}
The constants are independent of $s,x,n$ and $\Delta$.
\end{corollary}

\section{Error analysis}\label{sec:error analysis}
In this section, we use perturbation theory for Markov chains \cite{SS00} to establish convergence rates for the periodic measures of the truncated Euler--Maruyama scheme toward the corresponding periodic measures of SDE \eqref{SDEs}. For $1\leq p<\alpha$, let
\[V_p(x):=(1+|x|^2)^\frac{p}{2},~~x\in\br^d.\]
Throughout this section, the numerical transition kernel is understood through the projected initial value:
\[\hat P(s,x,s+n\Delta,\cdot)=\bp(\hat X_{s+n\Delta}^{s,\pi_\Delta(x)}\in\cdot).\]
In particular, $\hat P(s,x,s,\cdot)=\delta_{\pi_\Delta(x)}$.

\begin{proposition}\label{prop:exact}
Suppose that Assumptions \ref{A:b}, \ref{A:sigma} and \ref{A:periodic} hold. Then, for every $1\leq p<\alpha$, there exists a constant $C_1>0$ such that for all $s\in\br$, $x\in\br^d$ and $n\in\bn_0$,
\begin{equation*}
    \cw_1(P(s,x,s+n\Delta,\cdot),\mu_{s+n\Delta})\leq C_1\re^{-\gamma n\Delta}V_p(x),
\end{equation*}
where $\gamma$ is introduced in Corollary \ref{cor:W(P(s,x,t,cdot),mut)}. The constant $C_1$ may depend on $p$, but is independent of $s,x,n$ and $\Delta$.
\end{proposition}

\begin{proposition}\label{prop:P hat P}
Suppose that Assumptions \ref{cond:one-side Lip and initial value} and \ref{A:periodic} hold. Then, for every $1<p<\alpha$, there exist constants $C_2>0$ and $\Delta^*>0$ such that for all $s\in\br$, $x\in\br^d$, $\Delta\in(0,\Delta^*]$ and $n\in\bn_0$,
\begin{equation*}
   \cw_1(P(s,x,s+n\Delta,\cdot),\hat{P}(s,x,s+n\Delta,\cdot))\leq C_2 \re^{\bar{\gamma}n\Delta}V_p(x)\Delta^{((\frac{1}{\alpha}\wedge\frac{2}{3})-\frac{1}{3})\wedge\frac{p-1}{3q}},
\end{equation*}
where $\bar{\gamma}$ is introduced in Corollary \ref{cor:X_i-hat{X}_i}. The constant $C_2$ may depend on $p$, but is independent of $s,x,n$ and $\Delta$.
\end{proposition}

\begin{proof}
For $n\geq1$, we have
\begin{align*}
    \cw_1(P(s,x,s+n\Delta,\cdot),\hat P(s,x,s+n\Delta,\cdot))&\leq\be[|X_{s+n\Delta}^{s,x}-\hat X_{s+n\Delta}^{s,\pi_\Delta(x)}|]\\
    &\leq C_2\re^{\bar\gamma n\Delta}V_p(x)\Delta^{\left(\left(\frac1\alpha\wedge\frac23\right)-\frac13\right)\wedge\frac{p-1}{3q}}.
\end{align*}
For $n=0$, the distance equals $|x-\pi_\Delta(x)|$. The choice $\vp(R)=\ell(1+2R^q)$ and $M\geq3\ell$ gives
\[R_\Delta^q=\frac{M}{2\ell}\Delta^{-\frac{1}{3}}-\frac12\geq\Delta^{-\frac{1}{3}}.\]
Then we obtain
\begin{align*}
    |x-\pi_\Delta(x)|&=(|x|-R_\Delta)_+\leq |x|^pR_\Delta^{1-p}\leq |x|^p\Delta^{\frac{p-1}{3q}}\leq V_p(x)
    \Delta^{\left(\left(\frac1\alpha\wedge\frac23\right)-\frac13\right)\wedge\frac{p-1}{3q}}.
\end{align*}
Therefore, this proves the proposition for all $n\in\bn_0$.
\end{proof}

\begin{theorem}
Suppose that Assumptions \ref{cond:one-side Lip and initial value} and \ref{A:periodic} hold. Then, for every $1<p<\alpha$, there exist constants $\hat C>0$ and $\Delta^*\in(0,1)$ such that, for every $N\in\bn$ satisfying $\Delta:=\frac{\tau}{N}\in(0,\Delta^*]$, there exists a real number $n=n(\Delta)\geq2$ such that, for all $s\in\br$, $x\in\br^d$, and $i\in\bn_0$ with $i\geq n$,
\begin{equation}\label{eq:cw_1(cl(hat X, mu) leq Delta}
    \cw_1(\hat{P}(s,x,s+i\Delta,\cdot),\mu_{s+i\Delta})\leq \hat C V_p(x)\Delta^{\bigl[((\frac{1}{\alpha}\wedge\frac{2}{3})-\frac{1}{3})\wedge\frac{p-1}{3q}\bigr]\zeta}
\end{equation}
and
    \begin{equation}\label{eq:cw(mu,hat mu}
        \cw_1(\hat{\mu}_s,\mu_s)\leq  \hat{C}\Delta^{\bigl[((\frac{1}{\alpha}\wedge\frac{2}{3})-\frac{1}{3})\wedge\frac{p-1}{3q}\bigr]\zeta},
    \end{equation}
where $\zeta=\frac{\gamma}{\gamma+\bar{\gamma}}\in(0,1)$.
\end{theorem}

\begin{proof}
    Based on the restriction of $\Delta$ in Lemma \ref{lem:Lyapunov condition for TEM}, let \[\Delta^*:=\min\left\{\exp\left(-\frac{2(\gamma+\bar\gamma)}{\left(\left(\frac1\alpha\wedge\frac23\right)-\frac13\right)\wedge\frac{p-1}{3q}}\right),\frac{1}{2},\frac{1}{p\lambda},\frac{\lambda^3}{(4M^2)^3}\right\}.\]
    This choice is independent of $\Delta$ and satisfies $\Delta^*\in(0,1)$. Moreover,  for every $0<\Delta\leq\Delta^*$,
    \[n:=\frac{\bigl[((\frac{1}{\alpha}\wedge\frac{2}{3})-\frac{1}{3})\wedge\frac{p-1}{3q}\bigr]\log \Delta}{-\Delta(\gamma+\bar{\gamma})}\geq \frac{2}{\Delta}\geq 2.\]
    Moreover, we have
    \begin{equation}\label{eq:perturbation_balance}
        \Delta^{((\frac{1}{\alpha}\wedge\frac{2}{3})-\frac{1}{3})\wedge\frac{p-1}{3q}}\re^{\bar{\gamma}n\Delta}=\re^{-\gamma n\Delta}=\Delta^{[((\frac{1}{\alpha}\wedge\frac{2}{3})-\frac{1}{3})\wedge\frac{p-1}{3q}]\zeta}.
    \end{equation}
    By the triangle inequality and Propositions \ref{prop:exact} and \ref{prop:P hat P}, for every $s\in\br$, $x\in\br^d$ and $i\in\bn_0$,
\begin{equation}\label{eq:perturbation_block_estimate}
\begin{aligned}
    \cw_1(\hat P(s,x,s+i\Delta,\cdot),\mu_{s+i\Delta})&\leq\bigl[C_1\re^{-\gamma i\Delta}+C_2\re^{\bar\gamma i\Delta}\Delta^{((\frac1\alpha\wedge\frac{2}{3})-\frac13)\wedge\frac{p-1}{3q}}\bigr]V_p(x).
\end{aligned}
\end{equation}
Iterating Lemma \ref{lem:Lyapunov condition for TEM}, with $\hat X_0=\pi_\Delta(x)$, yields, for every $i\in\bn_0$,
\begin{equation}\label{eq:perturbation numerical moment}
\begin{aligned}
    \be[V_p(\hat{X}_{s+i\Delta}^{s,\pi_\Delta(x)})]&=\int_{\bar{B}(0,R_\Delta)}V_p(y)\hat P(s,x,s+i\Delta,\d y)\\
    &\leq(1-2^{-1}p\lambda\Delta)^iV_p(\pi_\Delta(x))+\frac{2C_0}{p\lambda}(1-(1-2^{-1}p\lambda\Delta)^i)\\
    &\leq(1-2^{-1}p\lambda\Delta)^iV_p(x)+\frac{2C_0}{p\lambda}(1-(1-2^{-1}p\lambda\Delta)^i).
\end{aligned}
\end{equation}
Let
    \[\text{Lip}_1:=\{f:\br^d\to\br: f(0)=0 \text{ and } |f(x)-f(y)|\leq |x-y|\}.\]
Fix $i\in\bn$ and $1\leq j\leq i$. For $f\in\operatorname{Lip}_1$, the Chapman--Kolmogorov equation gives
\begin{align*}
    &~\hat P(s,s+i\Delta)f(x)-\mu_{s+i\Delta}(f)\\
    =&~\int_{\bar{B}(0,R_\Delta)}[\hat P(s+(i-j)\Delta,s+i\Delta)f(y)-\mu_{s+i\Delta}(f)]\hat P(s,x,s+(i-j)\Delta,\d y).
\end{align*}
The Kantorovich--Rubinstein duality formula and \eqref{eq:perturbation_block_estimate} yield
\begin{align*}
    &~\cw_1\bigl(\hat P(s,x,s+i\Delta,\cdot),\mu_{s+i\Delta}\bigr)\\
    \leq&~\left[C_1\re^{-\gamma j\Delta}+C_2\re^{\bar\gamma j\Delta}\Delta^{\left(\left(\frac1\alpha\wedge\frac23\right)-\frac13\right)\wedge\frac{p-1}{3q}}\right]\int_{\bar{B}(0,R_\Delta)}V_p(y)\hat P(s,x,s+(i-j)\Delta,\d y).
\end{align*}
For every $i,j\in\bn$ with $j\leq i$, observe that
\begin{align}\label{eq:inequality}
    (1-2^{-1}p\lambda\Delta)^{i-j}V_p(x)+\frac{2C_0}{p\lambda}
    \left[1-(1-2^{-1}p\lambda\Delta)^{i-j}\right]\leq\frac{2C_0}{p\lambda}V_p(x).
\end{align}
It follows from \eqref{eq:perturbation numerical moment} and \eqref{eq:inequality} that
\begin{equation}\label{eq:W(L(hat Xn),mun) convex}
\begin{aligned}
    \cw_1\bigl(\hat P(s,x,s+i\Delta,\cdot),\mu_{s+i\Delta}\bigr)&\leq\frac{2C_0}{p\lambda}\left[C_1\re^{-\gamma j\Delta}+C_2\re^{\bar\gamma j\Delta}\Delta^{\left(\left(\frac1\alpha\wedge\frac23\right)-\frac13\right)\wedge\frac{p-1}{3q}}\right]V_p(x).
\end{aligned}
\end{equation}
The case $i=j$ is included, since $\hat P(s,x,s,\cdot)=\delta_{\pi_\Delta(x)}$. For every  $i\in\bn_0$ with $i\geq n$, choose $j=\lceil n\rceil$. Then $1\leq j\leq i$ and $n\leq j<n+1$. Therefore, \eqref{eq:perturbation_balance} gives
\begin{align*}
    C_1\re^{-\gamma j\Delta}+C_2\re^{\bar\gamma j\Delta}\Delta^{\left(\left(\frac1\alpha\wedge\frac23\right)-\frac13\right)\wedge\frac{p-1}{3q}}&\leq C_1\re^{-\gamma n\Delta}+C_2\re^{\bar{\gamma}\Delta}\re^{\bar\gamma n\Delta}\Delta^{\left(\left(\frac1\alpha\wedge\frac23\right)-\frac13\right)\wedge\frac{p-1}{3q}}\\
    &\leq(C_1+C_2\re^{\bar{\gamma}})\Delta^{\left[\left(\left(\frac1\alpha\wedge\frac23\right)-\frac13\right)\wedge\frac{p-1}{3q}\right]\zeta}.
\end{align*}
Since $V_p\geq1$, the one-step Lyapunov estimate at initial value $0$ gives
\[1\leq\hat P(s,s+\Delta)V_p(0)\leq1-\frac{p\lambda}{2}\Delta+C_0\Delta.\]
Hence $\frac{2C_0}{p\lambda}\geq1$. Substituting the bound into \eqref{eq:W(L(hat Xn),mun) convex} proves \eqref{eq:cw_1(cl(hat X, mu) leq Delta}, with $\hat C:=\frac{2C_0}{p\lambda}(C_1+C_2\re^{\bar{\gamma}})$.

For fixed $s$ and $\Delta$, the numerical periodic measure is supported on $\bar{B}(0,R_\Delta)$, so $\hat\mu_s(V_p)<\infty$. Using $\hat P^*(s,s+N\Delta)\hat\mu_s=\hat\mu_s$ and integrating \eqref{eq:perturbation numerical moment} with $i=N$ against $\hat\mu_s$ gives
\[\hat\mu_s(V_p)\leq(1-2^{-1}p\lambda\Delta)^N\hat\mu_s(V_p)+\frac{2C_0}{p\lambda}[1-(1-2^{-1}p\lambda\Delta)^N].\]
Since $(1-2^{-1}p\lambda\Delta)^N<1$, we have
\begin{equation}\label{eq:perturbation_periodic_moment}
    \hat\mu_s(V_p)=\int_{\bar{B}(0,R_\Delta)}V_p(x)\hat\mu_s(\d x)\leq\frac{2C_0}{p\lambda}.
\end{equation}
This estimate is uniform in $s,N$, and $\Delta$. Keep $j=\lceil n\rceil$ and choose an integer $k\geq1$ such that $kN\geq j$. Periodicity and the evolution relation give
\[\hat P^*(s,s+k\tau)\hat\mu_s=\hat\mu_s,~~\mu_{s+k\tau}=\mu_s.\]
The Kantorovich--Rubinstein duality formula yields
\begin{align*}
    \cw_1(\hat\mu_s,\mu_s)
    &\leq\int_{\bar{B}(0,R_\Delta)}\cw_1\bigl(\hat P(s,x,s+k\tau,\cdot),\mu_{s+k\tau}\bigr)\hat\mu_s(\d x).
\end{align*}
Integrating \eqref{eq:W(L(hat Xn),mun) convex} with $i=kN$ against $\hat\mu_s$ and using \eqref{eq:perturbation_periodic_moment}, we obtain
\begin{align*}
    \cw_1(\hat\mu_s,\mu_s)
    &\leq\frac{2C_0}{p\lambda}\left[C_1\re^{-\gamma j\Delta}+C_2\re^{\bar\gamma j\Delta}\Delta^{\left(\left(\frac1\alpha\wedge\frac23\right)-\frac13\right)\wedge\frac{p-1}{3q}}\right]\hat{\mu}_s(V_p)\\
    &\leq\frac{4C_0^2}{(p\lambda)^2}(C_1+C_2\re^{\bar{\gamma}})\Delta^{\left[\left(\left(\frac1\alpha\wedge\frac23\right)-\frac13\right)\wedge\frac{p-1}{3q}\right]\zeta}.
\end{align*}
Taking $\hat{C}:=4C_0^2(p\lambda)^{-2}(C_1+C_2\re^{\bar{\gamma}})$ proves \eqref{eq:cw(mu,hat mu}.
\end{proof}

\section*{Acknowledgements}
We acknowledge the financial supports from an EPSRC grant (ref EP/S005293/2).
\addcontentsline{toc}{section}{Acknowledgements}

\addtolength{\itemsep}{-1.5 em}
\setlength{\itemsep}{-3pt}
\footnotesize

\addcontentsline{toc}{section}{References}

\bibliographystyle{plain}  
\bibliography{reference}   

\end{document}